\documentclass[11pt,reqno]{amsart}

\usepackage{amsmath,amssymb,amsthm,amsfonts,latexsym, amsthm, amscd, mathrsfs, stmaryrd}
\usepackage[colorlinks=red,linkcolor=blue,citecolor=blue]{hyperref}
\usepackage{graphicx,color,colortbl}
\usepackage{upgreek}
\usepackage[mathscr]{euscript}
\usepackage{hhline}
\usepackage{float}
\usepackage{todonotes}
\numberwithin{equation}{section}
\usepackage{dynkin-diagrams}
\usepackage{epic}
\allowdisplaybreaks
\numberwithin{equation}{section}

\usepackage{tikz}
\usetikzlibrary{
    cd,
	shapes,
	arrows,
	positioning,
	decorations.markings,
	decorations.pathmorphing,
	circuits.logic.US,
	circuits.logic.IEC,
	fit,
	calc,
	plotmarks,
	matrix
}

\tikzset{
>=stealth',
help lines/.style={dashed, thick},
axis/.style={<->},
important line/.style={thick},
connection/.style={thick, dotted},
punkt/.style={
rectan\mathrm{GL}e,
rounded corners,
draw=black, thick,
text width=4.5em,
minimum height=2em,
text centered,
},
pil/.style={
->,
thick,
gray,
shorten <=2pt,
shorten >=2pt,}
}

\usepackage{array}
\usepackage{adjustbox}
\usepackage{cleveref}
\usepackage{enumerate}

\newtheorem{proposition}{Proposition}[section]
\newtheorem{lemma}[proposition]{Lemma}
\newtheorem{problem}[proposition]{Problem}
\newtheorem{corollary}[proposition]{Corollary}
\newtheorem{theorem}[proposition]{Theorem}
\theoremstyle{definition}
\newtheorem{definition}[proposition]{Definition}
\newtheorem{remark}[proposition]{Remark}
\newtheorem{example}[proposition]{Example}

\newtheorem{alphatheorem}{Theorem}

\makeatletter

\newcommand{\arxiv}[1]{\href{http://arxiv.org/abs/#1}{\tt arXiv:\nolinkurl{#1}}}

\newcommand{\Rmnum}[1]{\expandafter\@slowromancap\romannumeral #1@}

\def \A{\mathcal A}

\def \g{\mathfrak{g}}

\def \sl{\mathfrak{sl}}

\def \bB{{\mathbf B }}

\def \N{\mathbb{N}}
\def \Q{\mathbb{Q}}
\def \Z{\mathbb{Z}}
\def \I{\mathbb{I}}

\def \bZ{\mathbb{Z}}

\def \la{\lambda}

\def \ad{\text{ad}\, }
\def \ov{\overline}
\def \un{\underline}

\def \Xi{X_\imath}

\newcommand{\nc}{\newcommand}
\nc{\browntext}[1]{\textcolor{brown}{#1}}
\nc{\greentext}[1]{\textcolor{green}{#1}}
\nc{\redtext}[1]{\textcolor{red}{#1}}
\nc{\bluetext}[1]{\textcolor{blue}{#1}}
\nc{\brown}[1]{\browntext{ #1}}
\nc{\green}[1]{\greentext{ #1}}
\nc{\red}[1]{\redtext{ #1}}
\nc{\blue}[1]{\bluetext{ #1}}

\def \Q {\mathbb Q}

\def \TT{\mathbf T}
\def \TTd{\dot{\mathbf T}}
\def \bt{\mathbf t}
\newcommand{\wt}{\text{wt}}

\def \vs{\varsigma}

\def \F{\mathbb{F}}

\def \a{\mathfrak{a}}
\def \wItau{\I_{\tau}}

\def \g{\mathfrak{g}}

\def \sl{\mathfrak{sl}}

\def \N{\mathbb{N}}
\def \Q{\mathbb{Q}}

\def \Z{\mathbb{Z}}
\def \I{\mathbb{I}}
\def \Br{\mathrm{Br}}

\def \bZ{\mathbb{Z}}

\def \A{\mathcal{A}}

\def \one{{\mathbf 1}}
\def \bs{\boldsymbol{\varsigma}}

\def \fX{\Upsilon}
\def \wI{\I}
\def \oc{\ov{c}}
\def \htau{\widehat{\tau}}
\def \cB{\mathcal{B}}

\def \bs{\mathbf{r}} 
\def \bvs{{\boldsymbol{\varsigma}}} 
\def \ba{\mathbf{a}}

\def \tTT{\widetilde{\mathbf{T}}}

\def \Id{\mathrm{Id}}
\def \dm{\diamond}

\newcommand{\U}{\mathbf{U}}
\newcommand{\Ui}{\U^\imath}
\newcommand{\dUi}{\dot\U^\imath}
\newcommand{\dUiA}{\dot\U_{\!\mathcal{A}}^\imath}

\newcommand{\tUi}{\widetilde{{\mathbf U}}^\imath}

\newcommand{\qbinom}[2]{\begin{bmatrix} #1\\#2 \end{bmatrix} }

\def \Ifin{\wI^{\mathrm{fin}}}

\def \fwItau{\wItau^{\mathrm{fin}}}

\newcommand{\dv}[2]{{B}_{#1}^{{(#2)}}}
\newcommand{\dvi}[2]{{B}_{i,#1}^{{(#2)}}}
\newcommand{\ev}{\bar{0}}
\newcommand{\odd}{\bar{1}}
\newcommand{\cbinom}[2]{\left\{\begin{matrix} #1\\#2 \end{matrix} \right\}}

\usepackage{cleveref}

\crefname{definition}{Definition}{Definitions}
\crefname{example}{Example}{Examples}
\crefname{lemma}{Lemma}{Lemmas}
\crefname{conjecture}{Conjecture}{Conjectures}
\crefname{corollary}{Corollary}{Corollaries}
\crefname{proposition}{Proposition}{Propositions}
\crefname{theorem}{Theorem}{Theorems}
\crefname{alphatheorem}{Theorem}{Theorems}
\crefname{remark}{Remark}{Remarks}
\crefname{equation}{}{}
\crefname{enumi}{}{}
\crefname{section}{Section}{Section}

\begin{document}

\title[Relative braid group symmetries on modules]{Relative braid group symmetries on modified $\mathrm{i}$quantum groups and their modules}

\author[Weiqiang Wang]{Weiqiang Wang}
	\address{Department of Mathematics, University of Virginia, Charlottesville, VA 22904, USA}
	\email{ww9c@virginia.edu}

	\author[Weinan Zhang]{Weinan Zhang}
	\address{School of Mathematical Sciences, Shenzhen University, Shenzhen, Guangdong, 
518086, P.R.China}
	\email{mathzwn@szu.edu.cn}
	

\begin{abstract}
We present a comprehensive generalization of Lusztig's braid group symmetries for quasi-split iquantum groups. Specifically, we give 3 explicit rank one formulas for symmetries acting on integrable modules over a quasi-split iquantum group of arbitrary Kac-Moody type with general parameters. These symmetries are formulated in terms of idivided powers and iweights of the vectors being acted upon. We show that these symmetries are consistent with the relative braid group symmetries on iquantum groups, and they are also related to Lusztig's symmetries via quasi $K$-matrices. Furthermore, through appropriate rescaling, we obtain compatible symmetries for the integral forms of modified iquantum groups and their integrable modules. We verify that these symmetries satisfy the relative braid group relations.
\end{abstract}

\maketitle

\begin{quote}
\begin{center}
{\em In memory of Gail Letzter 
}
\end{center}
\end{quote}

\setcounter{tocdepth}{2}
\tableofcontents

\vspace{1em}
\section{Introduction}

\subsection{Background}
The foundation of quantum symmetric pairs $(\U, \Ui)$ was developed by Gail Letzter \cite{Let99, Let02} (and then in \cite{Ko14}), who pursued harmonic analysis from this perspective. The iquantum group $\Ui=\Ui_\bvs$, depending on a parameter $\bvs$, is a coideal subalgebra of a Drinfeld-Jimbo quantum group $\U$. In the past decade, it has become increasingly well known and fruitful to view Drinfeld-Jimbo quantum groups as iquantum groups of diagonal type, and the $\imath$-program \cite{BW18a} aims at generalizing all fundamental (algebraic, geometric, and categorical) constructions of quantum groups to the setting of iquantum groups; see \cite{Wa23} for a survey on the $\imath$-program including iSchur duality, super Kazhdan-Lusztig theory, (quasi) $K$-matrix, icanonical bases, iHall algebra, iDrinfeld presentation, and others.  

Braid group symmetries for quantum groups and their modules developed by Lusztig \cite{Lus90a, Lus90c, Lus93} have played a fundamental role in understanding fine structures of quantum groups such as PBW bases and canonical bases; they also play a central role in geometric and categorical representation theory. 

In the theory of symmetric pairs, the relative root systems and relative Weyl groups often play the role of root systems and Weyl groups for complex semisimple Lie algebras; see \cite{OV90}. Kolb and Pellegrini \cite{KP11} made an audacious conjecture that there exist relative braid group symmetries on iquantum groups of finite type. Relying on computer calculations, they constructed and verified the relative braid group symmetries for quasi-split iquantum groups of finite type (and also type AII) in special parameters. See \cite{Ch07, MR08} for earlier examples of relative braid group symmetries on iquantum groups. 

The $\imath$-program has provided some conceptual frameworks and powerful tools studying the relative braid group action on iquantum groups. Relative braid group symmetries have been naturally realized as BGP reflection functors within the framework of iHall algebras \cite{LW21, LW22}. However, the Hall algebra approach does not lead to braid group actions on the modules (for either quantum groups or iquantum groups). 

By making an ansatz with constructions for quantum groups \cite{KR90, LS90, Lus90c} and developing new characterizations and intertwining properties of quasi $K$-matrices, in the previous work \cite{WZ23} the authors constructed relative braid group symmetries for iquantum groups of finite type through Lusztig's braid group symmetries on quantum groups, completely settling the conjecture in \cite{KP11}. This approach has been generalized in \cite{Z23} to quasi-split iquantum groups of Kac-Moody type \cite{Ko14}. It was not known back then whether or not these relative braid group symmetries preserve the integral form modified iquantum groups. 

Formulas of relative braid group operators acting on integrable $\U$-modules with weights bounded above (viewed as $\Ui$-modules by restriction) are given in \cite{WZ23} in terms of Lusztig's braid group operators and rank one quasi $K$-matrices; the weight bounded above constraint is imposed since the quasi $K$-matrix lies in the positive half of $\U$. These formulas are not always explicit enough for direct computations and they are not broad enough to make sense for $\Ui$-modules which are not restrictions of $\U$-modules. 

The rank one formula for the braid group action on integrable modules over a quantum group has played a most fundamental role and formed the starting point of Lusztig's approach (see \cite[5.2.1]{Lus93} or \eqref{eq:mod1}--\eqref{eq:mod2}). This has led to the following basic open problem for iquantum groups.

\begin{problem} \label{prob:1}
    Find rank one formulas for relative braid group symmetries on all integrable $\Ui$-modules and modified iquantum groups, preserving the integral $\Z[q,q^{-1}]$-forms.
\end{problem}
The complexity of this problem for iquantum groups is also due to the existence of 3 quasi-split finite types of rank one, among 8 finite types of (real) rank one; see \cite{BW18b}. 

\subsection{What is achieved in this paper?}

The goal of this paper is to provide a comprehensive answer to Problem \ref{prob:1} in the generality of quasi-split iquantum groups $\Ui$ of arbitrary Kac-Moody type for general parameters $\bvs =(\vs_i)_{i\in \I}$. 

The relative Weyl group $W^\circ$ associated to a quasi-split Satake diagram $(\I,\tau)$ can be identified with the Coxeter group generated by the simple reflections $\bs_i =\bs_{\tau i}$, for $i\in \Ifin =\{i\in \I \mid c_{i,\tau i}=2,0,-1\}$; see \eqref{def:Ifinite}--\eqref{risi}. For each of the 3 quasi-split finite types of rank one, we present an explicit braid group formula acting on integrable $\Ui$-modules in terms of idivided powers. 
We show that these actions are compatible with the relative braid group action on $\Ui$ in \cite{WZ23, Z23}, and that our formulas coincide with the (non-explicit) formulas {\em loc. cit.} on integrable $\U$-modules with weights bounded above. 

By a suitable modification of the above formulas with specific parameters, we produce formulas for relative braid group actions that preserve the integral forms of integrable $\Ui$-modules. Moreover, by the requiring compatibility, we obtain relative braid group actions with explicit formulas on the modified iquantum group and its integral form. 

It is well known that braid group actions in the quantum group setting \cite{Lus93} provide a wealth of formulas involving divided powers and $q$-root vectors and long computations are sometimes unavoidable; see also \cite{Ja95}. The proofs of various relative braid group action formulas and compatibilities in this paper are reduced to proving several highly nontrivial identities involving idivided powers and $q$-root vectors in $\Ui$. We are also led to develop new formulas relating canonical bases, icanonical bases, and quasi $K$-matrices. 

Let us describe the main results in some details. 

 \subsection{New rank one formulas}\label{sec:formula}

The quasi-split iquantum group $\Ui$ is a subalgebra of $\U$ generated by $B_i$ and $k_i$ for $i\in \I$; see \eqref{def:iQG}. 
The notion of idivided powers (introduced in \cite{BW18a, BeW18} and denoted by $B_{i,\ov{p}}^{(k)}$, for $\tau i=i\in \I$, $k\ge 1$ and $\ov{p} \in \Z_2$) on (modified) iquantum groups has become indispensable in integral forms and icanonical bases \cite{BW18b}, iSerre relations \cite{CLW21a, Car23}, and Serre-Lusztig relations as well as earlier braid group formulas \cite{CLW21b}. 
The idivided powers in split rank one case depend on an iweight lying in the iweight lattice $X_\imath$ (see \cite{BW18b}). 
For $\tau i\neq i\in \I$, the idivided powers $B_i^{(k)}$ are defined as Lusztig's. 

Let $M=\oplus_{\ov{\la} \in X_\imath} M_{\ov{\la}}$ be an integrable $\Ui$-module; cf. Definition~\ref{def:integrable}. The following are the key new formulas for linear operators $\TT'_{i,-1}$ on $M$, for $i\in \Ifin$, introduced in this paper. We shall refer to the 3 cases below as split, diagonal, and quasi-split type rank one cases, respectively; See  
\eqref{eq:ibraid1'}, \eqref{eq:ibraid2'} and \eqref{eq:ibraid3'} for $\TT''_{i,+1} (v)$. Set $\ov{\la}_{i} =\langle h_i,\ov{\lambda}\rangle$, $\la_{i,\tau} =\langle h_i-h_{\tau i},\ov{\lambda}\rangle$, and $v\in M_{\ov{\lambda}}$ below.

\vspace{2mm}
$\boxed{\text{(i). Split type}~c_{i,\tau i}=2}$:
\begin{align*}
\TT'_{i,-1}(v) & =\sum_{k\geq 0:\:\ov{k}=\ov{\la}_i}  (-1)^{-k/2}(q_i^2\vs_i)^{-k/2} \dvi{\ov{k}}{k}v;
\end{align*}

\vspace{2mm}
$\boxed{\text{(ii). Diagonal type}~c_{i,\tau i}=0}$:
\begin{align*}
\TT'_{i,-1}(v)&= (-1)^{\la_{i,\tau}/2} q_i^{-\la_{i,\tau}/2} \sum_{a-b=\la_{i,\tau}} (-1)^b q_i^{-b} \vs_i^{-(a+b)/2} B_{i}^{(a)}  B_{\tau i}^{(b)}v;
\end{align*}

\vspace{2mm}
$\boxed{\text{(iii). Quasi-split type}~c_{i,\tau i}=-1}$:
\begin{align*} 
    \TT'_{i,-1}(v)&= q_i^{\la_{i,\tau}^2/2} k_i^{-\la_{i,\tau}}\sum_{t\ge 0,l\ge 0} (-1)^{t+l}  q_i^{ -\frac{(t-l)^2}{2}-\frac{t+l}{2}} k_i^{t-l} \vs_i^{\frac{\la_{i,\tau}}{2}-t } \vs_{\tau i}^{-\frac{\la_{i,\tau}}{2}-l} B^{(t)}_{\tau i} B^{(t+l)}_{i}B^{(l)}_{\tau i}v.
\end{align*}
In our setting, the formula in the diagonal type recovers Lusztig's formula in \cite[2.2(b)]{Lus09} (up to a rescaling); the original formulas in \cite[5.2]{Lus93} in different forms are a $q$-deformation of some classic formulas. In contrast, for the split and quasi-split types, even the classical limits of our formulas in (i) and (iii) are new. 

\subsection{Compatible braid group actions}

Denote by $\fX$ the quasi $K$-matrix associated to the quantum symmetric pair $(\U,\Ui_\bvs)$, by $\fX_{i}$ the rank one quasi $K$-matrix associated to the rank one Satake subdiagram $(\{i,\tau i\}, \tau)$. The quasi $K$-matrix in distinguished parameters satisfies intertwining properties with bar involutions \cite{BW18a, BK19, BW18b, BW21}; the quasi $K$-matrix for general parameters was suggested by Appel and Vlaar in \cite{AV20} and then formulated in \cite{WZ23} in terms of intertwining properties with anti-involutions (similar to \cite{Lus93}).

The relative braid group symmetries $\TT'_{i,-1}$, for $i\in \Ifin$, are constructed in \cite{WZ23, Z23} on $\Ui$. They are characterized by $\TT'_{i,\bvs,-1}(x) \fX_i =\fX_i T'_{\bs_i,\bvs, -1} (x), \forall x\in \Ui$; see \eqref{eq:compTT}, where $T'_{\bs_i,\bvs, -1}$ is a normalized Lusztig braid group operator \eqref{braid_Urescale} on $\U$ associated with $\bs_i\in W$ in \eqref{risi}.

We show that the new formulas on integrable $\Ui$-module $M$ given in \S \ref{sec:formula} are compatible with these symmetries on $\Ui$. 

\begin{alphatheorem}
[\cref{thm:split,thm:diag,thm:qs-split,prop:split2,prop:diag2,prop:quasisplit2}]
\label{thm:main1Intro}
Let $M$ be an integrable $\Ui$-module, and let $i\in \Ifin$. The following identities hold, for any $x\in \Ui,v\in M$: 
\begin{equation*}
\TT'_{i,-1}(x v)=\TT'_{i,-1}(x) \TT'_{i,-1}(v),\qquad
\TT''_{i,+1}(x v)=\TT''_{i,+1}(x) \TT''_{i,+1}(v).
\end{equation*}
\end{alphatheorem}

In \cite[\S10.2]{WZ23} and \cite{Z23}, the relative braid group symmetries on integrable $\U$-modules with weights bounded above were defined via quasi $K$-matrix as the right-hand sides of the identities \eqref{Ti:old-definition} below. Our next main result, Theorem~\ref{thm:main2Intr}, relates the symmetries on $\Ui$-modules via explicit rank one formulas in this paper to Lusztig's through quasi $K$-matrix. In particular, the current formulas specialize to those defined {\em loc. cit.}.

\begin{alphatheorem} [\cref{thm:res-split,thm:res-diag,thm:res-qssplit}]
\label{thm:main2Intr}
Let $M$ be an integrable $\U$-module whose weights are bounded above, and $i\in \Ifin$. Assume that the parameter $\bvs$ is balanced. Then the following identities hold, for all $v\in M$:
\begin{align}  \label{Ti:old-definition}
\TT'_{i,-1}(v)=\fX_{i} T'_{\bs_i,\bvs,-1}(v),\qquad \TT''_{i,+1}(v)=T''_{\bs_i,\bvs,+1}(\fX_{i}^{-1}) T''_{\bs_i,\bvs,+1}(v).
\end{align}
\end{alphatheorem}

More importantly, the current explicit definitions of $\TT'_{i,-1}$ and $\TT''_{i,+1}$ work for arbitrary parameters and for arbitrary integrable $\Ui$-modules (not necessarily integrable $\U$-modules with weights bounded above), and it remains compatible with the relative braid symmetries on $\Ui$ by \cref{thm:main1Intro}.

\subsection{Formulas for braid group actions on idivided powers}

Our current approach has led to new closed formulas for relative braid symmetries acting on idivided powers in $\Ui$. We list the formulas for $\TT'_{i,-1}$ below. 

\begin{alphatheorem} [\cref{thm:split-T1DP,thm:diag-T1DP,thm:qsplit-T1DP}]
\label{thm:main3-iDP}
Let $i\in\Ifin$ and $j\in \I$ such that $j\neq i,\tau i$. Set $\alpha=-c_{ij}, \beta=-c_{\tau i,j}$. Let $n\ge 0,\ov{t}\in \Z_2$. Assume that $j=\tau j$. We have
\vspace{2mm}
\begin{enumerate}
    \item[(i)] 
\boxed{\text{Split type}~ c_{i,\tau i}=2}: 
\begin{align*}
\begin{split}
\TT'_{i,-1}(\dv{j,\ov{t}}{n})&=\sum_{ r+s =n\alpha} (-1)^{r} q_i^{r} \dv{i,\ov{p+n\alpha}}{s} \dv{j,\ov{t}}{n} \dv{i,\ov{p}}{r}
\\
&\quad +\sum_{u\geq 1}\sum_{\substack{r+s+2u=n\alpha\\ \ov{r}=\ov{p}}} (-1)^{r} q_i^{r} \dv{i,\ov{p+n\alpha}}{s} \dv{j,\ov{t}}{n} \dv{i,\ov{p}}{r};
\end{split}
\end{align*}
\item[(ii)]
\boxed{\text{Diagonal type}~ c_{i,\tau i}=0}:
\begin{align*} 
\TT'_{i,-1}(\dv{j,\ov{t}}{n})&= \sum_{u=0}^{ \min(n\alpha,n\beta)} \sum_{r=0}^{n\alpha-u}\sum_{s=0}^{n\beta-u} (-1)^{r+s }q_i^{r(-u+1)+s(u+1) } 
\\ 
  &\qquad \times B_i^{(n\alpha-r-u)}B_{\tau i}^{(n\beta-s-u)}\dv{j,\ov{t}}{n} B_{\tau i}^{(s)}B_i^{(r)}  k_i^u;
\end{align*}
\item[(iii)]
\boxed{\text{Quasi-split type}~ c_{i,\tau i}=-1}:
\begin{align*} 
\TT'_{i,-1}(\dv{j,\ov{t}}{n})&=\sum_{u,w\geq 0} \sum_{t=0}^{n\beta-w} \sum_{s=0}^{n\beta+n\alpha-w-u} \sum_{r=0}^{n\alpha-u} (-1)^{t+r+s}
 q_i^{t(-2w+1) + r(u+1)+s(w-2u+1)+uw+3ut}
 \\
&\quad\times q_i^{-\frac{u^2+w^2}{2}} B_i^{(n\beta-w-t)}B_{\tau i}^{(n\alpha+n\beta-w-u-s)}B_i^{(n\alpha-u-r)} \dv{j,\ov{t}}{n} B_i^{(r)} B_{\tau i}^{(s)} B_i^{(t)} k_i^{w-u}.
\end{align*}
\end{enumerate}

If $\tau j\neq j$, then $\TT'_{i,-1}(\dv{j}{n})$ is given by Formulas (i)--(iii) with $\dv{j,\ov{t}}{n}$ replaced by $\dv{j}{n}$. 
\end{alphatheorem}

Let us comment on the proofs of \cref{thm:main1Intro,thm:main2Intr,thm:main3-iDP}.
By applying some general algebra isomorphisms, one can reduce the proofs of \cref{thm:main1Intro,thm:main2Intr,thm:main3-iDP} for general parameters $\bvs$ to any special parameter adapted to each statement and each locally finite type. 

It was shown in \cite[Chap. 37]{Lus93} that a braid group operator $T_i'$ on $\U$ sends the Chevalley generators (and their divided powers) to suitable $q$-root vectors (and higher order $q$-root vectors). For quasi-split iquantum groups, $q$-root vectors and their higher order analogues have been introduced in \cite{CLW21b,CLW23,Z23} for each of the 3 rank one finite type. The automorphism $\TT'_{i,-1}$ on $\Ui$ sends the generators $B_j$ to a suitable $q$-root vector, cf. \cite{CLW21b, Z23}. 

On the other hand, the rank one braid formulas for $\TT'_{i,-1}$ on integrable $\Ui$-modules in this current work are given in terms of idivided powers. Therefore, proving \cref{thm:main1Intro} boils down to verifying various highly nontrivial identities involving idivided powers and $q$-root vectors (or their higher order analogues). See \cref{prop:BBij2,prop:BBij0} for the split type (i), \cref{prop:qs-rkone,prop:qsBB} for the diagonal type (ii), and \cref{thm:qs-rk1,thm:qs-rktwo} for the quasi-split type (iii), respectively. \cref{thm:main3-iDP}   follows from the higher order identities. 

For the proof of \cref{thm:main2Intr}, we choose to work with a specific parameter $\bvs_\star$ so the $\imath$bar involution on $\Ui$ and $\U$-modules can be defined via quasi $K$-matrix \cite{BW18a,BW21}. In the split rank one case, a formula expressing Lusztig canonical basis via icanonical basis on simple $\U(\sl_2)$-modules is obtained in \cref{prop:inverse}. The quasi-split rank one case is much more challenging, and to that end we construct a new monomial basis on any  finite-dimensional simple $\U(\sl_3)$-module viewed as a $\Ui$-module, which is of independent interest and will have additional futuree applications; see \cref{thm:qsibasis}. Then the transition matrices between Lusztig canonical basis and the monomial basis on a simple $\U(\sl_3)$-module are obtained in \cref{prop:BBBFFF,prop:FB}. These formulas and  \cref{thm:main1Intro} (as well as some identities established toward the proof of \cref{thm:main1Intro}) are used to prove \cref{thm:main2Intr}. The compatibility between relative group action and Lusztig's braid group action in \cref{thm:main2Intr} was helpful in the initial formulation of the rank one formulas.

\subsection{Relative braid group actions and integral forms}

A notable feature of iquantum groups is the {\em nonexistence} of an integral $\A$-form on $\Ui$ in general (where $\A=\Z[q,q^{-1}]$) compatible with $\A$-forms of {\em all} highest weight integrable $\U$-modules; this phenomenon occurs already in split rank one and motivated the notion of idivided powers depending on parities (which led to the notion of iweights). The idivided powers with $\vs=q^{-1}$ associated with parity $\ov{0}$ (or respectively, $\ov{1}$) preserve only the integral forms of the simple $\U(\sl_2)$-modules $L(n\omega)$ for $n\in \N$ even (or  respectively, odd); see \cite{BW18a, BeW18}.

There is an $\A$-form $\dUiA$ in the modified iquantum group $\dUi$, for parameters $\bvs$ which satisfies suitable constraints \eqref{vsi_same} and \eqref{parameter_iCB}; the icanonical basis of $\dUi$ forms an $\A$-basis for $\dUiA$ \cite{BW18b, BW21}. For $i\in \Ifin$, we introduce linear operators 
\[
\TTd'_{i,-1}: M \longrightarrow M,
\qquad
\TTd''_{i,+1}: M\longrightarrow M 
\]
by rescaling the definitions of $\TT'_{i,-1}$ and $\TT''_{i,+1}$ on any integrable $\Ui$-module $M$; they are now given by actions of elements ${\bf t}_i', {\bf t}_i''$ in (a completion of) $\dUiA$ on $M$; we refer to formulas \eqref{ti'}--\eqref{ti''}, \eqref{diagti'}--\eqref{diagti''} and \eqref{qss-ti'}--\eqref{qss-ti''} for ${\bf t}_i', {\bf t}_i''$, and list the formulas for ${\bf t}_i'$ below:
\vspace{2mm}

$\boxed{\text{(i). Split type}~ c_{i,\tau i}=2,\ \vs_i=q_i^{-1}}:$
\begin{align*}
{\bf t}_i' = \sum_{m\geq 0} \sum_{\substack{\ov{\lambda}\in X_\imath\\\ov{\la_i}=\ov{0}}} (-1)^{m}  q_i ^{-m} \dvi{\ov{0} }{2m}\one_{\ov{\lambda}}
+\sum_{m\geq 0} \sum_{\substack{\ov{\lambda}\in X_\imath\\\ov{\la_i}=\ov{1}}} (-1)^{m} q_i^{-m} \dvi{\ov{1} }{2m+1}\one_{\ov{\lambda}};
\end{align*}
$\boxed{\text{(ii). Diagonal type}~ c_{i,\tau i}=0,\ \vs_i=\vs_{\tau i}=1}:$ 
\begin{align*}
    {\bf t}_i' =   \sum_{ \ov{\lambda}\in X_\imath }  \sum_{a-b=\la_{i,\tau}} (-1)^b q_i^{-b}  B_{i}^{(a)}  B_{\tau i}^{(b)} \one_{\ov{\lambda}};
\end{align*}
$\boxed{\text{(iii). Quasi-split type}~ c_{i,\tau i}=-1,\ \vs_{\tau i} \in q_i^{\Z}}:$ 
\begin{align*}
    {\bf t}_i' =   \sum_{ \ov{\lambda}\in X_\imath }  q_i^{-\la_{i,\tau}(\la_{i,\tau}-1)/2}  \sum_{t\ge 0,l\ge 0} (-1)^{t+l}  q_i^{ -\frac{(t-l)(t-l+1)}{2}-t-l+\la_{i,\tau}(t-l)}  \vs_{\tau i}^{-\la_{i,\tau}+t-l}   B^{(t,t+l,l)}_{i} \one_{\ov{\lambda}}.
\end{align*}
By definition, $\TTd'_{i,-1}$ and $\TTd''_{i,+1}$ preserve an $\A$-form $M_\A$ of $M$. These elements ${\bf t}_i', {\bf t}_i''$ should be viewed as $\mathrm{i}$-counterparts of the elements $s_{i,-1}', s_{i,+1}''$ in \eqref{si'si''} or \cite[2.2(b)]{Lus09}; the formula in (ii) is the same as Lusztig's.

\begin{alphatheorem}  [\cref{thm:integralTi}]
\label{thm:main4-integral}
Let $i\in \Ifin$. 
\begin{enumerate}
    \item 
For any $x\in \dUi$, there exist unique $x',x''\in \dUi$ such that
$
x' {\bf t}_i' ={\bf t}_i' x,\ x'' {\bf t}_i'' ={\bf t}_i'' x.
$
\item
The maps $\TTd'_{i,-1} \colon \dUi \rightarrow \dUi\; (x\mapsto x')$ and $\TTd''_{i,+1}\colon \dUi \rightarrow \dUi \; (x\mapsto x'')$ are automorphisms of $\dUi$, mutually inverse to each other; see Table~\ref{table:rkone-int}--\ref{table:rktwo-int2} for explicit formulas. 
\item 
Let $M$ be an integrable $\Ui$-module. For any $v\in M$ and $x\in \dUi$, we have
$\TTd'_{i,-1} (xv) =\TTd'_{i,-1}(x) \TTd'_{i,-1}(v)$ and $\TTd''_{i,+1}(xv) =\TTd''_{i,+1}(x) \TTd''_{i,+1}(v)$.
\item 
The automorphisms $\TTd'_{i,-1} $ and $\TTd''_{i,+1}$ of $\dUi$ preserve the $\A$-form $\dUiA$. 
\end{enumerate}
\end{alphatheorem}

Formulas for actions of $\TTd'_{i,-1} $ and $\TTd''_{i,+1}$ on idivided powers in $\dUi$ are available based on the formulas in \cref{thm:main3-iDP}; see \cref{prop:dotTiDP}. 

For $i,j\in \fwItau$ such that $j\neq i,\tau i$, we denote by $m_{ij}$ the order of $\bs_i\bs_j$ in the relative Weyl group $W^\circ$; see \eqref{risi}. Just as in \cite{Lus93}, we can define two more variants of relative braid group operators $\TTd''_{i,-1}$ and $\TTd'_{i,+1}$; see \eqref{TTdM}.

\begin{alphatheorem}  [\cref{thm:TTd_braid,cor:braid}]
\label{thm:main5-braid}
    Let $i,j\in \Ifin$ such that $j\neq i,\tau i$ and $m_{ij}<\infty$. When acting on $\dUi$ or on any integrable $\Ui$-module the following relative braid relations hold:
    \begin{align*}
        \underbrace{\TTd'_{i,-1} \TTd'_{j,-1} \TTd'_{i,-1} \ldots}_{m_{ij}} 
        &=
        \underbrace{\TTd'_{j,-1} \TTd'_{i,-1} \TTd'_{j,-1} \ldots}_{m_{ij}} \ ,
        \\
        \underbrace{\TTd''_{i,+1} \TTd''_{j,+1} \TTd''_{i,+1} \ldots}_{m_{ij}} 
        &=
        \underbrace{\TTd''_{j,+1} \TTd''_{i,+1} \TTd''_{j,+1} \ldots}_{m_{ij}} \ .
    \end{align*}
    Moreover, the operators $\TT'_{i,-1}$ as well as $\TT''_{i,+1}$ on any integrable $\Ui$-module satisfy the same relative braid relations as above.
\end{alphatheorem}

Let us briefly comment on the proofs of these two theorems, say for $\TTd'_{i,-1}$, which is a rescaled version of $\TT'_{i,-1}$. The integrality of $\TTd'_{i,-1}$ in \cref{thm:main4-integral}(4) follows from the integrality of $\TTd'_{i,-1}$ on integrable $\Ui$-modules. The compatibility in \cref{thm:main4-integral}(3) as well as \cref{thm:main5-braid} follow from their (non-modified) counterparts for $\TT'_{i,-1}$ in \cref{thm:main1Intro} and in \cite{WZ23} after some additional checking on the compatibility of scalars arising from the rescaling. 


It will be a natural though very challenging question to find explicit rank one formulas on integrable modules over iquantum groups {\em beyond} the quasi-split type.

The relative braid group action has played an indispensable role in Drinfeld presentations of affine iquantum groups (cf. \cite{LWZ24, LPWZ25}) as it is used for the construction of various $q$-root vectors as new Drinfeld generators. In particular, the rank one formulas in this paper have been used to prove some crucial identities stated in these papers. 

Lusztig's rank one braid group formula for quantum groups admits a categorification in terms of Rickard complexes \cite{CR08}, which has remarkable applications to derived equivalences. We expect that our 3 rank one relative braid group formulas will lead to iRickard complexes in $2$-iquantum groups and derived equivalences in settings where $\mathfrak{sl}_2$-categorification may not be available. 
\subsection{Organization}

The paper is organized as follows. 
In \cref{sec:prelim}, we set up the basics on quantum groups and Lusztig's braid group symmetries. We review from \cite{WZ23} the relative braid group symmetries on iquantum groups, and their connections to Lusztig's braid group symmetries via quasi $K$-matrix. We introduce their higher order analogue of $q$-root vectors in iquantum groups, and present some new higher order Serre-Lusztig relations.

We then deal with the proofs of the main results in \cref{thm:main1Intro,thm:main2Intr,thm:main3-iDP} separately in the 3 cases: (i) split, (ii) diagonal, and (iii) quasi-split types of rank one.
In \cref{sec:split} and Appendix~\ref{sec:proof}, we prove these 3 main theorems in the case for $i=\tau i$. In \cref{sec:diagonal}, we prove these 3 main theorems in the case for $c_{i,\tau i}=0$.

The last and most complex case for $c_{i,\tau i}=-1$ occupies Sections~\ref{sec:quasi basis} and \ref{sec:quasisplit}. We first study systematically the quasi-split rank one quantum symmetric pair $(\U(\sl_3),\Ui(\sl_3))$ in \cref{sec:quasi basis}. Recall Lusztig's canonical basis for the simple $\U(\sl_3)$-module $L(m\omega_1+n\omega_2)$ (descending from $\U^-(\sl_3)$) can be made very explicit. We obtain a remarkable $\imath$-version of this basis (called monomial basis) for
$L(m\omega_1+n\omega_2)$ viewed as an $\Ui$-module in terms of idivided powers. Explicit transition matrices between the canonical basis and the monomial basis on $L(m\omega_1+n\omega_2)$ are obtained. 
In \cref{sec:quasisplit}, we prove the main results in \cref{thm:main1Intro,thm:main2Intr,thm:main3-iDP} in the quasi-split rank one case when $c_{i,\tau i}=-1$. The transition matrices of bases obtained in Section~\ref{sec:quasi basis} play a crucial role here. 

In \cref{sec:integral}, for each of the three rank one finite type, we rescale the  braid operators $\TT'_{i,-1}$ and $\TT''_{i,+1}$ to define new operators $\TTd'_{i,-1}$ and $\TTd''_{i,+1}$ on integrable $\Ui$-modules. Through compatibility, this induces symmetries on $\dUi$ and its integral $\A$-form $\dUiA$ which again satisfy the relative braid group relations; see \cref{thm:main4-integral,thm:main5-braid}. Tables for formulas of actions of $\TTd'_{i,-1}$ and $\TTd''_{i,+1}$ on the generators can be found in Section~\ref{sec:tables}. 

The main results of this paper hold in a slight generality of locally quasi-split iquantum groups as in \cite{Z23}.

\section{Preliminaries}
\label{sec:prelim}

In this section, we review basic on quantum groups and iquantum groups, including quasi K-matrix, idivided powers and relative braid group action on $\Ui$. Several formulas and identities in \S\ref{subsec:rootvector_diag}
--\ref{subsec:rootvector_qs} involving higher order $q$-root vectors are new.

\subsection{Quantum groups}

Let $(\I,\cdot)$ be a Cartan datum and $(Y,X,\langle \cdot ,\cdot\rangle,\dots)$ be a root datum of type $(\I,\cdot)$; cf. \cite[2.2.1]{Lus93}. i.e., we have
\begin{itemize}
\item[(a)] a perfect bilinear pairing $\langle\cdot,\cdot\rangle:Y\times X\rightarrow\Z$ between two finitely generated free abelian groups $Y,X$;
\item[(b)] an embedding $\I\hookrightarrow X$ ($i\mapsto \alpha_i$) and an embedding $\I\hookrightarrow Y$ ($i\mapsto h_i$) such that $\langle h_i,\alpha_j\rangle =2\frac{i\cdot j}{i\cdot i}$ for all $i,j\in \I$.
\end{itemize}

We call $X$ the weight lattice. Let $c_{ij} = \langle h_i,\alpha_j\rangle $ and $C=(c_{ij})_{i,j\in \I} $ be the Cartan matrix. 
Let $s_i:Y\rightarrow Y,i\in \I$ be the simple reflection given by $s_i(\mu)=\mu-\langle \mu,\alpha_i\rangle h_i$. 
Let $W=\langle s_i\mid i\in \I \rangle$ be the Weyl group with the length function $\ell(\cdot)$. The group $W$ also acts on $X$ via $s_i(\lambda)=\lambda-\langle h_i,\lambda \rangle \alpha_i$. Let $\g$ be the symmetrizable Kac-Moody algebra associated to the Cartan datum $(\I,\cdot)$.

Let $q$ be an indeterminate and set $q_i=q^{\frac{i\cdot i}{2}}$. We denote, for $i\in \I,m,r\in\N$,
\begin{align}
[r]_i =\frac{q_i^r-q_i^{-r}}{q_i-q_i^{-1}},
 \quad
[r]_i!=\prod_{s=1}^r [s]_i, \quad 
\qbinom{m}{r}_{q_i} =\qbinom{m}{r}_i =\frac{[m]_i! }{[r]_i! [m-r]_i!}.
\end{align}
Let $\Q(q)$ be the field of rational functions in $q$ and $\F$ be the algebraic closure of $\Q(q)$.

For $A, B$ in an $\F$-algebra, we shall denote $[A,B]_{q^a} =AB -q^aBA$, and $[A,B] =AB - BA$. The  Drinfeld-Jimbo quantum group $\U$ is defined to be the $\F$-algebra generated by $E_i,F_i$ for $i\in\I$, and $K_\mu$, $\mu\in Y$, subject to the following relations:
\begin{align}
K_0=1, & \quad K_{\mu}K_{\nu}=K_{\mu+\nu}, \quad \quad \forall \mu,\nu\in Y,
\label{eq:KK}
\\
K_\mu E_i &=q^{\langle \mu,\alpha_i\rangle} E_iK_\mu, \qquad K_\mu F_i =q^{-\langle \mu,\alpha_i\rangle} F_iK_\mu, 
\\
[E_i,F_j] &=\delta_{ij}\frac{K_i-K_i^{-1}}{q_i-q_i^{-1}}, \qquad \text{ where } K_i=K_{h_i}^{\frac{i\cdot i}{2}},
\label{Q4}
\end{align}
and the quantum Serre relations, for $i\neq j \in \I$,
\begin{align}
& \sum_{r=0}^{1-c_{ij}} (-1)^r  E_i^{(r)} E_j  E_i^{(1-c_{ij}-r)}=0,
  \label{eq:serre1} \\
& \sum_{r=0}^{1-c_{ij}} (-1)^r  F_i^{(r)} F_j  F_i^{(1-c_{ij}-r)}=0,
  \label{eq:serre2}
\end{align}
where $E_i^{(r)}=\frac{E_i^r}{[r]_i !}$, $F_i^{(r)}=\frac{F_i^r}{[r]_i !}$ are the divided powers.

The algebra $\U$ is $\Z\I$-graded by the weight function $\wt$ where $\wt(E_i)=\alpha_i,\wt(F_i)=-\alpha_i,\wt(K_\mu)=0$ for $i\in \I,\mu\in Y$.

Let $\psi$ be the bar involution on $\U$ and $\sigma$ be the anti-involution on $\U$ such that
\begin{align}
&\psi: q\mapsto q^{-1}, \quad E_i\mapsto E_i, \quad F_i\mapsto F_i, \quad K_\mu \mapsto K_{-\mu}.\\
&\sigma:  E_i\mapsto E_i, \quad F_i\mapsto F_i, \quad K_\mu \mapsto K_{-\mu}.
\end{align}

We denote, for $t\in \Z,s\in \N,i\in \I$, 
\[
[K_i;t]_i =\frac{K_i q_i^t -K_i^{-1} q_i^{-t}}{q_i - q_i^{-1}}, 
\qquad 
\qbinom{K_i;t}{s}_i =\frac{\prod_{j=1}^s [K_i;t-j+1]_i}{[s]_i!}.
\]
 
 \begin{lemma}
 \cite[\S3.1.9]{Lus93}  
    \label{lem:EF}
 We have, for $\ell,k\geq 0, $
 \begin{align}
 E_i^{(\ell)} F_i^{(k)} =\sum_{s=0}^{\min(\ell,k)}  F_i^{(k-s)}\qbinom{K_i;2s-\ell-k+1-j}{s}_i  E_i^{(\ell-s)}.
 \end{align}
 \end{lemma}

Let $\Br(W)$ be the braid group associated to the Weyl group $W$. Lusztig introduced the braid group symmetries $T_{i,e}',T_{i,e}''$, for $i\in\I$ and $e=\pm1$, on the quantum group $\U$; cf. \cite[\S37.1.3]{Lus93}. We recall the formulation of $T'_{i,-1}$ below.  

\begin{proposition}
[\cite{Lus90a}]
   \label{prop:braid1}
Set $\alpha=-c_{ij}$. There exists an automorphism $T'_{i,-1}$, for $i\in \I$, on $\U$ such that
\begin{align*}
&T'_{i,-1}(K_\mu)= K_{s_i(\mu)},
\qquad \;\;\forall \mu\in Y,\\
&T'_{i,-1}(E_i)=- K_i  ^{-1} F_i  ,\qquad T'_{i,-1}(F_i)=- E_i K_i ,\\
&T'_{i,-1}(E_j)= \sum_{s=0}^\alpha (-1)^s q_i^{-s} E_i^{(s)} E_j E_i^{(\alpha-s)},\qquad  j\neq i,\\
&T'_{i,-1}(F_j)= \sum_{s=0}^\alpha (-1)^s q_i^{s} F_i^{(\alpha-s)} F_j F_i^{(s)}, \qquad  j\neq i.
\end{align*}
Moreover, $T'_{i,-1}$, for $i\in \I$, satisfy the braid relations in $W$.    
\end{proposition}

The symmetries  $T_{i,e}',T_{i,e}''$, for $e=\pm 1$, are related to each other via conjugations by the bar involution $\psi$ or the anti-involution $\sigma$
\begin{align*}
T'_{i,e}=\sigma T''_{i,-e} \sigma,\qquad T'_{i,e}=\psi T'_{i,-e} \psi,\qquad T''_{i,e}=\psi T''_{i,-e} \psi.
\end{align*}
 
For a reduced expression $w = s_{i_1}
\cdots s_{i_r}$ of $w \in W$, we define
$T'_{w,e} := T'_{i_1,e}\cdots T'_{i_r,e}$ in $\mathrm{Aut}(\U),$
which is independent of the choice of reduced expressions. Similarly, we define the automorphisms $T''_{w,e} $ of $\U$.

Let $\ba = (a_i)_{i\in\I} \in (\F^\times)^\I
$. Following \cite[Proposition 2.1]{WZ23}, we have an automorphism $\Phi_\ba$ on $\U$ such that
\begin{align}
\label{def:Phi}
    \Phi_\ba:K_i\mapsto K_i,\quad E_i\mapsto a_i^{1/2}E_i,\quad F_i\mapsto a_i^{-1/2}F_i.
\end{align}

A $\U$-module is called integrable if $E_i,F_i$ act locally nilpotently. Let $M$ be any integrable $\U$-module with a weight space decomposition
\[
M=\bigoplus_{\mu \in X} M_\mu. 
\]
Following \cite[5.2.1]{Lus93}, there are linear operators $T'_{i,e},T''_{i,e}$, for $e=\pm1$, on $M$ defined by
\begin{align}\label{eq:mod1}
T'_{i,e}(v)&=\sum_{\substack{a,b,c \geq 0;\\a- b+c=\lambda_i}}(-1)^b q_i^{e(b-ac)} F_i^{(a)}E_i^{(b)}F_i^{(c)}v, \\
T''_{i,e}(v)&=\sum_{\substack{a,b,c \geq 0;\\-a+b-c=\lambda_i}}(-1)^b q_i^{e(b-ac)} E_i^{(a)}F_i^{(b)}E_i^{(c)}v,
\label{eq:mod2}
\end{align}
for  $v\in M_\lambda,\lambda_i=\langle h_i,\lambda\rangle.$ For $e=\pm 1$ and $i\in \I$, following \cite[2.1]{Lus09} we define
\begin{align}
\label{si'si''}
    s_{i,e}' &=\sum_{a,b\in \Z, \la \in X,\la_i =a+b} (-1)^b q_i^{eb} F_i^{(a)} 1_\la E_i^{(b)},
    \\
    s_{i,e}'' &=\sum_{a,b\in \Z, \la \in X,\la_i =a+b} (-1)^a q_i^{ea} F_i^{(a)} 1_\la E_i^{(b)}.
\end{align}
By \cite[2.2(b)]{Lus09}, we have
\begin{align}
    T_{i,e}'(v) =s_{i,e}' v,
    \qquad
    T_{i,e}''(v) =s_{i,e}'' v.
\end{align}

The symmetries $T'_{i,e},T''_{i,e}$ on $M$ satisfy the following compatibility condition: for any $u\in \U,v\in M$,
\begin{equation}\label{eq:mod3}
T'_{i,e}(uv)=T'_{i,e}(u)T'_{i,e}(v),\qquad T''_{i,e}(uv)=T''_{i,e}(u)T''_{i,e}(v).
\end{equation}
\subsection{The iroot datum and iquantum groups} \label{subsec:irootdatum}

 We call a permutation $\tau$ of the set $\I$ an involution 
of the Cartan datum $(\I, \cdot)$ if $\tau^2 =\Id$ and $\tau(i) \cdot \tau(j) = i \cdot j$ for $i$, $j \in \I$.  Note we allow $\tau =1$. 
Such involutions $\tau$ are naturally in bijection with diagram involutions on the Dynkin diagram of $\g$. We call a pair $(\I,\tau)$ a (quasi-split) {\em Satake datum} or a {\em Satake diagram}. 

An involution $\tau$ of $(\I, \cdot)$ naturally induces an involution on $\g$ and an involution on $\U$, both denoted again by $\tau$. Let $\omega$ be the Chevalley involution on $\g$. Associated to $(\I,\tau)$, we have a quasi-split symmetric pair $(\g,\g^{\omega \tau})$. 

We shall always assume that $\tau$ extends to an involution on $X$ and an involution on $Y$, respectively, such that the perfect bilinear pairing is invariant under the involution $\tau$. 
Following \cite{BW18b}, we introduce
\begin{align}
  \label{XY}
 \begin{split}
X_{{\imath}} &= X \big / \{ \la +\tau(\la) \vert \la \in X\},
 \\
Y^{\imath} &= \{h \in Y \big \vert \tau (h) =-h \}.
\end{split}
\end{align}

We denote by $\overline{\la} \in X_{\imath}$ (and call it an iweight) the image of $\la \in X$; we shall refer to $\la$ a lift of $\overline{\la}$ and call $X_{\imath}$ the iweight lattice. This induces a well-defined bilinear pairing  
$$ 
\langle \cdot, \cdot \rangle: 
Y^{\imath}  \times X_{\imath} \longrightarrow \Z
$$
defined by $\langle \gamma, \overline{\la} \rangle  : = \langle \gamma, \la \rangle$, 
where $\la \in X$ is any preimage of $\overline{\la}$ and $\gamma \in Y^\imath$.

The quasi-split {\em iquantum group} $\Ui_{\bvs}$ associated to the Satake datum $(\I,\tau)$ with parameter $\bvs=(\vs_i)_{i\in \wI}$ is the subalgebra of $\U$ generated by
\begin{align}
\label{def:iQG}
\begin{split}
&B_i=F_i+ \vs_i E_{\tau i} K_i^{-1}, \qquad k_i =K_i K_{\tau i}^{-1}, \qquad (i\in \wI).
\end{split}
\end{align}
We always assume that the parameter $\bvs=(\vs_i)_{i\in \wI}$ satisfies that (cf. \cite{Let99}, \cite[\S5.1]{Ko14})
\begin{align}
\label{vsi_same}
\vs_i=\vs_{\tau i},\qquad \text{ if } i\cdot \tau i=0.
\end{align}
We call a parameter $\bvs=(\vs_i)_{i\in \wI}$ {\em balanced} if $\vs_i=\vs_{\tau i}$ for each $i\in \wI$. The pair $(\U,\Ui_\bvs)$ forms a quantum symmetric pair and its $q\rightarrow 1$ limit is the classical symmetric pair $(\g, \g^{\omega \tau})$. Let $\U^{\imath0}$ be the subalgebra of $\Ui$ generated by $k_i$, for $i\in \wI$. 

\begin{lemma} \cite[Lemma 2.5.1]{Wat21} 
\label{lem:iso-parameter}
For any two parameters $\bvs,\bvs'$, there exists an $\F$-algebra isomorphism \begin{align} \label{phibvs2}
\begin{split}
\phi_{\bvs,\bvs'} :\Ui_{\bvs} & \longrightarrow \Ui_{\bvs'}
\\
B_i  \mapsto \sqrt{\vs_{ i} (\vs_i')^{-1}} B_i, 
& \quad 
k_i\mapsto \sqrt{\vs_{ i} \vs_{\tau i}' (\vs_i' \vs_{\tau i})^{-1}}k_i, \quad \text{for }i\in \I.
\end{split}
\end{align}
\end{lemma}
It is also clear that $\phi_{\bvs',\bvs''}\circ \phi_{\bvs,\bvs'}=\phi_{\bvs ,\bvs''}$.

\begin{lemma}[\text{cf. \cite[Propositions 3.12 and 3.14]{WZ23}}] 
   \label{lem:inv}
$\qquad$
\begin{enumerate}
    \item For an arbitrary parameter $\bvs$, there exists an anti-involution $\sigma_\tau$ on $\Ui$  such that $ \sigma_\tau (B_i) = B_{\tau i}, \sigma_\tau (k_i) =k_i$, for $i\in \wI$. 
    \item For any balanced parameter $\bvs$, there exists an involution $\widehat\tau$ on $\Ui$ such that $ \widehat\tau (B_i) = B_{\tau i}, \widehat\tau (k_i) =k_i^{-1}$, for $i\in \wI$. 
    \item For any balanced parameter $\bvs$, there exists an anti-involution $\sigma_\imath$ on $\Ui$ which fixes $B_i$ and sends $k_i \mapsto k_i^{-1}$, for $i\in \wI$. 
\end{enumerate}
\end{lemma}

For a balanced parameter $\bvs$, $\sigma_\tau$ is simply the composition of $\sigma_\imath$ and $\htau$.

\subsection{Relative Weyl groups}

The relative Weyl group associated to the quasi-split Satake diagram $(\I,\tau)$ is defined by
\[
 W^\circ :=\{w\in W\mid w\tau =\tau w\}.
\]

Let $\wItau$ be a fixed set of representatives for $\tau$-orbits on $\wI$. The {\em real rank} of a Satake diagram $(\I,\tau)$ is defined to be the cardinality of $\wItau$. Denote
\begin{align}
\label{def:Ifinite}
\fwItau=\{i\in \wItau \mid c_{i,\tau i}=2,0,-1\},
\qquad
\Ifin =\{i\in \I \mid c_{i,\tau i}=2,0,-1\}.
\end{align}
i.e., $\Ifin$ contains exactly those $i\in \I$ such that the Dynkin subdiagram $\{i,\tau i\}$ is finite type.
For $i\in \Ifin$, define $\bs_i$ to be the longest element in the Weyl group associated with the subdiagram $\{i,\tau i\}$, that is, 
\begin{align} \label{risi}
\bs_i=\begin{cases}
s_i, & \text{ if } c_{i,\tau i}=2,
\\
s_i s_{\tau i}, & \text{ if } c_{i,\tau i}=0,
\\
s_i s_{\tau i}s_i, & \text{ if } c_{i,\tau i}=-1.
\end{cases}
\end{align}
Note that $\bs_i =\bs_{\tau i}$.

 For each Satake diagram $(\I,\tau)$, there is a relative root system with the simple system $\{\ov{\alpha}_i \mid i\in \fwItau\}$, where $\ov{\alpha}_i$ is identified with the following element (cf. \cite[\S 2.3]{DK19})
\begin{align}
 \label{def:ovalpha}
\ov{\alpha}_i:
=\frac{\alpha_i+\alpha_{\tau i}}{2},\qquad (i\in \fwItau).
\end{align}
The group $W^\circ$ can be identified with the Weyl group of the relative root system with simple reflections $\{\bs_i \mid i\in \fwItau\}$.


Let $(\oc_{ij})_{i,j\in \fwItau}$ be the Cartan matrix for the relative root system. Using \eqref{def:ovalpha}, we have 
\begin{align}
\label{eq:ovc}
\oc_{ij}=\frac{c_{i,j}+c_{\tau i,j}}{1+\frac{1}{2}c_{i,\tau i}}.
\end{align}

\subsection{Quasi $K$-matrix and relative braid group action}

 \begin{proposition} \cite[Theorem 3.16]{WZ23}
   \label{prop:qK}
 Let $(\U,\Ui_\bvs)$ be a quasi-split quantum symmetric pair of arbitrary Kac-Moody type with a general parameter $\bvs$. There exists a unique element $\fX_{\bvs}=\sum_{\mu \in \N \I} \fX_{\bvs}^\mu$, for $\fX_{\bvs}^\mu\in \U_\mu^+$, such that $\fX_{\bvs}^0=1$ and the following identities hold:
 \begin{align}\label{eq:fX1av}
 B_{\tau i}   \fX_{\bvs} &=  \fX_{\bvs}  \sigma \tau\big(B_i\big),\\
 x  \fX_{\bvs}  &=  \fX_{\bvs}  x,
 \end{align}
 for $i\in \wI$ and $x \in \U^{\imath 0}  $. Moreover, $\fX_\bvs^\mu =0$ unless $\tau (\mu) = \mu$.
 \end{proposition}
 
The element $\fX=\fX_{\bvs}$ is called the quasi $K$-matrix associated to $(\U,\Ui_\bvs)$; we omit the index $\bvs$ if there is no ambiguity. For $i\in\I$, denote by $\fX_i=\fX_{i,\bvs}$ the quasi $K$-matrix associated to the real rank one subdiagram $( \{i,\tau i\} ,\tau|_{\{i,\tau i\} })$. The formulation of $\fX$ for a general parameter $\bvs$ in \cref{prop:qK} was inspired by \cite{AV20}, who studied quasi $K$-matrix in general parameters in a more complicated formulation. In case of special parameters $\bvs$ so that $\Ui_\bvs$ admits a bar involution, the quasi $K$-matrix was first formulated as an intertwiner for bar involutions on $\Ui_\bvs$ and on $\U$; cf. \cite{BW18a, BK19, BW18b}.

 Following \cite{WZ23}, let $\bvs_{\dm}=(\vs_{\dm,i})_{i\in \wI}$ be the distinguished parameter given by
\begin{align}
\label{def:bvs}
\vs_{\diamond,i}= - q_i^{-1-\frac12 c_{i,\tau i}} .
\end{align} 
Let $\ov{\cdot}:\F\rightarrow \F$ be the $\Q$-linear map sending $q\mapsto q^{-1}$ and set $\ov{\bvs_\dm}\bvs=(\ov{\vs_{i,\dm}}\vs_i)_{i\in \wI}$. 

Recall the rescaling automorphism $\Phi_{\ba}$ of $\U$ \eqref{def:Phi}. Following \cite[(10.1)]{WZ23}, we define the normalized braid group symmetries on $\U$:
\begin{align}
\label{braid_Urescale}
   T'_{i,\bvs, -1} :=\Phi_{\ov{\bvs_\diamond}\bvs} T'_{i,-1} \Phi_{\ov{\bvs_\diamond}\bvs}^{-1},
   \qquad
   T''_{i,\bvs, +1} :=\Phi_{\ov{\bvs_\diamond}\bvs} T''_{i,+1} \Phi_{\ov{\bvs_\diamond}\bvs}^{-1},
   \qquad
   i\in\Ifin.
\end{align}
Accordingly, we have $T'_{w,\bvs, +1},T''_{w,\bvs, +1}$ on $U$, for $w\in W$.
There are also linear operators $T'_{i,\bvs, -1},T''_{i,\bvs, +1}$ on an integrable $\U$-module $M$ which are compatible with these normalized braid group symmetries on $\U$. The actions of $T'_{i,\bvs, -1},T''_{i,\bvs, +1}$ on $M$ is given by
\begin{align}
\label{braid_rescale}
T'_{i,\bvs,-1}v=(\ov{\vs_{\dm,i}}\vs_i)^{-m/2} T'_{i,-1}v,
\qquad
T''_{i,\bvs,+1}v=(\ov{\vs_{\dm,i}}\vs_i)^{-m/2} T''_{i,+1}v,
\end{align}
for $v\in M_\lambda,m=\langle h_i,\lambda\rangle.$

We constructed relative braid group symmetries $\TT'_{i,e},\TT''_{i,e}$ on $\Ui$ in \cite{WZ23,Z23}. We recall the formulation and properties for $\TT'_{i,-1}$. 

\begin{theorem} [\text{cf. \cite{WZ23,Z23}}]
\label{thm:braid-iQG}
Let $e=\pm1$ and $\bvs$ be an arbitrary parameter. For $i\in\fwItau$, there exists an automorphism $\TT'_{i,\bvs,-1}$ on $\Ui_\bvs$ such that
\begin{itemize}
\item[(a)] 
for a balanced parameter $\bvs$, we have
\begin{align}
\label{eq:compTT}
\TT'_{i,\bvs,-1}(x) \fX_i =\fX_i T'_{\bs_i,\bvs, -1} (x),\qquad \forall x\in \Ui.
\end{align}
\item[(b)] 
for arbitrary parameters $\bvs,\bvs'$, we have
\begin{align}
\label{eq:phiTT}
\phi_{\bvs,\bvs'}\circ \TT'_{i,\bvs,-1}=\TT'_{i,\bvs',-1} \circ \phi_{\bvs,\bvs'}.
\end{align}
\end{itemize}
Moreover, the symmetries $\TT'_{i,-1}$ satisfy braid relations in $W^\circ$.
\end{theorem}
(We write $\TT'_{i,-1}$ for $\TT'_{i,\bvs,-1}$ whenever the parameter $\bvs$ is clear from the context.)

\begin{proof}
Let $\tUi$ be the universal iquantum groups \cite{LW22} (also cf. \cite{WZ23}). The algebra $\Ui_\bvs$ is recovered via a central reduction on $\tUi$. The relative braid group symmetries $\tTT'_{i,-1}$ were constructed on $\tUi$ in \cite[Theorem 3.1]{Z23} in the Kac-Moody setting, generalizing \cite{WZ23} in the finite-type setting. 

In the Kac-Moody setting, the symmetries $\TT'_{i,\bvs_\dm,-1}$ for the distinguished parameter $\bvs_\dm$ is obtained from $\tTT'_{i,-1}$ via the central reduction, and $\TT'_{i,\bvs ,-1}$ for a general parameter $\bvs$ is defined by
\[
 \TT'_{i,\bvs,-1}=\phi_{\bvs_\dm,\bvs}\circ\TT'_{i,\bvs_\dm,-1} \circ \phi_{\bvs,\bvs\dm};
\]
this construction is similar to the one in \cite[\S9.4]{WZ23} for finite type. The identity \eqref{eq:compTT} is then established in the same way as \cite[Theorem 9.10]{WZ23}, while the identity \eqref{eq:phiTT} is established in the same way as \cite[Theorem 10.1]{WZ23}.
\end{proof}

 \subsection{$\mathrm{i}$Divided powers and integrable $\Ui$-modules}

We recall from \cite{BW18a, BeW18, CLW21a} the idivided powers.

\begin{itemize}
\item[(i)] Let $i=\tau i$. The idivided powers $B_{i,\ov{p},\bvs}^{(m)}$ or simply $\dvi{\ov{p}}{m}$ when there is no confusion, for $ m\ge 1$ and $\ov{p}\in \Z_2$, are defined by
\begin{align}
\label{def:idv}
\begin{split}
&\dvi{\ov{0},\bvs}{m}=\frac{1}{[m]_i!}
B_i^{1+\delta_{\ov{m},\ov{0}}} \prod_{r=1}^{\lceil m/2 \rceil -1}  (B_i^2- q_i \vs_i [2r]_i^2), 
\\
&\dvi{\ov{1},\bvs}{m}=\frac{1}{[m]_i!}
B_i^{\delta_{\ov{m},\ov{1}}} \prod_{r=1}^{\lfloor m/2 \rfloor} (B_i^2- q_i \vs_i [2r-1]_i^2).
\end{split}
\end{align}
Also, set $B_{i,\ov{p}}^{(0)}=1$. The idivided powers satisfy the recursive relations:
\begin{align} \label{recursion}
  B \dvi{\ov{p}}{m} =
  \begin{dcases}
   [m+1]_{i} \dvi{\ov{p}}{m+1}  , & \text{ if } \ov{p} \not = \ov{m},
   \\
   [m+1]_{i} \dvi{\ov{p}}{m+1} + q_i\vs_i [m]_i \dvi{\ov{p}}{m-1} , & \text{ if } \ov{p} = \ov{m}.
    \end{dcases}
\end{align}

\item[(ii)] 
Let $i\neq \tau i$. The idivided powers $B_i^{(m)}$, for $m \ge 0$, are defined by 
\begin{align}
\label{eq:iDPusual}
    B_i^{(m)}=\frac{B_i^m}{[m]_i!}.
\end{align} 
\end{itemize}

\begin{lemma}\label{lem:div-par}
Let $i\in \wI$ such that $i=\tau i$. Let $\bvs=(\vs_i)_{i\in \I}$ and $\bvs'=(\vs'_i)_{i\in \I}$ be two parameters. We have 
\[
\phi_{\bvs,\bvs'}\Big(\dvi{\ov{p},\bvs}{m}\Big)=\vs_i^{m/2}(\vs_{i}')^{-m/2} \dvi{\ov{p},\bvs'}{m},
\]
for all $m\in \N$ and $\ov{p}\in \Z_2$.
\end{lemma}

\begin{proof}
This follows from Lemma~\ref{lem:iso-parameter} and the definition of $\dvi{\ov{p},\bvs}{m}$.
\end{proof}

Following \cite[Section 3.3]{BW18b}, an $\Xi$-weighted $\Ui$-module $M$ is a $\Ui$-module with a decomposition $M =\oplus_{\ov{\la} \in \Xi} M_{\ov{\la}}$ such that 
\[
  K_\gamma x = q^{\langle \gamma, \overline{\la} \rangle} x,\qquad\forall \gamma \in Y^\imath,\ov{\la}\in X_\imath,x\in M_{\ov{\la}}.
\]
By \eqref{def:iQG}, we have $B_i v\in M_{\ov{\lambda-\alpha_i}}$ for any $v\in M_{\ov{\lambda}}$.

\begin{definition}
\label{def:integrable}
A $\Ui$-module $M$ is called {\em integrable} if $M$ is $\Xi$-weighted and each $x\in M_{\ov{\la}}$ for any $\ov{\la}\in X_\imath$ is annihilated by $B_j^{(k)}$ ($\tau j\neq j\in \wI$) and $B_{i,\ov{\la}_i}^{(k)}$ ($\tau i= i\in \wI$), for $k\gg 0$.
\end{definition}

Denote by $\mathcal C$ the category of integrable $\Ui$-modules. We shall always assume that the $\Ui$-modules in this paper lie in $\mathcal C$. 

\begin{remark}
Definition~\ref{def:integrable} was known to us for a number of years. 
In \cite[Def.~ 3.3.4]{Wat24b}, Watanabe formulated a notion of integrable $\Ui$-modules in a different way. By \cite[Prop. 4.3.1]{Wat24b}, any module in the category $\mathcal C$ is integrable in the sense of that paper.
\end{remark}

\begin{lemma}  \label{lem:integrable}
    Any integrable $\U$-module vied as a $\Ui$-module by restriction lies in $\mathcal C$. 
\end{lemma}

\begin{proof}
    Let $M$ be an integrable $\U$-module. Take any weight vector $v \in M$. Let $i\in \I$. Denote by $S$ the finite-dimensional simple module generated by $v$ over the quantum $\sl_2$ generated by $E_i, F_i$. Write $s=\dim S$.

    If $i\neq \tau i$, by quantum binomial theorem we have 
    \[
    B_i^{(n)}=\frac{B_i^n}{[n]_i!} =  \sum_{a=0}^n  q_i^{\binom{a+1}{2} \langle h_i,\alpha_{\tau i}\rangle-a(n-a)} \vs_i^a K_i^{-a} E_{\tau i}^{(a)} \ F_i^{(n-a)}. 
    \]
    Then $B_i^{(n)}$ annihilates the module $S$ (and hence $B_i^{(n)} v=0$), for $n\ge 2s$. 

    If $i=\tau i$, by \cite[Theorems 2.10, 3.6]{BeW18}, $B_{i,\ov{s-1}}^{(n)}$ annihilates the module $S$ (and hence $B_{i,\ov{s-1}}^{(n)} v=0$), for $n\ge s$. 
    The lemma is proved. 
\end{proof}
In particular, if an element $x\in \Ui$ acts as $0$ on every integrable $\Ui$-module, then $x=0$, since the same statement holds when $\Ui $ is replaced by $\U$ \cite{Lus93}.

\begin{example}
    Let $L(n)$ be the simple $\U(\sl_2)$-module of highest weight $n\in \N$ with highest weight vector $\eta_n$. There exists (see \cite{BW18a}) a unique $\Ui$-module homomorphism $\varphi_n: L(n) \rightarrow L(n-2)$, $\eta_n \mapsto \eta_{n-2}$, for $n\ge 2$. The kernel $K_n =\ker \varphi_n$ is an integrable $\Ui$-module with a basis $\{B_{\ov{n}}^{(n)}\eta_n, B_{\ov{n}}^{(n-1)}\eta_n\}$, though $K_n$ is not a $\U(\sl_2)$-module.
\end{example}

\begin{lemma}\label{lem:AI}
Let $i\in\wI$ such that $i=\tau i$. Let $\ov{\lambda}\in X_\imath$ with a lift $\lambda\in X$. Then the parity of $\langle h_i, \lambda \rangle$ does not depend on the choice of $\lambda$.
\end{lemma}

\begin{proof}
By the definition of $X_\imath$ \eqref{XY}, it suffices to show that $\langle h_i,\beta+\tau\beta \rangle\in 2\Z$,  for any $\beta\in X$. Indeed, we have
$\langle h_i,\beta+\tau\beta \rangle
=\langle h_i,\beta \rangle +\langle \tau h_i,  \beta \rangle
= 2\langle h_i,\beta \rangle\in 2\Z.
$
\end{proof}

\subsection{Formulas for relative braid group symmetries}
\label{sec:braid-formula}
In this subsection, we recall from \cite{WZ23,Z23} the formulas for $\TT'_{i,-1},\TT''_{i,+1}$  for the distinguished parameter $\bvs_\dm$ in the Kac-Moody type. 

Let $i\in \fwItau$. Set $\tau_{i}$ to be the diagram involution associated to the longest element in the subgroup of $W$ generated by $ s_i,s_{\tau i} $. Explicitly, for $j\in \{i,\tau i\}$, we have
\begin{align*}
\tau_{i} (j)=
\begin{cases}
j & \text{ if } c_{i,\tau i}=2,0,
\\
\tau j & \text{ if } c_{i,\tau i}=-1.
\end{cases}
\end{align*}
Note that $\tau_{i} (i)\in \{i,\tau i\}$ and $\tau_i$ commutes with the restriction of $\tau$.
\begin{proposition}\label{prop:rkone}
Let $i\in \fwItau$. Set the parameter $\bvs=\bvs_\dm$.
\begin{itemize}
\item[(1)]\cite[Proposition 4.11 and Proposition 6.2]{WZ23}
For any $x\in \U^{\imath 0} $, $\TT'_{i,-1}(x)=T'_{i ,-1}(x)$ and $\TT''_{i,+1}(x)=T''_{i ,+1}(x)$, i.e., 
\[
\TT'_{i,-1}(K_\mu)=K_{\bs_i\mu},\qquad \TT''_{i,+1}(K_\mu)=K_{\bs_i\mu}, \qquad \forall \mu\in Y^\imath.
\]
\item[(2)]\cite[Theorem 4.14 and Proposition 6.3]{WZ23}
We have 
\begin{align*}
 &\TT'_{i,-1}(B_i)=q_i^{- c_{i,\tau i}/2 +1}   B_{\tau_{ i} \tau i} k_{\tau_{ i} \tau i}^{ -1},
 \qquad
 \TT'_{i,-1}(B_{\tau i})=q_i^{- c_{i,\tau i}/2 +1}   B_{\tau_{ i}  i} k_{\tau_{ i}   i}^{ -1}
 \\
 &\TT''_{i,+1}(B_i)=q_i^{c_{i,\tau i}/2-1}  B_{\tau_{ i} \tau i}  k_{\tau_{ i}  i}^{-1},
 \qquad\quad
 \TT''_{i,+1}(B_{\tau i})=q_i^{c_{i,\tau i}/2-1}  B_{\tau_{ i}  i}  k_{\tau_{ i} \tau i}^{-1}.
\end{align*}
\end{itemize}
\end{proposition}

We next recall from \cite{Z23} the rank two formulas for $\TT'_{i,-1}$ for the distinguished parameter $\bvs_\dm$. For $j\in \I$ with $j\neq i,\tau i$ and $\ov{t}\in \Z_2$, we set  
\begin{align}
    \label{Xjn}
    X_{j,n,\ov{t}}=
    \begin{dcases}
        \dv{j,\ov{t}}{n}, & \text{ if } \tau j=j,
        \\
        \dv{j}{n}, & \text{ if }\tau j\neq j.
    \end{dcases}
\end{align}

To simplify the notation, we sometimes suppress the dependence on $\ov{t}$ in the notation $X_{j,n,\ov{t}}$ and write $X_{j,n}$.

\subsubsection{$\boxed{\text{Split type}~c_{i,\tau i}=2}$} 

In this case, $\vs_{i,\dm}=-q_i^{-2}$.  For $j\ne i \in \I$, we let $b_{i,j;n,m}\in \Ui$, for $m,n\geq 0$, be the elements defined as follows: if $m-nc_{ij}$ is odd, then we set
\begin{align}\notag
b_{i,j;n,m} &:=\sum_{u\ge 0} \sum_{\substack{r+s+2u=m\\ \ov{p}=\ov{r}+1}} (-1)^{r+u} q_i^{-(m+nc_{ij})(r+u)+r-u} \qbinom{\frac{m+n c_{ij}-1}{2}}{u}_{q_i^2}\dv{i,\ov{p }}{s} X_{j,n} \dv{i,\ov{p+n c_{ij}}}{r}
\\\label{splitdiv1}
&\quad +\sum_{u\ge 0} \sum_{\substack{r+s+2u=m\\ \ov{p}=\ov{r}}} (-1)^{r+u} q_i^{-(m+n c_{ij}-2)(r+u)-r-u} \qbinom{\frac{m+n c_{ij}-1}{2}}{u}_{q_i^2}\dv{i,\ov{p }}{s} X_{j,n} \dv{i,\ov{p+n c_{ij}}}{r};
\end{align}
if $m-n c_{ij}$ is even, then we set
\begin{align}\notag
b_{i,j;n,m} &:=\sum_{u\ge 0} \sum_{\substack{r+s+2u=m\\ \ov{p}=\ov{r}+1}} (-1)^{r+u} q_i^{-(m+n c_{ij}-1)(r+u)-u} \qbinom{\frac{m+n c_{ij}}{2}}{u}_{q_i^2}\dv{i,\ov{p }}{s} X_{j,n} \dv{i,\ov{p+n c_{ij}}}{r}
\\\label{splitdiv2}
&\quad +\sum_{u\ge 0} \sum_{\substack{r+s+2u=m\\ \ov{p}=\ov{r}}} (-1)^{r+u} q_i^{-(m+n c_{ij}-1)(r+u)-u} \qbinom{\frac{m+n c_{ij}-2}{2}}{u}_{q_i^2}\dv{i,\ov{p }}{s} X_{j,n} \dv{i,\ov{p+n c_{ij}}}{r}.
\end{align}
The element $b_{i,j;n,m}$ is identified with $\tilde{y}'_{i,j;n,m,\ov{p},\ov{t},1}$ in \cite[(6.3)-(6.4)]{CLW21b} via the central reduction. 

The element $b_{i,j;n,m}$ satisfies the following recursive relation \cite[Theorem~ 6.2]{CLW21b}
\begin{align}
\label{def:splitBij}
\begin{split}
&-q_i^{-(n c_{ij}+2m)}b_{i,j;n,m} B_i +B_i b_{i,j;n,m}\\
 =& \; [m+1]_i  b_{i,j;n,m+1} + [n c_{ij} +m-1]_i q_i^{-2m-n c_{ij}} b_{i,j;n,m-1},
\end{split}
\end{align}
where $b_{i,j;n, 0}=X_{j,n}$ and $b_{i,j;n, m}=0$ if $m<0$. Moreover, we have \cite[Theorem~ 6.3]{CLW21b} 
\begin{align}
    \label{SerreL:split}
    b_{i,j;n, m}=0, \qquad \text{ for } m>-nc_{ij}.
\end{align}

\subsubsection{$\boxed{\text{Diagonal type}~c_{i,\tau i}=0}$} 
\label{subsec:rootvector_diag}

In this case, $\vs_{i,\dm}=\vs_{\tau i,\dm}=-q_i^{-1}$. For $j \in \I$ such that $j\neq i,\tau i$ and for $n,m_1,m_2\ge 0$, we define elements $b_{i,\tau i,j;n,m_1,m_2}$ in $\Ui$ as follows:
 \begin{align}\notag
 &b_{i,\tau i,j;n,m_1,m_2}=\sum_{u=0}^{\min(m_1,m_2)} \sum_{r=0}^{m_1-u}\sum_{s=0}^{m_2-u} (-1)^{r+s}q_i^{r(-n c_{ij}-m_1-u+1)+s(-n c_{\tau i,j}-m_2+u+1)}\times
\\
& \qquad \times q_i^{-u(n c_{ij}+m_1)} \qbinom{-n c_{\tau i,j}-m_2+u}{u}_i B_i^{(m_1-r-u)}B_{\tau i}^{(m_2-s-u)}X_{j,n} B_{\tau i}^{(s)}B_i^{(r)} k_i^u .
\label{diagdiv}
\end{align}
The element $b_{i,\tau i,j;1,m_1,m_2}$ is identified with the element $b_{i,\tau i,j;m_1,m_2}$ in \cite[Definition~  5.6, Proposition 5.11]{Z23} via the central reduction. The elements $b_{i,\tau i,j;n,m_1,m_2}$, for general $n\ge 1$, are new. 

\begin{lemma}
    The elements $b_{i,\tau i,j;n,m_1,m_2}$ satisfy the following recursive relations:
 \begin{align}\notag
&-q_i^{-(n c_{ij}+2m_1)} b_{i,\tau i, j;n,m_1,m_2} B_i + B_i b_{i,\tau i, j;n,m_1,m_2} 
\\\label{def:diagBij1}
=& \; [m_1+1]_i\; b_{i,\tau i, j;n,m_1+1,m_2} 
- q_i^{-(n c_{ij}+2m_1+1)} [-n c_{\tau i,j}-m_2+1]_{i} \; b_{i,\tau i, j;n,m_1,m_2-1}k_i, 
\\\notag
&-q_i^{-(n c_{\tau i,j}+2m_2)}b_{i,\tau i, j;n,m_1,m_2}  B_{\tau i}
+ B_{\tau i} b_{i,\tau i, j;n,m_1,m_2}
\\\label{def:diagBij2}
=& \; [m_2+1]_{i}\; b_{i,\tau i, j;n,m_1,m_2+1} - q_i^{-(n c_{\tau i,j}+2m_2+1)} [-n c_{ i,j}-m_1+1]_{i} \; b_{i,\tau i, j;n,m_1-1,m_2}k_i^{-1},
 \end{align}
 where $b_{i,\tau i, j;n,0,0}=X_{j,n}$ and $b_{i,\tau i, j;n,m_1,m_2}=0$ if either one of $m_1,m_2$ is negative. Moreover, we have 
\begin{align}
    \label{SerreL:diagonal}
    b_{i,\tau i,j;n,m_1,m_2}=0, \qquad \text{ for } m_1>-nc_{ij} \text{ or } m_2>-nc_{\tau i,j}.
\end{align}
\end{lemma}

\begin{proof}
    The recursions follow by similar arguments as in \cite[Proof of Proposition 5.11]{Z23}. The relation \eqref{SerreL:diagonal} follows by induction on $m_1, m_2$ (with higher order $q$-Serre relation as the base case).
\end{proof}

\subsubsection{$\boxed{\text{Quasi-split type}~c_{i,\tau i}=-1}$} 
\label{subsec:rootvector_qs}

In this case, $\vs_{i,\dm}=\vs_{\tau i,\dm}=-q_i^{-1/2}$. For $j \in \I$ such that $j\neq i,\tau i$ and for $n,a,b,c\geq 0$, we define $b_{i,\tau i,j;n,a,b,c}$ in $\Ui$ as follows: 
\begin{align}\notag
&b_{i,\tau i, j;n,a,b,c} =\sum_{u,v\geq 0} \sum_{t=0}^{a-v} \sum_{s=0}^{b-v-u} \sum_{r=0}^{c-u} (-1)^{t+r+s}\qbinom{-n c_{\tau i,j}-b+c+v}{v}_i \qbinom{-n c_{ij}-c+u}{u}_i\times
\\\notag
&\qquad\times  q_i^{t(b-2c-n c_{ij}-a-2v+1)+ v(b-2c-n c_{ij}-a-\frac{v}{2})+ r(-n c_{ij}-c+u+1)+s(c-n c_{\tau i,j}-b+v-2u+1)}\times
\\\label{eq:qsdv1}
&\qquad\times q_i^{u(c-n c_{\tau i,j}-b+v-\frac{u}{2}+3t)}B_i^{(a-v-t)}B_{\tau i}^{(b-v-u-s)}B_i^{(c-u-r)} X_{j,n} B_i^{(r)} B_{\tau i}^{(s)} B_i^{(t)} k_i^{v-u}.
\end{align}
The element $b_{i,\tau i,j;1,a,b,c}$ is identified with the element $b_{i,\tau i,j;a,b,c}$ in \cite[Definition 6.6, Theorem 6.16]{Z23} via the central reduction. The elements $b_{i,\tau i,j;n,a,b,c}$, for $n\ge 1$, are new. 

\begin{lemma}\label{lem:qsrecur}
The elements $b_{i,\tau i,j;n,a,b,c}$ satisfy the following recursive relations:
\begin{align}\notag
& B_i  b_{i,\tau i, j;n,a,b,c} -q_i^{b-2a-2c-n c_{ij}} b_{i,\tau i, j;n,a,b,c} B_i
\\\label{def:qsqsBij1}
=& \; [a+1]_i b_{i,\tau i, j;n,a+1,b,c}-q_i^{b-2a-2c-nc_{ij}-3/2}[-b+c+1-n c_{\tau i,j}]_i b_{i,\tau i, j;n,a,b-1,c} k_i,
\end{align}
and
\begin{align}\notag
& B_{\tau i}  b_{i,\tau i, j;n,a,b,c}-q_i^{-2b+a+c-n c_{\tau i,j}} b_{i,\tau i, j;n,a,b,c}B_{\tau i}
\\\notag
=& \;[b-a+1 ]_i b_{i,\tau i, j;n,a,b+1,c} + [c+1]_i b_{i,\tau i, j;n,a-1,b+1,c+1}
\\\notag
&-q_i^{-2b+c+a-n c_{\tau i,j}-3/2} [-nc_{ij}-a+b-2c+1]_i b_{i,\tau i, j;n,a-1,b,c} k_i^{-1}
\\\label{def:qsqsBij2}
&+q_i^{-2b+c+a-n c_{\tau i,j}-3/2}[ nc_{ij}+c-1]_i b_{i,\tau i, j;n,a,b,c-1}k_i^{-1},
\end{align}
 where $b_{i,\tau i, j;n,0,0,0}=X_{j,n}$ and $b_{i,\tau i, j;n,a,b,c}=0$ if any one of $a,b,c$ is negative.
 \end{lemma}

\begin{proof}
    Follows by similar arguments as in \cite[Proof of Theorem 6.16]{Z23}.
\end{proof}

\begin{lemma}\label{lem:qsvanish}
We have $$b_{i,\tau i,j;n,a,b,c}=0, 
$$
if either $\textnormal{(i)}~c>-nc_{ij}$, or $\textnormal{(ii)}~b-c>-nc_{\tau i,j}$, or $\textnormal{(iii)}~a+c>-nc_{ij}-nc_{\tau i,j}$.
\end{lemma}

\begin{proof}
We formulate the proof for $n=1$; the proof for general $n$ follows by the same strategy. Let $\ad$ be the adjoint action in a Hopf algebra $\U$, which is an algebra homomorphism \cite{Ja95}. It suffices to show that $\ad(E_i^{(a)}E_{\tau i}^{(b)}E_i^{(c)})E_j=0$ if $a,b,c$ satisfy the condition (i) or (ii) or (iii). Indeed, this implies that elements $y_{i,\tau i,j;a,b,c},x_{i,\tau i,j;a,b,c}$ introduced in \cite[Definition 6.1]{Z23} are equal to $0$ and then the desired statement follows from the intertwining relations in \cite[Proposition 6.10--6.11]{Z23}.

Let $x,y\in \U$ be two homogeneous elements with weights denoted $|x|,|y|$. Recall from \cite{Ja95} that $\ad(E_i)x=[E_i,x]_{q_i^{\langle h_i,|x| \rangle}}$, $\ad(E_i)y=[E_i,y]_{q_i^{\langle h_i,|y| \rangle}}$, and the Leibniz rule 
\[
\ad(E_i)(xy)= \ad(E_i)(x) \cdot y +q_i^{\langle h_i,|x| \rangle} x\cdot \ad(E_i)(y).
\]

By \cite{Ja95}, the $q$-Serre relation \eqref{eq:serre1} can be rewritten as $\ad(E_i^{(1-c_{ij})})E_j=0$. Hence 
\begin{enumerate}
    \item[(i)] 
    $\ad(E_i^{(c)})E_j=0$ if $c>-c_{ij}$. 
\end{enumerate}
Since $c_{i,\tau i}=-1$, we have $\ad(E_{\tau i}^{(2)})E_i=0$. Then, as $\ad(E_i^{(c)})E_j$ is a linear combination of $E_i^xE_jE_i^{c-x}$, for $0\le x \le c$, by the Leibniz rule, $\ad(E_{\tau i}^{(b)})(E_i^xE_jE_i^{c-x})=0$ if $b>c-c_{\tau i,j}$, and hence 
\begin{enumerate}
    \item[(ii)]
$\ad(E_{\tau i}^{(b)}E_i^{(c)})E_j=0$ if $b-c>-c_{\tau i,j}$.
\end{enumerate}

It remains to show that (iii) $\ad(E_i^{(a)}E_{\tau i}^{(b)}E_i^{(c)})E_j=0$ if $a+c>-c_{ij}-c_{\tau i,j}$. We assume that $a+c> b$; otherwise, $b\ge a+c>-c_{ij}-c_{\tau i,j}$ implies that either $c>-c_{ij}$ or $b-c>-c_{\tau i,j}$, which go back to the previous two cases.

Let $a+c>b$ and $a+c>-c_{ij}-c_{\tau i,j}$. By \cite[Lemma 42.1.2(d)]{Lus93}, $E_i^{(a)}E_{\tau i}^{(b)}E_i^{(c)}$ is a linear combination of $E_{\tau i}^{(c-k)}E_i^{(a+c)}E_{\tau i}^{(a-l)}$ for some $k,l\in \N$ such that $k+l=a+c-b$. The condition $a+c>-c_{ij}-c_{\tau i,j}$ forces that either $a-l>-c_{\tau i,j}$ or $c+l>-c_{ij}.$ 
In either cases, we have $\ad(E_{\tau i}^{(c-k)}E_i^{(a+c)}E_{\tau i}^{(a-l)})E_j=0$ by (ii) established above (applied to different notation here). This establishes Case (iii).
\end{proof}


Now we are ready to summarize the rank two formulas for $\TT_{i,-1}'$ in all 3 cases.
In case $c_{i,\tau i}=2$, it is due to \cite[Theorem 4.11]{Z23}  (which proved  \cite[Conjecture 6.5]{CLW21b}).
In case $c_{i,\tau i}=0$, it was due to \cite[Theorem 5.17]{Z23} (which proved \cite[Conjecture 3.7]{CLW23}).
In case $c_{i,\tau i}=-1$, it is due to  \cite[Theorem 6.14]{Z23}. The formulas {\em loc. cit.} were formulated in the universal iquantum group setting, and we have reformulated below for $\Ui$ for the distinguished parameter $\bvs_\dm$ via central reduction.

\begin{proposition}
\label{prop:TiBjZ}
Set the parameter $\bvs =\bvs_\dm$. Then, for $i\in \Ifin,j \in \I$ such that $j \neq i, \tau i$, we have
    \begin{align*}
    \TT_{i,-1}'(B_j)=
        \begin{dcases}
            b_{i,j;1,-c_{ij}}, & \text{ if } c_{i,\tau i}=2, 
            \\
        b_{i,\tau i,j;1,-c_{ij},-c_{\tau i,j}}, & \text{ if } c_{i,\tau i}=0, 
            \\
        b_{i,\tau i,j;1,-c_{\tau i,j},-c_{ij}-c_{\tau i,j},-c_{i,j}}, & \text{ if } c_{i,\tau i}=-1.
        \end{dcases}
    \end{align*}
\end{proposition}
In this paper, we shall generalize these formulas to $\TT_{i,-1}'(X_{j,n})$, for $n\ge 1$; see \eqref{Xjn}.

\subsection{Bar involution on $\Ui$ and quasi $K$-matrix}
\label{subsec:ibar}

When the parameter $\bvs$ satisfies the condition $\vs_{\tau i}=q_i^{-c_{i,\tau i} } \ov{\vs_i}$, there exists an anti-linear bar involution $\psi^\imath$ on $\Ui$ such that (cf. \cite{CLW21a})
\begin{align}
\psi^\imath(B_i)=B_i,\quad K_\mu=K_{-\mu},\quad \forall i\in \I,\mu\in Y^\imath.
\end{align}
Moreover, we have the following relation (cf. \cite[Theorem 3.8]{BW18a}, \cite{BW18b}, \cite{BK19})
\begin{align}\label{def:psi-fX}
\psi^\imath(x) \fX =\fX \psi(x),\qquad \forall x\in \Ui.
\end{align}

Let $M$ be an integrable $\U$-module equipped with a bar involution $\psi$ whose weights are bounded above. One can also define an anti-linear map $\psi^\imath$ on $M$ by $\psi^\imath (v)=\fX_i \psi (v),$ for $v\in M$. Then $\psi^\imath(xv)=\psi^\imath(x)\psi^\imath (v)$ for any $x\in \Ui,v\in M$; cf. \cite[Proposition 5.1]{BW18b}.

We define a balanced parameter $\bvs_\star=(\vs_{i,\star})_{i\in\wI}\in(\F^\times)^{\wI}$ such that  
\begin{align}\label{def:bvs-star}
\vs_{i,\star}=q_i^{-c_{i,\tau i}/2}.
\end{align}

Then $\bvs_\star$ satisfies the condition $\vs_{\tau i}=q_i^{-c_{i,\tau i} } \ov{\vs_i}$. Thus, there exist anti-linear involutions $\psi^\imath$ on $\Ui$ and $\U$-modules as discussed above. 

 

 \section{Braid group operators of split type on integrable $\Ui$-modules}
 \label{sec:split}
 
In this section, we fix $i\in \wI$ such that $i=\tau i $. We present explicit formulas for linear operators  $\TT'_{i,-1},\TT''_{i,+1}$ (which are shown to be mutually inverse) on any integrable $\Ui$-modules and show that they are compatible with corresponding symmetries on $\Ui$ given in \cite{WZ23, Z23}. We further show that $\TT'_{i,-1},\TT''_{i,+1}$ specialize to the corresponding ones {\em loc. cit.} on integrable $\U$-modules by relating to Lusztig symmetries. 

\subsection{Rank one formulas of split type}

Let $(\U,\Ui_\bvs)$ be a quasi-slit quantum symmetric pair with a general parameter $\bvs=(\vs_j)_{j\in \I}$. Let $M\in \mathcal{C}$ be an integrable $\Ui$-module.

For $i\in \wI$ with $i=\tau i$, we define linear operators $\TT'_{i,-1},\TT''_{i,+1}$ on $M$ by letting  
\begin{align}
\label{eq:ibraid1}
\TT'_{i,-1}(v) & =\sum_{k\geq 0:\:\ov{k}=\ov{\la}_i}  (-1)^{-k/2}(q_i^2\vs_i)^{-k/2} \dvi{\ov{k}}{k}v, 
\\
\label{eq:ibraid1'}
\TT''_{i,+1} (v) & =\sum_{k\geq 0:\:\ov{k}=\ov{\la}_i}  (-1)^{k/2}\vs_i^{-k/2} \dvi{\ov{k}}{k}v,
\end{align}
for any iweight vector $v\in M_{\ov{\lambda}}$, where $\la_i = \langle h_i, \lambda \rangle, \ov{\la}_i=\ov{\langle h_i, \lambda \rangle}\in \Z_2$ (which is well defined by Lemma~\ref{lem:AI}). Only finitely many summands on the right-hand sides of \eqref{eq:ibraid1}--\eqref{eq:ibraid1'} are nonzero since $M$ is integrable. 

We also introduce the following elements in some completion of $\Ui$, for $\ov{p} \in \Z_2$:
\begin{align}
\label{def:fbvs} 
    f_{\ov{p},\bvs}(B_i):=& \; \sum_{\substack{k\geq 0\\\ov{k}=\ov{p}}}  (-1)^{-k/2} (q_i^2\vs_i)^{-k/2} \dvi{\ov{k}}{k},
    \\
    \label{def:ovfbvs}
    \tilde{f}_{\ov{p},\bvs}(B_i):= &\sum_{\substack{k\geq 0\\\ov{k}=\ov{p}}}  (-1)^{k/2}\vs_i^{-k/2} \dvi{\ov{k}}{k}.
\end{align} 
They act as well-defined linear operators on any integrable $\Ui$-modules. By definition \eqref{eq:ibraid1}--\eqref{eq:ibraid1'}, for any $v \in M_{\ov{\la}}$ we have \begin{align*}
     \TT'_{i,-1}v=f_{\ov{\la}_i,\bvs} (B_i) v, 
     \qquad
     \TT''_{i,+1}v=\tilde{f}_{\ov{\la}_i,\bvs} (B_i) v.
 \end{align*}

\subsection{Compatibility with braid group symmetries on $\Ui$}

 \begin{lemma}
 \label{lem:musplit}
 Let  $i\in \I$ such that $i=\tau i$. For any $\mu\in Y^\imath$, we have $\langle \mu,\alpha_i\rangle=0$ and $\bs_i \mu=\mu$.
 \end{lemma}

\begin{proof}
Since $\mu\in Y^\imath$, $\tau \mu=-\mu$. Then we have $\langle \mu,\alpha_i \rangle = \langle \tau\mu,\tau\alpha_i \rangle = -\langle \mu,\alpha_i \rangle$ and thus $\langle \mu,\alpha_i \rangle=0$. Therefore, $\bs_i\mu=\mu-\langle \mu,\alpha_i \rangle h_i =\mu$.
\end{proof}

Recall from \cref{thm:braid-iQG} the braid group symmetries $\TT'_{i,-1}$ and $\TT''_{i,+1}$ on $\Ui$. The proof of the next theorem will be reduced to verifying a rank two formula, which is the $n=1$ special case of  Proposition~\ref{prop:BBij0} below.

\begin{theorem}
\label{thm:split}
Let $i\in \I$ such that $i=\tau i$. For any $x\in \Ui,v\in M$, we have
\begin{align}\label{eq:split}
 \TT'_{i,-1}(x) \TT'_{i,-1} (v) =  \TT'_{i,-1} (x v).
\end{align}
\end{theorem}

\begin{proof}

Suppose that the identity \eqref{eq:split} holds for $x,y\in \Ui$. Since $\TT'_{i,-1}$ is an algebra automorphism on $\Ui$, we have
\[
\TT'_{i,-1}(x y v)=\TT'_{i,-1}(x) \TT'_{i,-1}(y v)=\TT'_{i,-1}(x) \TT'_{i,-1}(y) \TT'_{i,-1}(v)=\TT'_{i,-1}(xy) \TT'_{i,-1}(v).
\] 
That is, the identity \eqref{eq:split} also holds for $xy$. Thus it suffices to  verify \eqref{eq:split} as $x$ runs over a generating set of $\Ui$ and any $v\in M_{\ov{\lambda}}$, for $\ov{\lambda}\in X_\imath$; set $\ov{p}=\ov{\langle h_i,\lambda\rangle}\in \Z_2$. 

We first prove \eqref{eq:split} for $x=K_\mu \; (\mu\in Y^\imath)$. By Lemma~\ref{lem:musplit} and Proposition~\ref{prop:rkone}, we have $\TT'_{i,\bvs_\dm,-1}(K_\mu)=K_{\mu}$. For any parameter $\bvs$, using Lemma~\ref{lem:iso-parameter} and \cref{thm:braid-iQG}(b), $\TT'_{i,\bvs ,-1}(K_\mu)=K_\mu$. It remains to show that $\TT'_{i,-1}(K_\mu v)=K_\mu \TT'_{i,-1}(v) $. Note that the formula \eqref{eq:ibraid1} is a summation of $\dvi{\ov{k}}{k}$, which are polynomials in $B_i$. By Lemma~\ref{lem:musplit}, we have $K_\mu B_i = B_i K_\mu$. Therefore, $K_\mu$ commutes with $\TT_{i,-1}'$ when acting on $M$, as desired.

We next prove \eqref{eq:split} for $x=B_i$. Recall from Proposition~\ref{prop:rkone}(2) that $\TT'_{i,\bvs_\dm,-1}(B_i)=B_i$. By Lemma~\ref{lem:iso-parameter} and \cref{thm:braid-iQG}(2), we have $\TT'_{i,\bvs,-1}(B_i)=B_i$ for a general parameter $\bvs$. On the other hand, note that $B_iv\in M_{\ov{\lambda-\alpha_i}}$ and $\ov{\langle h_i,\lambda\rangle}=\ov{\langle h_i,\lambda-\alpha_i\rangle}$. Then, by \eqref{eq:ibraid1}, the actions of $\TT'_{i,-1}$ on both $v,B_iv$ are given by a sum of idivided powers, which are polynomials of $B_i$. Hence, we have 
$\TT'_{i,-1}(B_i) f_{\ov{p},\bvs} (B_i) =B_i f_{\ov{p},\bvs} (B_i) =f_{\ov{p},\bvs} (B_i) B_i$. 

Finally, we prove \eqref{eq:split} for $x=B_j\, (j\neq i)$.  By \eqref{eq:ibraid1}, we have $\TT'_{i,-1}v=f_{\ov{p},\bvs} (B_i) v$ and $\TT'_{i,-1}(B_j v)=f_{\ov{p}+\ov{c_{ij}},\bvs} (B_i) B_j v$. Thus, it suffices to verify that 
\begin{align}\label{eq:fsplit}
\TT'_{i,\bvs,-1}(B_j) f_{\ov{p},\bvs} (B_i) =f_{\ov{p+c_{ij}},\bvs} (B_i) B_j.
\end{align}
By Lemma~\ref{lem:div-par} and the definition of $f_{\ov{p},\bvs}(B_i)$, we have $f_{\ov{p},\bvs}(B_i)=\phi_{\bvs_\dm,\bvs} f_{\ov{p},\bvs_\dm}(B_i)$ for any $\ov{p}\in \Z_2$. By \cref{thm:braid-iQG}(2), we have $\TT'_{i,\bvs,-1}\phi_{\bvs_\dm,\bvs}(B_j)=\phi_{\bvs_\dm,\bvs} \TT'_{i,\bvs_\dm,-1}(B_j)$. Therefore, it suffices to verify \eqref{eq:fsplit} for the distinguished parameter $\bvs_\dm$. 

Recall from \S\ref{sec:braid-formula} that $\TT'_{i,\bvs_\dm, -1}(B_j)=b_{i,j;1,-c_{ij}}$ for $j\in \wI,j\neq i$. Note that $b_{i,j;1,0}=B_j$. Thus, the desired identity \eqref{eq:fsplit} for $\bvs=\bvs_\dm$ can be reformulated as
\begin{align}\label{eq:fsplit2}
 b_{i,j;1,-c_{ij}} f_{\ov{p},\bvs_\dm} (B_i) =  f_{\ov{p}+\ov{c_{ij}},\bvs_\dm}(B_i) b_{i,j;1,0}.
\end{align}
The proof of the identity \eqref{eq:fsplit2} is given as a special case (when $n=1$) of Proposition~\ref{prop:BBij0} in the next subsection.

Combining all the above cases, we have proved the theorem.
\end{proof}

Our immediate goal is to derive the counterpart of Theorem~\ref{thm:split} for $\TT''_{i, +1}$.

 \begin{lemma}\label{lem:sbar}
 Set the parameter $\bvs =\bvs_\star$ in \eqref{def:bvs-star}. For any $\ov{p}\in\Z_2,j\in \I$, we have 
 \begin{align*}
 \psi^\imath\big(f_{\ov{p},\bvs_\star}(B_i)\big)&=(-1)^p \tilde{f}_{\ov{p},\bvs_\star}(B_i),
 \\
 \psi^\imath\big(\TT'_{i, -1}(B_j)\big)&=(-1)^{c_{ij}} \TT''_{i, +1}(B_j).
 \end{align*}
 \end{lemma}

 \begin{proof}
 Recall that $\vs_{\star,i}=q_i^{-1}$ and the bar involution $\psi^\imath$ on $\Ui_{\bvs_\star}$ exists; see \S\ref{subsec:ibar}. By definition \eqref{def:idv}, $\dvi{\ov{k},\bvs_\star}{k}$ are invariant under $\psi^\imath$. Then the first identity follows by the formulas \eqref{def:fbvs}--\eqref{def:ovfbvs}.

 For $j=i$, recall from Proposition~\ref{prop:rkone} that $\TT'_{i, -1}(B_i)=\TT''_{i, +1}(B_i)=B_i$, and hence the second identity in the lemma follows. For $j\neq i$, the formulas of $\TT'_{i, -1}(B_j),\TT''_{i, +1}(B_j)$ are obtained by applying the central reduction and $\phi_{\bvs_\dm, \bvs_\star}$ to the formulas  conjectured in \cite{CLW21b} and established partially in \cite{LW22} and fully in \cite[Theorem 3.7(i)]{Z23}. Explicitly, we have
 \begin{align*}
\TT'_{i, -1}(B_j)=& \;\sum_{r+s=-c_{ij}} (-1)^{r+\frac{r+s}{2}} q_i^{ \frac{r-s}{2}} \dv{i,\ov{p-c_{ij}}}{s} B_j \dv{i,\ov{p}}{r}
\\
&+\sum_{u\geq 1}\sum_{\substack{r+s+2u=-c_{ij}\\ \ov{r}=\ov{p}}} (-1)^{r+\frac{r+s}{2}} q_i^{ \frac{r-s}{2}} \dv{i,\ov{p-c_{ij}}}{s} B_j \dv{i,\ov{p}}{r},
\\ 
\TT''_{i, +1}(B_j)=& \; \sum_{r+s=-c_{ij}} (-1)^{s+\frac{r+s}{2}} q_i^{\frac{s-r}{2}} \dv{i,\ov{p-c_{ij}}}{s} B_j \dv{i,\ov{p}}{r}
\\
&+\sum_{u\geq 1}\sum_{\substack{r+s+2u=-c_{ij}\\ \ov{r}=\ov{p}}} (-1)^{s+\frac{r+s}{2}} q_i^{\frac{s-r}{2}} \dv{i,\ov{p-c_{ij}}}{s} B_j \dv{i,\ov{p}}{r}.
 \end{align*}
 
The second identity in the lemma follows from these two formulas.
 \end{proof}

\begin{proposition} \label{prop:split2}
Let $i=\tau i$. For any $x\in \Ui,v\in M$, we have
\begin{align}\label{eq:split'}
 \TT''_{i,+1}(x) \TT''_{i,+1} v =  \TT''_{i,+1} (x v).
\end{align}
\end{proposition}

\begin{proof}
By a similar argument as in the proof of Theorem~\ref{thm:split}, it suffices to verify \eqref{eq:split'} when $x$ runs over a generating set for $\Ui$, for $\bvs=\bvs_\star$ and $v\in M_{\ov{\lambda}}$ for $\ov{\lambda}\in X_\imath$; set $\ov{p}=\ov{\langle h_i,\lambda\rangle}\in \Z_2$. 

When $x$ is one of $K_\mu,B_i$ for $\mu\in Y^\imath,i\in \I$, the proof of \eqref{eq:split'} is essentially the same as the proof of \eqref{eq:split} in Theorem~\ref{thm:split}.

It remains to prove \eqref{eq:split'} for $x=B_j$ with $j\neq i$. The bar involution $\psi^\imath$ exists for $\Ui$ with parameter $\bvs_\star$; cf. \S\ref{subsec:ibar}. By Lemma~\ref{lem:sbar}, we have $\psi^\imath\big( f_{\ov{p},\bvs_\star} (B_i) \big)=\tilde{f}_{\ov{p},\bvs_\star} (B_i) $. Applying $\psi^\imath$ to \eqref{eq:fsplit}, we obtain the following identity
\begin{align}\label{eq:fsplit'}
\TT''_{i,\bvs_\star,+1}(B_j) \tilde{f}_{\ov{p},\bvs_\star} (B_i) =\tilde{f}_{\ov{p+c_{ij}},\bvs_\star} (B_i) B_j.
\end{align}
This proves the proposition.
\end{proof}

\subsection{Identities in split rank two with parameter $\bvs_\dm$}
\label{sec:splitrank2}

Set the parameter $\bvs=\bvs_\dm$, and thus, $\vs_i=-q_i^{-2}$ thanks to $i=\tau i$. 

We shall use the standard roof and floor notation for integers $\lceil a \rceil$ 
and $\lfloor a \rfloor$, for $a\in \mathbb R$. 
Recall the element $b_{i,j;n,m}$ from \eqref{splitdiv1}-\eqref{splitdiv2} and \eqref{def:splitBij}, for $n,m \ge 0$. 

\begin{proposition}
\label{prop:BBij2}
Let $\alpha =-c_{ij}$, for $i\neq j\in \wI$. The following formulas hold, for $n, k\ge 0$: 
\begin{align}\notag
 &b_{i,j;n,n\alpha}\dvi{\ov{k}}{k} 
 =\sum_{x=0}^{n\alpha} q_i^{(k-x)(n\alpha-x) } \times
 \\
 &\qquad\quad \times \bigg(\sum_{y=0}^{\lceil \frac{n\alpha-x}{2} \rceil} (-1)^y q_i^{2y(\lceil \frac{n\alpha-x}{2} \rceil-1-n\alpha+x)}\qbinom{\lceil \frac{n\alpha-x}{2} \rceil}{y}_{q_i^2} \dvi{\ov{k+n \alpha}}{k-x-2y}\bigg) b_{i,j;n,n\alpha-x},
 \label{Com:qrootB2}
 \\\notag
&b_{i,j;n,n\alpha}\dvi{\ov{k+1}}{k} 
=\sum_{x=0}^{n\alpha} q_i^{(k-x)(n\alpha-x) } \times
\\
&\qquad\quad \times\bigg(\sum_{y=0}^{\lfloor \frac{n\alpha-x}{2} \rfloor} (-1)^y q_i^{2y(\lfloor \frac{n\alpha-x}{2} \rfloor -n\alpha+x)}\qbinom{\lfloor \frac{n\alpha-x}{2} \rfloor}{y}_{q_i^2} \dvi{\ov{k+1+n \alpha}}{k-x-2y}\bigg) b_{i,j;n,n\alpha-x}.
 \label{Com:qrootB3}
\end{align}
\end{proposition} 
 
The long and difficult proof for Proposition~\ref{prop:BBij2} is postponed to Appendix \ref{sec:proof}. 

By definition \eqref{def:fbvs}, we have
\begin{align}
    \label{f_dm}
    f_{\ov{0},\bvs_\dm} (B_i)=\sum_{k=0}^\infty \dvi{\ev}{2k},
    \qquad
    f_{\ov{1},\bvs_\dm} (B_i)=\sum_{k=0}^\infty \dvi{\odd}{2k+1}.
\end{align}
The formula in the next proposition for $n=1$ completes the proof of Theorem \ref{thm:split}, and its proof relies essentially on \cref{prop:BBij2}.

\begin{proposition}\label{prop:BBij0}
Let $\alpha =-c_{ij}$, for $i\neq j\in \wI$. Then for any $n\ge 1$, we have 
\begin{align*}
 b_{i,j;n,n\alpha} f_{\ov{p},\bvs_\dm} (B_i) =  f_{\ov{p+n\alpha},\bvs_\dm}(B_i) b_{i,j;n,0}.
\end{align*}
\end{proposition}

\begin{proof}
By Proposition~\ref{prop:BBij2}, we have 
\begin{align}
 \label{Com:qrootf0}
 b_{i,j;n,n\alpha} f_{\ov{p},\bvs_\dm} (B_i) 
 &= \sum_{x\ge 0} \xi_{n,n\alpha, x,\ov{p}} b_{i,j;n,n\alpha-x},
\end{align}
where
\begin{align}
 \label{eq:Aax}
\xi_{n,n\alpha, x,\ov{p}}  
& = \sum_{k:\:\ov{k}=\ov{p}}  \sum_{y=0}^{\lceil \frac{n\alpha-x}{2} \rceil} (-1)^y q_i^{(k-x)(n\alpha-x) +2y(\lceil \frac{n\alpha-x}{2} \rceil-1-n\alpha+x)} \qbinom{\lceil \frac{n\alpha-x}{2} \rceil}{y}_{q_i^2} \dvi{\ov{p+n \alpha}}{k-x-2y}. 
\end{align}
Setting $d=k -x -2y$ above, we rewrite the $q_i$-power in \eqref{eq:Aax} as 
\[
(k-x)(n\alpha-x) +2y(\lceil \frac{n\alpha-x}{2} \rceil -1-n\alpha+x)
=  d(n\alpha -x) + 2y \big( \lceil \frac{n\alpha-x}{2} \rceil -1 \big).
\]
The coefficient of $\dvi{\ov{p+n \alpha}}{d}$ in $\xi_{n,n\alpha, x,\ov{p}}$ for any $d\ge 0$, denoted by $[\dvi{\ov{p+n \alpha}}{d}] \xi_{n,n\alpha, x,\ov{p}}$, is given by
\begin{align*}
[\dvi{\ov{p+n \alpha}}{d}] \xi_{n,n\alpha, x,\ov{p}}  
&= \delta_{\ov{d},\ov{p-x}}\sum_{y=0}^{\lceil \frac{n\alpha-x}{2} \rceil} 
(-1)^y q_i^{d(n\alpha -x) + 2y \big( \lceil \frac{n\alpha-x}{2} \rceil -1 \big)} \qbinom{\lceil \frac{n\alpha-x}{2} \rceil}{y}_{q_i^2} 
\\
&=\delta_{\ov{d},\ov{p-x}} q_i^{d(n\alpha -x)} \sum_{y=0}^{\lceil \frac{n\alpha-x}{2} \rceil} 
(-1)^y (q_i^{-2})^{y \big(1- \lceil \frac{n\alpha-x}{2} \rceil  \big)} \qbinom{\lceil \frac{n\alpha-x}{2} \rceil}{y}_{q_i^{-2}}
\\
&=
\begin{cases}
0 & \text{ if }x < n\alpha
\\
\delta_{\ov{d},\ov{p-n\alpha}} & \text{ if }x = n\alpha,
\end{cases}
\end{align*}
where the last equality follows by a standard quantum binomial identity; cf. \cite[1.3.4]{Lus93} with $v=q_i^{-2}$. 

Summarizing, we can rewrite \eqref{eq:Aax} as 
\begin{align*}
\xi_{n,n\alpha, x,\ov{p}} =  
     \begin{cases}
0 & \text{ if }x < n\alpha,
\\
\sum_{d:\ov{d}=\ov{p-n\alpha}} \dvi{\ov{p+n \alpha}}{d} & \text{ if } x = n\alpha,
\end{cases}
\end{align*}
and hence rewrite \eqref{Com:qrootf0} as 

\[
b_{i,j;n,n\alpha} f_{\ov{p},\bvs_\dm} (B_i) = \sum\limits_{d:\:\ov{d}=\ov{p+n \alpha}} \dvi{\ov{p+n \alpha}}{d} \cdot b_{i,j;n,0} \overset{\eqref{f_dm}}{=} f_{\ov{p+n \alpha},\bvs_\dm}(B_i) b_{i,j;n,0}, 
\]
proving the proposition.
\end{proof}

%
%

\subsection{Canonical basis and icanonical basis in split rank one}

In this subsection, we restrict ourselves to the rank one quantum group $\U$ and the split rank one iquantum group $\Ui=\Q(q) [B]$ with $\vs =q^{-1}$, i.e., 
\begin{align*}
B=F+q^{-1} E K^{-1}.
\end{align*}
We then omit the first subindex $i$ for notations $E_i, F_i, T'_{i,e}, T''_{i,e}$ in this subsection. 

 Let $L(n)$ be the irreducible highest weight module over $\U=\U(\sl_2)$ with the highest weight $n\varpi$ and a highest weight vector $\eta =\eta_n$. The canonical basis for the $\U$-module $L(n)$ is  
$\mathcal{CB}:=\{F^{(m)} \eta\mid 0\le m\le n \}.$
 
 Set $v_k=F^{(k)}\eta$ if $0\leq k\leq n$ and $v_k=0$ otherwise. We have
 \begin{align}
 F v_k  =
 \begin{cases}
 [k+1]v_{k+1},& \text{ if } k<n\\
 0,& \text{ if } k=n,
 \end{cases}
 \qquad
 E v_k=
 \begin{cases}
 0,& \text{ if } k=0\\
 [n+1-k]v_{k-1},& \text{ if } k>0.
 \end{cases}
 \end{align}
 By induction, we have that, for $m\geq 0$, 
 \begin{align}
 F^{(m)}v_k=\qbinom{m+k}{m}v_{k+m},\qquad  E^{(m)} v_k =\qbinom{n+m-k}{m} v_{k-m}.
 \end{align}
 
Recall the idivided powers from \eqref{def:idv}, and here we will drop the indices $i$ and $\bvs$ to write $B^{(m)}_{\ov{a}}$, for $\ov{a}\in \Z_2$. There exists a bar involution $\psi^\imath$ on $\Ui$ which fixes $B$. 
The icanonical basis for $L(n)$ is (see \cite[Theorems 2.10, 3.6]{BeW18}, which proves a conjecture in \cite{BW18a})
 \[
\imath\mathcal{CB}:=\{B_{\ov{n}}^{(m)}\eta \mid 0\le m\le n\}.
 \]

We recall from \cite{BeW18} the formulas for the icanonical basis elements as a linear combination of $\mathcal{CB}$. A notation $\cbinom{m-\lambda-c}{c}$ was used {\em loc. cit.}, which is replaced here by a standard $q$-binomial notation thanks to
$\cbinom{m-\lambda-c}{c} = q^{2(m-\lambda)c} \qbinom{c+\lambda-m}{c}_{q^2}.$

 \begin{lemma} \cite[(2.16)-(2.17),(3.8)-(3.9)]{BeW18}
   \label{lem:BeW}
The following identities hold in $L(n)$, for $n,\ell \in \N$:
\begin{align}
\dv{\ov{n}}{n-2\ell}  \eta
& = \sum_{c=0}^{\lfloor \frac{n}2 \rfloor -\ell} q^{-2c^2-(2\ell -1)c}  \qbinom{\ell +c}{c}_{q^2}
F^{(n-2\ell-2c)}  \eta,
 \label{BviaF} \\
\dv{\ov{n}}{n-1-2\ell}   \eta
& = \sum_{c=0}^{\lfloor \frac{n-1}2 \rfloor -\ell} q^{-2c^2-(2\ell +1)c}  \qbinom{\ell +c}{c}_{q^2}
F^{(n-1-2\ell-2c)}  \eta.
\label{BviaF1}
\end{align}
%
\end{lemma}

It is remarkable that explicit inverse formulas to the formulas \eqref{BviaF}--\eqref{BviaF1} exist.
\begin{proposition}\label{prop:inverse}
The following identities hold in $L(n)$, for $n, \ell \in \N$:
\begin{align}
F^{(n-2\ell)}  \eta
& = \sum_{c=0}^{\lfloor \frac{n}2 \rfloor -\ell} (-1)^c q^{-(2\ell +1)c}  \qbinom{\ell +c}{c}_{q^2}
\dv{\ov{n}}{n-2\ell-2c}  \eta,
\label{FviaB}
  \\
F^{(n-1-2\ell)}  \eta
& = \sum_{c=0}^{\lfloor \frac{n-1}2 \rfloor -\ell} (-1)^c q^{-(2\ell +3)c}  \qbinom{\ell +c}{c}_{q^2}
\dv{\ov{n}}{n-1-2\ell-2c}  \eta.
\label{FviaB1}
\end{align}
\end{proposition}

\begin{proof}
The proofs of these 2 formulas are entirely similar, and we shall only provide the details for \eqref{FviaB}. Thanks to \cref{lem:BeW}, the transition matrix from $\mathcal{CB}$ to $\imath\mathcal{CB}$ is uni-triangular. 

Recall the following standard quantum binomial identity (cf. \cite[1.3.1(e)]{Lus93} and we have replaced $v$ therein by $q^2$): for $k\in \N$ and $x,y\in \Z$,
\begin{align}  \label{eq:1.3.1}
\sum_{a+c=k}
{(q^2)}^{a x -cy}
\qbinom{x}{c}_{q^2} \qbinom{y}{a}_{q^2}
= \qbinom{x+y}{k}_{q^2}.
\end{align}

We show that \eqref{FviaB} provides the inverse formula for \eqref{BviaF} by plugging the formula \eqref{FviaB} into RHS \eqref{BviaF} and simplifying:
\begin{align*}
\text{RHS } \eqref{BviaF}
&= \sum_{c=0}^{\lfloor \frac{n}2 \rfloor -\ell} \sum_{a=0}^{\lfloor \frac{n}2 \rfloor -\ell-c}
q^{-2c^2-(2\ell-1)c}  \qbinom{\ell+c}{c}_{q^2}
\\
&\qquad\qquad\qquad \times (-1)^a q^{-(2\ell+2c+1)a}  \qbinom{\ell+c+a}{a}_{q^2}
\dv{\ev}{n-2\ell-2c-2a}  \eta
\\
& \stackrel{(*)}{=} \sum_{k=0}^{\lfloor \frac{n}2 \rfloor -\ell}  (-1)^{k} q^k   \\
&\qquad
\times \sum_{a+c=k}
{(q^2)}^{a (-\ell-1) -c(\ell+k)}  
\qbinom{-\ell-1}{c}_{q^2} \qbinom{\ell+k}{a}_{q^2}
\dv{\ev}{n-2\ell-2k}  \eta
\\
& \stackrel{(**)}{=} \sum_{k=0}^{\lfloor \frac{n}2 \rfloor -\ell}  (-1)^{k} q^k  \qbinom{k-1}{k}_{q^2} \dv{\ev}{n-2\ell-2k}  \eta
\\
&=  \dv{\ev}{n-2\ell} \eta = \text{LHS } \eqref{BviaF} ,
\end{align*}
where we used $\qbinom{c+\lambda-m}{c}_{q^2} =(-1)^c \qbinom{m-\lambda-1}{c}_{q^2}$ for (*), and used \eqref{eq:1.3.1} for (**).
Therefore, the identity \eqref{FviaB} is valid.
\end{proof}
Only the special case for $\ell=0$ of the formula \eqref{FviaB} will be needed in a later section. 


\subsection{Relation with Lusztig symmetries}
\label{sec:rk1-mod}

Recalling that $i=\tau i$, we have $\bs_i=s_i$. We can view any $\U$-module as a $\Ui$-module by restriction. Recall from \eqref{braid_rescale}
the normalized Lusztig symmetries $T'_{i,\bvs, -1}$ on $\U$.
  
\begin{theorem} \label{thm:res-split}
Let $i\in \I$ such that $i=\tau i$ and $\bvs$ be a balanced parameter. Let $M$ be an integrable $\U$-module whose weights are bounded above. For any $v\in M$, we have
\begin{align}
\label{eq:sTmod}
    \TT'_{i,-1}(v)=\fX_{i,\bvs} T'_{i,\bvs,-1}(v),\qquad
    \TT''_{i,+1}(v)= T''_{i,\bvs,+1}(\fX_{i,\bvs}^{-1} v).
\end{align} 
\end{theorem}

\begin{proof}
For any two balanced parameters $\bvs,\bvs'$, the isomorphism $\phi_{\bvs ,\bvs'}$ is the restriction of $\Phi_{\ov{\bvs }\bvs'}$. Note that $\Phi_{\ov{\bvs}\bvs'} (\fX_{i,\bvs} ) = \fX_{i,\bvs'}$. Therefore, by \cref{thm:braid-iQG}, it suffices to prove \eqref{eq:sTmod} for any fixed balanced parameter, say, $\bvs_\star$. In this case, we have $\vs_i=q_i^{-1}$ thanks to $i=\tau i$ and \eqref{def:bvs-star}. Moreover, $M$ is equipped with the bar involution $\psi^\imath$ for this choice of parameter; cf. \S\ref{subsec:ibar} and \cite{BW18b}.
  
Set $\U_i\cong \U(\sl_2)$ to be the subalgebra of $\U$ generated by $E_i,F_i,K_i^{\pm1}$. Since $M$ is the direct sum of its irreducible $\U_i$-submodules, without loss of generality we assume that $v\in L_i$ for some irreducible $\U_i$-submodule $L_i$ of $M$. Let $\eta$ be the $\U_i$-highest weight vector for $L_i$. Set $n\in \N$ such that $K_i \eta=q_i^{n}\eta$.

Since $\vs_i=q_i^{-1}$, we can rewrite the formulas \eqref{def:fbvs}--\eqref{def:ovfbvs}, for $\ov{p}\in \Z_2$, as follows:
\begin{align}
  \label{eq:fevodd}
    f_{\ov{p},\bvs_\star}(B_i)=\sum_{k\geq 0:\:\ov{k}=\ov{p}}  (-q_i)^{-k/2} \dvi{\ov{k}}{k},
    \qquad
    \tilde{f}_{\ov{p},\bvs_\star}(B_i)=\sum_{k\geq 0:\:\ov{k}=\ov{p}}  (-q_i)^{k/2} \dvi{\ov{k}}{k}.
\end{align} 

By \eqref{eq:ibraid1}--\eqref{eq:ibraid1'}, we have $\TT'_{i,-1}v=f_{\ov{n},\bvs_\star}(B_i)v$ and $\TT''_{i,+1}v = \tilde{f}_{\ov{n},\bvs_\star}(B_i)v$. To prove the identities \eqref{eq:sTmod} for the parameter $\bvs_\star$, it suffices to prove the following.
\vspace{2mm}

$\boxed{\text{Claim }(\star)}$.   
$\fX_{i,\bvs_\star} T'_{i,\bvs_\star,-1}(v)=f_{\ov{n},\bvs_\star}(B_i)v,
$ and $
T''_{i,\bvs_\star,+1}(\fX_{i,\bvs_\star}^{-1} v)=  \tilde{f}_{\ov{n},\bvs_\star}(B_i)v.$

Write $b_k$ for $B_{i,\ov{n}}^{(k)}\eta$ and $v_k$ for $F_i^{(k)}\eta$; note that $v_0=\eta$. The icanonical basis of $L_i$ is given by $\{b_k\mid 0\leq k\leq n\}$ and the canonical basis of $L_i$ is given by $\{v_k\mid 0\leq k\leq n\}$. 
Recall from \cite[Proposition 5.2.2]{Lus93} that, for $0\leq k\leq n ,e=\pm 1,$
\begin{align*}
T'_e (v_k)= (-1)^k q^{ek(n-k+1)}_i v_{n-k},\qquad T''_{e}(v_k)=(-1)^{n-k} q^{e(n-k)(k+1)}_i v_{n-k}.
\end{align*}
By a direct computation via the above formulas and the renormalization \eqref{braid_rescale}, the actions of $T'_{i,\bvs_\star,-1},T''_{i,\bvs_\star,+1}$ are given by
\begin{align}  \label{Tvk_rescaled}
\begin{split}
T'_{i,\bvs_\star,-1} (v_k) &= (-1)^k (-q_i)^{(2k-n)/2}q_i^{-k(n-k+1)} v_{n-k},\\ 
T''_{i,\bvs_\star,+1}(v_k)&=(-1)^{n-k} (-q_i)^{(2k-n)/2}q_i^{(n-k)(k+1)} v_{n-k}.
\end{split}
\end{align}

We now reduce the proof of Claim ($\star$) to the special case for $v=\eta$, the highest weight vector. Indeed, suppose that $\TT'_{i,-1}\eta=f_{\ov{n}}(B_i)\eta$. Since $\TT'_{i,-1} (B_{i,\ov{n}}^{(k)} )=B_{i,\ov{n}}^{(k)}$, applying \cref{thm:split} we have, for any $0\leq k\leq n$,
\begin{align*}
\TT'_{i,-1}b_k
=\TT'_{i,-1} (B_{i,\ov{n}}^{(k)} \eta)
&=\TT'_{i,-1} (B_{i,\ov{n}}^{(k)} )\TT'_{i,-1} \eta
\\
&=B_{i,\ov{n}}^{(k)} f_{\ov{n},\bvs_\star}(B_i) \eta=f_{\ov{n},\bvs_\star}(B_i) b_k.
\end{align*}
This implies that $\TT'_{i,-1}=f_{\ov{n},\bvs_\star}(B_i)$ on $L_i$ since $\{b_k\mid 0\leq k\leq n\}$ is a basis for $L_i$. Thanks to Proposition~\ref{prop:split2}, a similar reduction works for $\TT''_{i,+1}$.

It remains to prove the following special case of Claim ($\star$) with $v=\eta$:
\begin{align} \label{specialClaim}
\fX_{i,\bvs_\star} T'_{i,\bvs_\star,-1} (\eta) =f_{\ov{n},\bvs_\star}(B_i)\eta,
\qquad
T''_{i,\bvs_\star,+1} (\fX_{i,\bvs_\star}^{-1} \eta) =\tilde{f}_{\ov{n},\bvs_\star}(B_i)\eta. 
\end{align}

Let $A=(a_{kl})$ be the transition matrix of the canonical and icanonical bases on $L(n)$, and denote its inverse by $A^{-1}=(a'_{kl})$; that is, for $0\leq l,k \leq n,$ we have
\begin{align}  \label{blvk}
b_k = \sum_{l\leq k} a_{kl} v_l,\qquad
v_k = \sum_{l\leq k} a'_{kl} b_l.
\end{align}
We shall often write $\ov{v} =\psi(v)$, for $v\in L_i$. By \cite[Theorem 5.7]{BW18b}, $\psi^\imath =\fX_{i,\bvs_\star} \psi$ fixes $b_k$, and thus we have
\begin{align}
\label{eq:fXB}
\fX_{i,\bvs_\star} \ov{b_k } =b_k,\qquad (0\leq k\leq n).
\end{align}

To prove the first formula in \eqref{specialClaim}, we compute the action of $\fX_{i,\bvs_\star} T'_{i,\bvs_\star,-1}$ on $\eta =v_0$ using \eqref{Tvk_rescaled}, \eqref{blvk} and \eqref{eq:fXB}:
\begin{align} 
  (-q_i)^{n/2}\fX_{i,\bvs_\star}  T'_{i,\bvs_\star,-1}(\eta)
 =\fX_{i,\bvs_\star} v_{n}
 =\fX_{i,\bvs_\star} \ov{v_{n}}
&=\fX_{i,\bvs_\star}\sum_l \ov{a'_{nl} b_l}  \notag
\\
&= \sum_l \ov{a'_{nl}} b_l 
=\sum_l \ov{a'_{nl}} B_{i,\ov{n}}^{(l)} \eta.
\label{eq:mod7}
\end{align}

Setting $\ell=0$ in \eqref{FviaB} and inserting the index $i$ back, we obtain 
\begin{align*}
F_i^{(n)} \eta =\sum_{c=0}^{\lfloor \frac{n}2 \rfloor} (-q_i)^{-c} \dvi{\ov{n}}{n-2c} \eta,
\end{align*}
that is, we have
\begin{align} \label{anl}
   a'_{nl} = 
   \begin{dcases}
        (-q_i)^{-c}, & \text{ if } l=n-2c,
        \\
        0, & \text{ otherwise}.
    \end{dcases}
\end{align}
Plugging \eqref{anl} into \eqref{eq:mod7} gives us
\begin{align*}
\fX_{i,\bvs_\star} T'_{i,\bvs_\star,-1} (\eta)
=(-q_i)^{-n/2}\sum_{c=0}^{\lfloor \frac{n}2 \rfloor} (-q_i)^{c} \dvi{\ov{n}}{n-2c}\eta
=\sum_{\ov{k}=\ov{n}, 0\leq k \leq n }(-q_i)^{-\frac{k}{2}}  \dvi{\ov{k}}{k}\eta.
\end{align*}
This proves the first formula in \eqref{specialClaim} since $\dvi{\ov{n}}{k}\eta=0$ for $k>n$.
 
Finally, we compute the action of $T''_{i,\bvs_\star,+1} \fX_{i,\bvs_\star}^{-1}$ on $\eta=v_0$ by \eqref{Tvk_rescaled}, \eqref{blvk}, \eqref{eq:fXB} and \eqref{anl}:
\begin{align*}
 T''_{i,\bvs_\star,+1} (\fX_{i,\bvs_\star}^{-1} \eta )&=T''_{i,\bvs_\star,+1}( \eta)= (-q_i)^{n/2} v_n
  =(-q_i)^{n/2}\sum_{l} a'_{nl}b_{l} 
  \\
 &= (-q_i)^{n/2}
    \sum_{c=0}^{\lfloor \frac{n}2 \rfloor} (-q_i)^{-c} \dvi{\ov{n}}{n-2c}\eta 
    =\sum_{\ov{k}=\ov{n}, 0\leq k \leq n }(-q_i)^{\frac{k}{2}}  \dvi{\ov{k}}{k}\eta.
\end{align*}
Hence the second formula in \eqref{specialClaim} holds. 

This completes the proofs of Claim ($\star$) and \eqref{eq:sTmod} for the parameter $\bvs_\star$, and hence \cref{thm:res-split}.
\end{proof}

\begin{corollary}\label{cor:mod-inv}
For $i\in \I$ with $i=\tau i$ and any integrable $\Ui$-module $M$, $\TT'_{i,-1}$ and $\TT''_{i,+1}$ are mutually inverse linear operators on $M$.
\end{corollary}

\begin{proof}
We first assume that $M$ is an integrable $\U$-module and $\bvs$ is a balanced parameter. Recall that $T'_{i,\bvs,-1} T''_{i,\bvs,+1}=\text{Id}_M$. Let $v\in M$. Using \eqref{eq:sTmod}, we have
\[
\TT'_{i,\bvs,-1}\TT''_{i,\bvs,+1}v=\fX_{i,\bvs} T'_{i,\bvs,-1} T''_{i,\bvs,+1}(\fX_{i,\bvs}^{-1} v)=\fX_{i,\bvs}  \fX_{i,\bvs}^{-1} v =v.
\]
Thus, for any $v\in M_{\ov{\lambda}}$, we have
$f_{\ov{p},\bvs}(B_i) \tilde{f}_{\ov{p},\bvs}(B_i)v=v.$ 
Since $M$ can be taken as an arbitrary integrable $\U$-module, we have, for any $\ov{p}\in \Z_2$,
\begin{align}
\label{eq:ff=1}
f_{\ov{p},\bvs}(B_i) \tilde{f}_{\ov{p},\bvs}(B_i)=1.
\end{align}
Now let $M$ be an integrable $\Ui$-module and $\bvs'$ be an arbitrary parameter. Applying the isomorphism $\phi_{\bvs,\bvs'}$ to \eqref{eq:ff=1}, we obtain
\begin{align}
f_{\ov{p},\bvs'}(B_i) \tilde{f}_{\ov{p},\bvs'}(B_i)=1.
\end{align}
This equality implies that $\TT'_{i,\bvs',-1}\TT''_{i,\bvs',+1}v=v$ for any $v\in M$. 

Using similar arguments, one can also show that $\TT''_{i,+1}\TT'_{i,-1}v=v$.
\end{proof}

\subsection{Formulas of $\TT'_{i,-1}$ on divided powers} 
Set the parameter $\bvs=\bvs_\dm$, and thus, $\vs_i=-q_i^{-2}$ thanks to $i=\tau i$. Recall $X_{j,n,\ov{t}}$ from \eqref{Xjn}.

\begin{theorem}
\label{thm:split-T1DP}
Let $\alpha =-c_{ij}$, for $i\neq j\in \wI$. Then we have $\TT'_{i,-1}(X_{j,n,\ov{t}})=b_{i,j;n,n\alpha}$, for $n\ge 0,\ov{t}\in \Z_2$; that is,
\begin{align}
\label{eq:splitT1}
\begin{split}
\TT'_{i,-1}(X_{j,n,\ov{t}})&=\sum_{ r+s =n\alpha} (-1)^{r} q_i^{r} \dv{i,\ov{p+n\alpha}}{s} X_{j,n,\ov{t}} \dv{i,\ov{p}}{r}
\\
&\quad +\sum_{u\geq 1}\sum_{\substack{r+s+2u=n\alpha\\ \ov{r}=\ov{p}}} (-1)^{r} q_i^{r} \dv{i,\ov{p+n\alpha}}{s} X_{j,n,\ov{t}} \dv{i,\ov{p}}{r}.
\end{split}
\end{align}
Moreover, the formula of $\TT''_{i,+1}(X_{j,n,\ov{t}})$ is given by
\begin{align}\label{eq:splitT2}
\begin{split}
\TT''_{i,+1}(X_{j,n,\ov{t}})&=\sum_{ r+s =n\alpha} (-1)^{r} q_i^{r}  \dv{i,\ov{p}}{r} X_{j,n,\ov{t}} \dv{i,\ov{p+n\alpha}}{s}
\\
&\quad +\sum_{u\geq 1}\sum_{\substack{r+s+2u=n\alpha\\ \ov{r}=\ov{p}}} (-1)^{r} q_i^{r} \dv{i,\ov{p}}{r} X_{j,n,\ov{t}} \dv{i,\ov{p+n\alpha}}{s}.
\end{split}
\end{align}
\end{theorem}

\begin{proof}
Recall from \cref{lem:inv} the anti-involution $\sigma^\imath$ and note that $\sigma^\imath$ fixes $X_{j,n,\ov{t}}.$ By \cite{WZ23}, we have $\TT''_{i,+1}(X_{j,n,\ov{t}})=\sigma^\imath\big(\TT'_{i,-1}(X_{j,n,\ov{t}})\big)$. Hence, by applying $\sigma^\imath$, \eqref{eq:splitT2} can be obtained from \eqref{eq:splitT1}.

We prove the formula \eqref{eq:splitT1}. Let $M$ be an integrable $\Ui$-module and $v\in M_{\ov{\lambda}}$ be an iweight vector.
By Theorem~\ref{thm:split}, we have
\begin{align}\label{eq:splitn1}
 \TT'_{i,-1}(X_{j,n,\ov{t}}) \TT'_{i,-1} (v) =  \TT'_{i,-1} \big(X_{j,n,\ov{t}} v\big).
\end{align}
On the other hand, note that $X_{j,n,\ov{t}} v\in M_{\ov{\lambda-n\alpha_j}}$. By \cref{prop:BBij0} and the definition \eqref{eq:ibraid1} of $\TT'_{i,-1}$, we have
\begin{align}\label{eq:esplitn2}
 b_{i,j;n,n\alpha}\TT'_{i,-1} (v) =  \TT'_{i,-1} \big(X_{j,n,\ov{t}} v\big).
\end{align}
Hence, we have
\[
\big(\TT'_{i,-1}(X_{j,n,\ov{t}})-b_{i,j;n,n\alpha}\big)\TT'_{i,-1} v =0.
\]
Replacing $v$ by $\TT''_{i,+1} v $ and using \cref{cor:mod-inv}, we see that $\TT'_{i,-1}(X_{j,n,\ov{t}})-b_{i,j;n,n\alpha}$ acts as $0$ on any integrable $\Ui$-module. Therefore, we must have $\TT'_{i,-1}(X_{j,n,\ov{t}})=b_{i,j;n,n\alpha}$, as desired. The explicit formula for $b_{i,j;n,n\alpha}$ is obtained by setting $m=n\alpha$ in \eqref{splitdiv2}.
\end{proof}

 \section{Braid group operators of diagonal type on integrable $\Ui$-modules}
 \label{sec:diagonal}
 
In this section, we fix $i\in \wI$ such that $c_{i,\tau i}=0$. We introduce linear operators  $\TT'_{i,-1},\TT''_{i,+1}$ (which are shown to be mutually inverse) via explicit formulas on any integrable $\Ui$-modules and show that they are compatible with corresponding symmetries on $\Ui$ given in \cite{WZ23, Z23}. We further show that $\TT'_{i,-1},\TT''_{i,+1}$ specialize to the corresponding ones {\em loc. cit.} on integrable $\U$-modules by relating to Lusztig symmetries. 

\subsection{Rank one formulas of diagonal type}

Let $M\in \mathcal{C}$ be an integrable $\Ui$-module. We define linear operators $\TT'_{i,-1},\TT''_{i,+1}$ on $M$ by letting 
\begin{align}
\label{eq:ibraid2}
\TT'_{i,-1}(v)&= (-1)^{\la_{i,\tau}/2} q_i^{-\la_{i,\tau}/2} \sum_{a-b=\la_{i,\tau}} (-1)^b q_i^{-b} \vs_i^{-(a+b)/2} B_{i}^{(a)}  B_{\tau i}^{(b)}v,
\\\label{eq:ibraid2'}
\TT''_{i,+1}(v)&= (-1)^{\la_{i,\tau}/2} q_i^{-\la_{i,\tau}/2} \sum_{b-a=\la_{i,\tau}} (-1)^b q_i^{b} \vs_i^{-(a+b)/2} B_{\tau i}^{(a)}  B_{i}^{(b)}v,
\end{align}
for any $v\in M_{\ov{\lambda}}$, where $\la_{i,\tau} =\langle h_i-h_{\tau i},\ov{\lambda}\rangle$. There are only finitely many nonzero summands on the right-hand sides of \eqref{eq:ibraid2}--\eqref{eq:ibraid2'} since $M$ is an integrable $\Ui$-module. 

We also introduce the following elements for $m\in \Z$ in some completion of $\Ui$:
\begin{align} 
\label{eq:zm}
z_{m,\bvs}( B_i,B_{\tau i}) 
:= (-1)^{m/2} q_i^{-m/2} \sum_{a-b=m} (-1)^b q_i^{-b} \vs_i^{-(a+b)/2} B_{i}^{(a)}  B_{\tau i}^{(b)},
\\\label{eq:zm'}
z'_{m,\bvs}( B_i,B_{\tau i}) 
:= (-1)^{m/2} q_i^{-m/2} \sum_{b-a=m} (-1)^b q_i^{b} \vs_i^{-(a+b)/2} B_{\tau i}^{(a)}  B_{ i}^{(b)}.
\end{align} 
Hence we have \begin{align} \label{ibraid:diag}
\TT'_{i,-1}(v) = z_{\la_{i,\tau},\bvs}( B_i,B_{\tau i})v, 
\qquad
\TT''_{i,+1}(v) = z'_{\la_{i,\tau},\bvs}( B_i,B_{\tau i})v.
\end{align}

\subsection{Compatibility with braid group symmetries on $\Ui$}

\begin{lemma}
\label{lem:diagmu}
For $\mu\in Y^\imath$, we have 
\begin{align*}
    \langle \mu,\alpha_i+\alpha_{\tau i}\rangle &=0 ,
    \\
\bs_i\mu &=\mu+\langle \mu,\alpha_{\tau i} \rangle(h_i-h_{\tau i}).
\end{align*}
\end{lemma}

\begin{proof}
For $\mu \in Y^\imath$, we have $\tau\mu=-\mu$, and hence, $\langle \mu,\alpha_i \rangle=\langle\tau \mu, \alpha_{\tau i}\rangle=-\langle  \mu, \alpha_{\tau i}\rangle$. This proves the first identity.
The second identity follows by a direct computation:
$\bs_i\mu=  s_is_{\tau i}(\mu)=\mu-\langle \mu,\alpha_i \rangle h_i -\langle \mu,\alpha_{\tau i} \rangle h_{\tau i}=\mu+\langle \mu,\alpha_{\tau i} \rangle(h_i-h_{\tau i}).$
\end{proof}

\begin{theorem}
\label{thm:diag}
Let $i\in \I$ such that $c_{i,\tau i}=0 $. For any $x\in \Ui,v\in M$, we have
\begin{align}\label{eq:qsplit}
 \TT'_{i,-1}(x) \TT'_{i,-1} v &=  \TT'_{i,-1} (x v). 
\end{align}
\end{theorem}

\begin{proof}
By the same argument at the first paragraph in the proof of Theorem~\ref{thm:split}, it suffices to verify the identity \eqref{eq:qsplit} as $x$ runs over a generating set of $\Ui$ and $v\in M_{\ov{\lambda}}$ for $\ov{\lambda}\in X_\imath$; set $m= \langle h_i-h_{\tau i},\ov{\lambda}\rangle $.

We first prove \eqref{eq:qsplit} for $x=K_\mu,\mu\in Y^\imath$.
Recall from \eqref{vsi_same} that $\vs_i=\vs_{\tau i}$ thanks to $c_{i,\tau i}=0$. By \cref{thm:braid-iQG}(b) and \cref{prop:rkone}, we have $\TT'_{i,\bvs ,-1}(K_\mu) = K_{\bs_i\mu}$.
By \eqref{eq:ibraid2}, the action of $\TT'_{i,-1}$ on $M$ is given by a summation of $B_{i}^{(a)}  B_{\tau i}^{(b)} $ for $a-b=\langle h_i-h_{\tau i},\ov{\lambda}\rangle$. We have
\begin{align*}
 K_{\bs_i\mu}  B_{i}^{(a)}  B_{\tau i}^{(b)}v
&= q^{-\langle \bs_i\mu, a\alpha_i+b\alpha_{\tau i}\rangle} 
B_{i}^{(a)}  B_{\tau i}^{(b)} K_{\bs_i\mu} v
\\
&= q^{-\langle\mu,  a\bs_i\alpha_i+b\bs_i\alpha_{\tau i}\rangle} 
B_{i}^{(a)}  B_{\tau i}^{(b)} K_{\bs_i\mu} v
\\
&= q^{ \langle  \mu, a\alpha_i+b\alpha_{\tau i}\rangle} 
B_{i}^{(a)}  B_{\tau i}^{(b)} K_{\bs_i\mu} v
\\
&= q^{ -m\langle  \mu,  \alpha_{\tau i}\rangle} 
B_{i}^{(a)}  B_{\tau i}^{(b)} K_{\bs_i\mu} v
\\
&= B_{i}^{(a)}  B_{\tau i}^{(b)} K_{ \mu} v,
\end{align*}
where we used $ K_{\bs_i\mu}=K_\mu K_{h_i-h_{\tau i}}^{\langle \mu,\alpha_{\tau i}\rangle}$ in the last equality by Lemma~\ref{lem:diagmu}.
Thus, we have proved $\TT'_{i,-1}(K_\mu v)=K_{\bs_i\mu} \TT'_{i,-1}(v)$ as desired.

We next prove \eqref{eq:qsplit} for $x=B_j$, for $j\in \I$. By \eqref{eq:ibraid2}, $\TT'_{i,-1} v=z_{m,\bvs}( B_i,B_{\tau i}) v$ and $\TT'_{i,-1} B_j v=z_{m-c_{ij}+c_{\tau i,j},\bvs}( B_i,B_{\tau i}) B_j v$. Hence, it suffices to establish the following rank one identity (in case $j\in \{i,\tau i\}$) or rank two identity (in case $j\not\in \{i,\tau i\}$):
\begin{align}
\label{eq:Tzm}
 \TT'_{i,-1}(B_j) z_{m,\bvs}( B_i,B_{\tau i})v=z_{m-c_{ij}+c_{\tau i,j},\bvs}( B_i,B_{\tau i}) B_jv.
\end{align}
By Lemma~\ref{lem:iso-parameter} and the definition \eqref{eq:zm}, we have $z_{m,\bvs}(B_i,B_{\tau i})=\phi_{\bvs_\dm,\bvs} z_{m,\bvs_\dm}(B_i,B_{\tau i})$ for any $m$. By \cref{thm:braid-iQG}(2), we have $\TT'_{i,\bvs,-1}\phi_{\bvs_\dm,\bvs}(B_j)=\phi_{\bvs_\dm,\bvs} \TT'_{i,\bvs_\dm,-1}(B_j)$. Therefore, it suffices to establish \eqref{eq:Tzm} for the distinguished parameter $\bvs_\dm$. 

For the distinguished parameter $\bvs_\dm$, the proof of \eqref{eq:Tzm} for $j\in \{i,\tau i\}$ will be carried out in Proposition~\ref{prop:qs-rkone} in \S\ref{subsec:rk1} and for $j\not\in \{i,\tau i\}$ in \cref{prop:qsBB} and \cref{cor:diagBB} in \S\ref{subsec:rk2} below. This proves the theorem. 
\end{proof}

Next we shall derive from \cref{thm:diag} its counterpart for $\TT''_{i,+1}$ using the $\mathrm{i}$bar involution $\psi^\imath$. 

\begin{lemma}\label{lem:Tbar}
Set the parameter $\bvs=\bvs_\star$. We have $\psi^\imath \TT'_{i,-1}(B_j)=(-1)^{c_{ij}-c_{\tau i,j}}\TT''_{i,+1}(B_j)$, for any $j\in \wI$.
\end{lemma}

\begin{proof}
Recall from \S\ref{subsec:ibar} that $\vs_{i,\star}=\vs_{\tau i,\star}=1$, and the bar involution $\psi^\imath$ on $\Ui_{\bvs_\star}$ exists. For $j=i,\tau i$, the desired statement follows from Proposition~\ref{prop:rkone}.

Let $j\neq i,\tau i$. We write $\alpha =-c_{ij}, \beta=-c_{\tau i,j}$. The rank two formulas of relative braid symmetries on the universal iquantum group were conjectured in \cite{CLW23} and established partially in \cite{LW22} and fully in \cite[Theorem 3.7(ii)]{Z23}. Applying the central reduction to the formulas {\em loc. cit.} at $\bvs_\dm$ and then applying $\phi_{\bvs_\dm,\bvs_\star}$, we obtain the formulas for $\TT'_{i,-1}(B_j),\TT''_{i,+1}(B_j)$ as below
\begin{align*}
\TT'_{i,-1}(B_j)&= \sum_{u=0}^{ \min(\alpha,\beta)} \sum_{r=0}^{\alpha-u}\sum_{s=0}^{\beta-u} (-1)^{r+s+u+\frac{\alpha+\beta}{2}}q_i^{r(-u+1)+s(u+1)+u -\frac{\alpha+\beta}{2}}
\\ 
  &\qquad \times B_i^{(\alpha-r-u)}B_{\tau i}^{(\beta-s-u)}B_j B_{\tau i}^{(s)}B_i^{(r)} k_i^u ,
  \\ 
\TT''_{i,+1}(B_j)&=\sum_{u=0}^{ \min(\alpha,\beta)} \sum_{r=0}^{\alpha-u}\sum_{s=0}^{\beta-u} (-1)^{r+s+u+\frac{\alpha+\beta}{2}}q_i^{r( -u+1)+s( u+1)+u -\frac{\alpha+\beta}{2}}
\\ 
  &\qquad \times q_i^{(\alpha-\beta)u} B_i^{(r)}B_{\tau i}^{(s)}B_j B_{\tau i}^{(\beta-s-u)}B_i^{(\alpha-r-u)} k_{\tau i}^u.
\end{align*}
The desired statement directly follows from these formulas.
\end{proof}


\begin{proposition}
\label{prop:diag2}
Let $i\in \I$ such that $c_{i,\tau i}=0$ and $M\in \mathcal C$. For any $x\in \Ui,v\in M$, we have
\begin{align} 
 \label{eq:qsplit2}
 \TT''_{i,+1}(x) \TT''_{i,+1} v &=  \TT''_{i,+1} (x v).
\end{align}
\end{proposition}

\begin{proof}
By the same argument as in the first paragraph of the proof of Theorem~\ref{thm:split}, it suffices to verify the identity \eqref{eq:qsplit2} as $x$ runs over a generating set of $\Ui$ and $v\in M_{\ov{\lambda}}$ for $\ov{\lambda}\in X_\imath$; set $m= \langle h_i-h_{\tau i},\ov{\lambda}\rangle $.

When $x=K_\mu$ for $\mu\in Y^\imath$, the proof of \eqref{eq:qsplit2} is essentially the same as the proof of \eqref{eq:qsplit} in Theorem~\ref{thm:diag}.

We prove \eqref{eq:qsplit2} for $x=B_j$\; $(j\in \I)$. By \eqref{eq:ibraid2'}, it suffices to show that 
\begin{align}
\label{eq:Tzm'}
 \TT''_{i,+1}(B_j) z'_{m,\bvs}( B_i,B_{\tau i})v =z'_{m-c_{ij}+c_{\tau i,j},\bvs}( B_i,B_{\tau i}) B_jv.
\end{align}
Using $\phi_{\bvs,\bvs'}$ and similar arguments in the proof of Theorem~\ref{thm:split}, it suffices to verify \eqref{eq:Tzm'} for the parameter  $\bvs=\bvs_\star$. Recall from \S\ref{subsec:ibar} that $\vs_{i,\star}=\vs_{\tau i,\star}=1$ and the bar involution $\psi^\imath$ on $\Ui$ exists for $\bvs_\star$. Using \eqref{eq:zm}--\eqref{eq:zm'}, we obtain the following identity
\begin{align}\label{eq:psi-zm}
\psi^\imath \tau \big(z_{-m,\bvs_\star}(B_i,B_{\tau i})\big)=(-1)^m z'_{m,\bvs_\star}(B_i,B_{\tau i}).
\end{align} 
Define a new $\Ui$-module $M^\star$ with the same underlying set as $M$ equipped with a twisted action $x\star a :=\psi^\imath \htau(x)a$, for $x\in \Ui, a\in M$. Note that $M^\star\in \mathcal C$. Replacing $j$ by $\tau j$ and $M$ by $M^\star$ in \eqref{eq:Tzm}, we obtain
\begin{align}
\label{eq:Tzmtau}
 \TT'_{i,-1}(B_{\tau j}) z_{-m,\bvs_\star}( B_i,B_{\tau i})\star v =z_{-m+c_{ij}-c_{\tau i,j},\bvs_\star}( B_i,B_{\tau i}) B_{\tau j}\star v.
\end{align}
Recall that $\htau \TT' _{i,-1}=\TT' _{i,-1} \htau$. Unraveling the $\star$-action in \eqref{eq:Tzmtau} and using \eqref{eq:psi-zm}, we obtain 
\begin{align}  \label{eq:desired}
\psi^\imath \Big(\TT' _{i,-1}(B_j)\Big) \cdot  z'_{m,\bvs_\star}( B_i,B_{\tau i})v = (-1)^{c_{ij}-c_{\tau i,j}} z'_{m-c_{ij}+c_{\tau i,j},\bvs_\star}( B_i,B_{\tau i}) B_j v.
\end{align}
By Lemma~\ref{lem:Tbar}, we have $\psi^\imath  \TT'_{i,-1}(B_j)=  (-1)^{c_{ij}-c_{\tau i,j}} \TT''_{i,+1}(B_j)  $. Then the relation \eqref{eq:Tzm'} follows from \eqref{eq:desired} as desired. 
\end{proof}


\subsection{Identities in rank one diagonal type with parameter $\bvs_\dm$}
\label{subsec:rk1}

Set the parameter $\bvs =\bvs_\dm$ in this subsection. We write $z_m(B_i,B_{\tau i})$ for $z_{m,\bvs_\dm}(B_i,B_{\tau i})$. Set $\Ui_{i,\tau i}$ to be the subalgebra of $\Ui$ generated by $B_j,k_j^{\pm1}$ for $j=i,\tau i$. The algebra $\Ui_{i,\tau i}$ itself is an iquantum group of type AIII$_{11}$. By \cite[Theorem 7.4]{Let02}, the defining relations of $\Ui_{i,\tau i}$ are 
\begin{align*}
[B_i,B_{\tau i}]&=q^{-1} \frac{k_i-k_i^{-1}}{q_i-q_i^{-1}}, \qquad 
k_i B_i k_i^{-1}=q_i^{-2} B_i, \qquad k_i B_{\tau i} k_i^{-1}=q_i^{2} B_{\tau i}.
\end{align*}
Thus, there is an algebra isomorphism
\begin{align}
\label{eq:iso-sl2}
\Ui_{i,\tau i}\cong \U(\sl_2), \quad q_i\mapsto q, \quad B_i \mapsto F_1,  \quad B_{\tau i} \mapsto   -q^{-1} E_1, \quad k_i\mapsto K_1.
\end{align}
Under this isomorphism, we can view an (integrable) $\Ui_{i,\tau i}$-module as an $\U(\sl_2)$-module, and then an iweight $\ov{\lambda}$ for $\Ui_{i,\tau i}$ is identified with a $\sl_2$-weight $m$, where $m=\langle h_i-h_{\tau i},\ov{\lambda}\rangle$. Hence, we shall use an integer to denote an iweight.

In addition,  by Proposition~\ref{prop:rkone}, the braid group action $\TT'_{i,-1}$ on $\Ui_{i,\tau i}$ is explicitly given by
\begin{align}\label{eq:iso-braid}
\TT'_{i,-1}(B_i)=q_i B_{\tau i} k_{i},\qquad \TT'_{i,-1}(B_{\tau i})=q_i B_{i} k_{i}^{-1}.
\end{align}
Then the symmetry $\TT'_{i,-1}$ on $\Ui_{i,\tau i}$ is identified with Lusztig's symmetry $T'_{1,-1}$ on $\U(\sl_2)$. By \eqref{eq:ibraid2}, the action of $\TT'_{i,-1}$ on integrable $\Ui$-modules is identified with the action of $T'_{1,-1}$ on $\U(\sl_2)$-modules.


Following \cite[\S2.1]{Lus09}, the action of $T'_{1,-1}$ on a weight $m$ vector $v$ is given by
\begin{align}\label{eq:CKM}
  T'_{1,-1} v= \sum_{a-b=m} (-1)^b q^{-b}  F_{1}^{(a)}  E_{1}^{(b)}v.
\end{align}

\begin{lemma} 
\label{lem:braid-neq}
Let $v'$ be a $\U(\sl_2)$-weight vector with weight $m'< m$. Then we have
\begin{align}\label{eq:braid-neq}
 \sum_{a-b=m} (-1)^b q^{b(m-m'-1)}  F_{1}^{(a)}  E_{1}^{(b)}v'=0.
\end{align}
\end{lemma}

\begin{proof}
We prove this lemma for $m\geq0$; the proof for $m<0$ is similar and hence omitted.

Using Lemma~\ref{lem:EF}, we rewrite the left-hand side of \eqref{eq:braid-neq} as
\begin{align*}
 &\sum_{a-b=m} (-1)^b q^{b(m-m'-1)}  F_{1}^{(a)}  E_{1}^{(b)}v'
 \\
 =& \; \sum_{b\geq 0} \sum_{s=0}^{b}  (-1)^b q^{b(m-m'-1)}\qbinom{m-m'}{s} E_{1}^{(b-s)}F_{1}^{(b+m-s)} v'
 \\ 
 =& \;   \sum_{s\geq 0} (-1)^s q_i^{s(m-m'-1)}\Big(\sum_{b\geq s} (-1)^{b-s} q^{(b-s)(m-m'-1)}    \qbinom{m-m'}{b-s} \Big)E_{1}^{( s)}F_{1}^{(m+s)} v'.
\end{align*}
Since $m> m'$, we have $\sum_{b\geq s} (-1)^{b-s} q^{(b-s)(m-m'-1)}    \qbinom{m-m'}{b-s}=0$, for fixed $s\ge 0$; cf. \cite[1.3.4]{Lus93}.
\end{proof}

\begin{proposition}
\label{prop:qs-rkone}
Let $M$ be an integrable $\Ui$-module and $v\in M_{\ov{\lambda}}$. Set  $m=\langle h_i-h_{\tau i},\ov{\lambda} \rangle $. We have 
\begin{align}
\label{qsrkone1}
z_{m-2}( B_i,B_{\tau i}) B_i v &= \TT'_{i,-1}(B_i) z_m( B_i,B_{\tau i})v,
\\\label{qsrkone2}
z_{m+2}( B_i,B_{\tau i}) B_{\tau i} v&= \TT'_{i,-1}(B_{\tau i}) z_m( B_i,B_{\tau i})v.
\end{align} 
\end{proposition}

\begin{proof}
The proofs for these two identities are similar and we only prove \eqref{qsrkone1} here. Under the isomorphism \eqref{eq:iso-sl2}, we view $M$ as a $\U(\sl_2)$-module and then $v$ is a $\U(\sl_2)$-weight vector of weight $m$. Using the compatibility of Lusztig's symmetry, we obtain 
\[
T'_{1,-1}(F_1v)=T'_{1,-1}(F_1)T'_{1,-1}v.
\] By \eqref{eq:CKM}, we have $T'_{1,-1}v=z_m(F_1,q^{-1}E_1)v$ and $T'_{1,-1}(F_1v)=z_{m-2}(F_1,q^{-1}E_1)v$. Thus, we have
\[
z_{m-2}(F_1,q^{-1}E_1)v=T'_{1,-1}(F_1) z_m(F_1,q^{-1}E_1)v.
\]
The desired identity follows once we identify $T'_{1,-1}$ and $\TT'_{1,-1}$; see \eqref{eq:iso-sl2}--\eqref{eq:iso-braid}.
\end{proof}

\subsection{Identities in rank two diagonal type with parameter $\bvs_\dm$}
\label{subsec:rk2}

Set the parameter $\bvs=\bvs_\dm$ in this subsection. Then we have $\vs_i=\vs_{\tau i}=-q_i^{-1}$.


Recall the element $b_{i,\tau i,j;n,m_1,m_2}$ from \eqref{diagdiv} and \eqref{def:diagBij1}-\eqref{def:diagBij2}. 
Write $b_{n, m_1,m_2}$ for $b_{i,\tau i,j;n,m_1,m_2}$, and $\alpha=-c_{ij}, \beta =-c_{\tau i,j}$.

\begin{lemma}
\label{lem:qsBBij}
We have, for any $ \ell\geq 0$,
\begin{align}
\label{eq:bBij1}
&b_{n, n\alpha,n\beta} B_i^{(\ell)}=\sum_{x=0}^\ell q_i^{(\ell-x)(n\alpha-x)-x^2} B_i^{(\ell-x)} b_{n, n\alpha,n\beta-x}k_i^x.
\end{align}
\end{lemma}

\begin{proof}
Follows by induction on $\ell$ using the recursive relation \eqref{def:diagBij1} and keeping in mind \eqref{SerreL:diagonal}.
\end{proof}

\begin{lemma}
\label{lem:qsbij}
We have, for any $0\leq x\leq n\beta$ and $\ell\geq 0$,
\begin{align}
\label{eq:bBij3}\notag
&\quad b_{n,n\alpha,n\beta-x} B_{\tau i}^{(\ell)}\\
& = \sum_{y=0}^x \sum_{r=0}^{n\alpha} (-1)^y q_i^{(\ell-r)(n\beta-2x+y)-\ell r-y} \qbinom{n\beta-x+y}{y}_i B_{\tau i}^{(\ell-r-y)} b_{n,n\alpha-r,n\beta-x+y} k_i^{-r}.
\end{align}
\end{lemma}

\begin{proof}
We prove \eqref{eq:bBij3} by induction on $\ell$ for any fixed $x$. The case $\ell=0$ is obvious. It suffices to show that $\eqref{eq:bBij3}_{\ell}\Rightarrow\eqref{eq:bBij3}_{\ell+1}$. Using the induction hypothesis and the recursive relation \eqref{def:diagBij2}, we have
 \begin{align*}
& [\ell+1]_i b_{n,n\alpha,n\beta-x} B_{\tau i}^{(\ell+1)}= b_{n,n\alpha,n\beta-x} B_{\tau i}^{(\ell)} B_{\tau i}
\\
=& \;  \sum_{y=0}^x \sum_{r=0}^{n\alpha} (-1)^y  q_i^{(\ell-r)(n\beta-2x+y)-\ell r-y-2r} \qbinom{n\beta-x+y}{y}_i  B_{\tau i}^{(\ell -r-y)} b_{n,n\alpha-r,n\beta-x+y} B_{\tau i} k_i^{-r}
\\
=& \;  \sum_{y=0}^x \sum_{r=0}^{n\alpha} (-1)^y  q_i^{(\ell -r)(n\beta-2x+y)-\ell r-y-2r} \qbinom{n\beta-x+y}{y}_i  B_{\tau i}^{(\ell -r-y)} \times
\\
&\qquad \times \big(q_i^{n\beta-2x+2y}B_{\tau i} b_{n,n\alpha-r,n\beta-x+y}
-q_i^{n\beta-2x+2y} [n\beta-x+y+1]_i b_{n,n\alpha-r,n\beta-x+y+1}
\\
&\qquad\qquad  +q_i^{-1} [r+1]_i b_{n,n\alpha-r-1,n\beta-x+y} k_i^{-1} \big) k_i^{-r} 
\\
=& \; \sum_{y=0}^x \sum_{r=0}^{n\alpha}(-1)^y q_i^{(\ell -r)(n\beta-2x+y)-\ell r-y-2r} \varepsilon_{y,r}B_{\tau i}^{(\ell +1-r-y)} b_{n,n\alpha-r,n\beta-x+y }k_i^{-r} ,
\end{align*}
where the scalar $\varepsilon_{y,r}$ is given by
\begin{align*}
\varepsilon_{y,r} &= q_i^{n\beta-2x+2y} \qbinom{n\beta-x+y}{y}_i [\ell -r-y+1]_i + q_i^{n\beta-2x+y+\ell +1} \qbinom{n\beta-x+y}{y}_i [r ]_i
\\
&\quad + q_i^{n\beta-2x+2y-\ell +r-1} \qbinom{n\beta-x+y-1}{y-1}_i [n\beta-x+y ]_i 
\\
&= q_i^{n\beta-2x+ y +r}\qbinom{n\beta-x+y}{y}_i [\ell +1]_i.
\end{align*}
Combining the above two computations, we obtain $\eqref{eq:bBij3}_{\ell +1}$ as desired.
\end{proof}

\begin{proposition}
\label{prop:qsBB}
Let $v$ be an iweight vector of iweight $m$. For any $j\in \wI,j\neq i,\tau i$, we have
\begin{align}
\label{eq:zBTBz}
  b_{n,n\alpha,n\beta} z_m( B_i,B_{\tau i})v &=z_{m+n\alpha-n\beta}( B_i,B_{\tau i}) b_{n,0,0} v.
\end{align}
\end{proposition}

\begin{proof}
For the parameter $\vs_i=\vs_{\tau i}=-q_i^{-1}$, the formula \eqref{eq:zm} for $z_m(B_i,B_{\tau i})$ specializes to
\begin{align}
z_m( B_i,B_{\tau i}) =  \sum_{a-b=m}   B_{i}^{(a)}  B_{\tau i}^{(b)}.
\end{align}
Using Lemmas~\ref{lem:qsBBij}--\ref{lem:qsbij}, we compute $b_{ n,n\alpha,n\beta}z_m( B_i,B_{\tau i})v$ as follows 
\begin{align*}
    b_{n,n \alpha,n\beta}z_m( B_i,B_{\tau i})v
    &=\sum_{a-b=m} b_{n,n\alpha,n\beta} B_{i}^{(a)}  B_{\tau i}^{(b)}v
    \\
    &=\sum_{a-b=m} \sum_{x=0}^{\min(a,n\beta)} q_i^{(a-x)(n\alpha-x)-x^2+2xb} B_i^{(a-x)} 
    b_{n, n\alpha,n\beta-x} B_{\tau i}^{(b)}k_i^xv
    \\
     &=\sum_{a-b=m} \sum_{x=0}^{\min(a,n\beta)} \sum_{y=0}^x \sum_{r=0}^{n\alpha}q_i^{(a-x)(n\alpha-r) }
     (-1)^y q_i^{(b-r)(n\beta- x+y) -y} \times
     \\
     &\qquad\times\qbinom{n\beta-x+y}{y}_i  B_i^{(a-x)} B_{\tau i}^{(b-r-y)} b_{n,n\alpha-r,n\beta-x+y}  v
     \\
     &=\sum_{r= 0}^{n\alpha} \sum_{a\geq0} \sum_{x=0}^{n\beta} \sum_{y=0}^x q_i^{a(n\alpha-r) }
     (-1)^{x-y} q_i^{(a+x-m-r)(n\beta- y) -x+y} \times
     \\
     &\qquad\times\qbinom{n\beta-y}{x-y}_i B_i^{(a)} B_{\tau i}^{(a-r+y-m)} b_{n,n\alpha-r,n\beta-y}  v
     \\
     &=\sum_{r= 0}^{n\alpha} \sum_{a\geq 0}  \sum_{y=0}^{n\beta} \sum_{ x=y}^{n\beta} q_i^{a(n\alpha-r) }
     (-1)^{x-y} q_i^{(a+x-m-r)(n\beta- y) -x+y} \times
     \\
     &\qquad\times\qbinom{n\beta-y}{x-y}_i  B_i^{(a)} B_{\tau i}^{(a-r +y-m)} b_{n,n\alpha-r,n\beta-y}  v.
\end{align*}
Using the following $q$-binomial identity
\begin{align*}
 \sum_{ x=y}^\beta(-1)^{x-y} q_i^{(x-y)(n\beta-y-1)}\qbinom{n\beta-y}{x-y}_i=0, \qquad\forall y<n\beta,
\end{align*}
we simplify the above formula of $b_{n, n\alpha,n\beta}z_m( B_i,B_{\tau i})v$ as
\begin{align}
\label{eq:qsbzij}
    b_{ n,n\alpha,n\beta}z_m( B_i,B_{\tau i})v
    & =\sum_{r= 0}^{\alpha}   \sum_{a\geq 0} q_i^{a(n\alpha-r) }
         B_i^{(a)} B_{\tau i}^{(a-r +n\beta-m)} b_{n,n\alpha-r, 0} v.
\end{align}

We claim that 
\begin{align}  \label{claim_BBb}
      \sum_{a\geq 0} q_i^{a(n\alpha-r) }
         B_i^{(a)} B_{\tau i}^{(a-r +n\beta-m)} b_{n,n\alpha-r, 0} v=0, \qquad \forall 0\leq r<n\alpha.
\end{align}
Indeed, note that $b_{n,n\alpha-r, 0} v$ has iweight $m-n\alpha-n\beta+2r$. Then this claim follows from the isomorphism $\Ui_{i,\tau i}\cong \U(\sl_2)$ and Lemma~\ref{lem:braid-neq}.

Now plugging \eqref{claim_BBb} into \eqref{eq:qsbzij}, we obtain the desired identity \eqref{eq:zBTBz}.
\end{proof}

We now complete the proof of \cref{thm:diag} by proving the desired identity \eqref{eq:Tzm}.

\begin{corollary}\label{cor:diagBB}
 For any $j\in \wI,j\neq i,\tau i$, the identity \eqref{eq:Tzm} holds.
\end{corollary}

\begin{proof}
Setting $n=1$ in Proposition~\ref{prop:qsBB}, we obtain 
\begin{align} \label{eq:bzm}
  b_{1,\alpha,\beta} z_m( B_i,B_{\tau i})v &=z_{m+\alpha-\beta}( B_i,B_{\tau i}) b_{1,0,0} v,
\end{align}
where $v$ can be any iweight vector of iweight $m$ in any integrable $\Ui$-module. 
Recall from \S\ref{sec:braid-formula} that $\TT_{i,\bvs_\dm,-1}'(B_j)=b_{1,\alpha,\beta}$ and $b_{1,0,0}=B_j$. Then the desired identity \eqref{eq:Tzm} follows by \eqref{eq:bzm}.
\end{proof}

\subsection{Relation with Lusztig symmetries}

We can view any $\U$-module as a $\Ui$-module by restriction. Recall from \eqref{braid_rescale} the normalized Lusztig symmetries $T'_{\bs_i,\bvs, -1}$ on $\U$. Thanks to $c_{i,\tau i}=0 $, we have $\bs_i=s_i s_{\tau i}$. 

\begin{theorem}
\label{thm:res-diag}
Let $c_{i,\tau i}=0 $ and $\bvs$ be a balanced parameter. Let $M$ be an integrable $\U$-module whose weights are bounded above. For any $v\in M$, we have
\begin{align}
\label{eq:ssTmod0}
  \TT'_{i,-1}(v)=\fX_{i,\bvs} T'_{\bs_i,\bvs,-1}(v),\qquad
  \TT''_{i,+1}(v)=T''_{\bs_i,\bvs,+1}(\fX_{i,\bvs }^{-1} v).
\end{align}
\end{theorem}

\begin{proof}
By the same argument as in the first paragraph of the proof of \cref{thm:res-split}, it suffices to prove the identities \eqref{eq:ssTmod0} for a fixed balanced parameter, say  $\bvs_\dm$.

We prove the first identity in \eqref{eq:ssTmod0}. Let $\U_{i,\tau i}$ be the subalgebra of $\U$ generated by $E_i,E_{\tau i},F_i,F_{\tau i},K_i^{\pm1},K_{\tau i}^{\pm1}$. Then $\Ui_{i,\tau i}$ is naturally a subalgebra of $\U_{i,\tau i}$. There is an algebra isomorphism
\begin{align}\label{eq:U-iso}
\U_{i,\tau i} &\longrightarrow \U(\sl_2) \otimes \U(\sl_2),
\\\notag
K_i &\mapsto 1\otimes K_1,\quad  K_{\tau i} \mapsto K_1^{-1}\otimes 1, \qquad q_i\mapsto q,
\\\notag
F_i &\mapsto 1\otimes F_1, \quad F_{\tau i} \mapsto -q^{-1} E_1\otimes 1,
\quad E_i\mapsto 1\otimes E_1,\quad E_{\tau i} \mapsto -q F_1\otimes 1.
\end{align}
Under this isomorphism, $T'_{\bs_i ,-1}=T'_{i,-1} T_{\tau i,-1}'$ is identified with $T_{1,-1}'\otimes T_{1,-1}'$. Let $\Theta$ be the quasi $R$-matrix associated to $\U(\sl_2)$. Using the formula in \cite[3.3.3]{DK19} for $\fX_i$ and the formula \cite[4.1.4]{Lus93} for $\Theta$, $\fX_i$ is identified with $\Theta$ via \eqref{eq:U-iso}.  Moreover, $B_i,B_{\tau i}$ are sent to $\Delta(F_1),-q^{-1}\Delta(E_1)$. Hence, we have the following commutative diagram
\begin{center}
\begin{tikzcd}
\U_{i,\tau i} \arrow[r, "\sim" ] 
& \U(\sl_2) \otimes \U(\sl_2)  \\
\Ui_{i,\tau i}  \arrow[r, "\sim " ]\arrow[u]
&  \U(\sl_2)\arrow[u, "\Delta "]
\end{tikzcd}
\end{center} 
where the isomorphism on the top is given by \eqref{eq:U-iso} and the isomorphism on the bottom is given by \eqref{eq:iso-sl2}. Via the above identifications, the first identity in \eqref{eq:ssTmod0} is equivalent to
\begin{align}
\Delta(T'_{1,-1})v =\Theta (T'_{1,-1}\otimes T'_{1,-1})v.
\end{align}
The last equality follows from \cite[Proposition 5.3.4]{Lus93} and then the first identity in \eqref{eq:ssTmod0} follows.

We next prove the second identity in \eqref{eq:ssTmod0}. Under the isomorphism \eqref{eq:iso-sl2}, $\TT''_{i,+1}$ is identified with $T''_{1,+1}$. Since $T''_{1,+1}$ and $T'_{1,-1}$ are mutually inverses, $\TT''_{1,+1}$ and $\TT'_{i,-1}$ are mutually inverses. Now, replacing $v$ by $\TT''_{1,+1}v$ in the first identity of \eqref{eq:ssTmod0}, we obtain the second identity in \eqref{eq:ssTmod0}.
\end{proof}

\begin{corollary}
\label{cor:qsmod-inv}
For $i\in \I$ with $c_{i,\tau i}=0$ and any integrable $\Ui$-module $M$, $\TT'_{i,-1},\TT''_{i,+1}$ are mutually inverse linear operators on $M$.
\end{corollary}

\begin{proof} 
With the help of \cref{thm:res-diag}, the proof for this corollary is the same as the proof for \cref{cor:mod-inv}.
\end{proof}

\subsection{Formulas of $\TT'_{i,-1}$ on divided powers}

Set the parameter $\bvs=\bvs_\dm$ in this subsection. Then we have $\vs_i=\vs_{\tau i}=-q_i^{-1}$. Recall $X_{j,n,\ov{t}}$ from \eqref{Xjn}.

\begin{theorem}
\label{thm:diag-T1DP}
Let $i$ be such that $c_{i,\tau i}=0$, and $j\in\I$ with $j\neq i,\tau i$. Set $\alpha=-c_{ij},\beta=-c_{\tau i,j}$. Then we have $\TT'_{i,-1}(X_{j,n,\ov{t}})=b_{i,\tau i, j;n,n\alpha,n\beta}$, for $n\ge 0,\ov{t}\in \Z_2$; that is, 
\begin{align}\notag
\TT'_{i,-1}(X_{j,n,\ov{t}})&= \sum_{u=0}^{ \min(n\alpha,n\beta)} \sum_{r=0}^{n\alpha-u}\sum_{s=0}^{n\beta-u} (-1)^{r+s }q_i^{r(-u+1)+s(u+1) } 
\\ 
  &\qquad \times B_i^{(n\alpha-r-u)}B_{\tau i}^{(n\beta-s-u)}X_{j,n,\ov{t}} B_{\tau i}^{(s)}B_i^{(r)}  k_i^u.
  \label{TiBjn:diag}
\end{align}
Moreover, the formula of $\TT''_{i,+1}(X_{j,n,\ov{t}})$ is given by
\begin{align}\notag
\TT''_{i,+1}(X_{j,n,\ov{t}})&= \sum_{u=0}^{ \min(n\alpha,n\beta)} \sum_{r=0}^{n\alpha-u}\sum_{s=0}^{n\beta-u} (-1)^{r+s }q_i^{r(-u+1)+s(u+1) } 
\\ 
  &\qquad \times k_i^{-u} B_i^{(r)} B_{\tau i}^{(s)}  X_{j,n,\ov{t}} B_{\tau i}^{(n\beta-s-u)}B_i^{(n\alpha-r-u)}.
  \label{TiBjn:diag2}
\end{align}
\end{theorem}

\begin{proof}
The proof is the same as for \cref{thm:split-T1DP} and hence omitted.
\end{proof}
\section{New monomial bases of simple $\U(\sl_3)$-modules}
\label{sec:quasi basis}

In this section, we construct a new monomial basis for any finite-dimensional simple $\U(\sl_3)$-module viewed as a $\Ui(\sl_3)$-module. We establish explicit transition matrices between the monomial basis and Lusztig's canonical basis on such a module. This section plays a preparatory role for Section~\ref{sec:quasisplit}. 

\subsection{Multiplicative formulas in $\Ui(\sl_3)$} \label{subsec:mult}

Recall $\U(\sl_3) =\langle E_a, F_a, K_a^{\pm 1} \mid a=1,2\rangle$. The quasi-split iquantum group $\Ui(\sl_3)$ of type AIII is the subalgebra of $\U(\sl_3)$ associated with a nontrivial diagram involution $\tau$ generated by $B_1,B_2,k_1$ with the parameter $\bvs=(\vs_1,\vs_2)$. We shall work exclusively with the quantum symmetric pair $\big(\U(\sl_3),\Ui(\sl_3)\big)$ in this section.

In this subsection, we fix the parameter $\vs_1=\vs_2=-q^{1/2}$ and will work out various multiplication formulas in $\Ui(\sl_3)$ involving the divided powers of $B_1, B_2$. Denote 
 \[
 \{k_1;a\}:=q^a k_1 +q^{-a} k_1^{-1}. 
 \]
 We shall freely use the following relations in $\Ui (\sl_3)$ with parameter $\vs_1=\vs_2=-q^{1/2}$ (cf., e.g., \cite[Theorem 7.1]{Ko14}): 
 \begin{align}
k_1 B_1 =q^{-3} B_1 k_1, \qquad &
k_1 B_2 =q^{3} B_2 k_1,
\label{KB1} \\
    B_1^{(2)}B_2 -B_1B_2B_1 +B_2B_1^{(2)}
      &= B_1 (q^{-3/2} k_1 + q^{3/2} k_1^{-1})=B_1\{k_1;-3/2\},
     \label{B121} \\
    B_2^{(2)}B_1 -B_2B_1B_2 +B_1B_2^{(2)}
      &= B_2 (q^{-3/2} k_1^{-1} + q^{3/2} k_1)=B_2\{k_1;3/2\}.
      \label{B212}
 \end{align}

 We introduce the shorthand notation, for $a,b,c \in \N$,  
\[
B^{(a,b,c)}_{2}= B_{1}^{(a)}B_{2}^{(b)}B_{1}^{(c)},
\qquad 
B^{(a,b,c)}_{1}= B_{2}^{(a)}B_{1}^{(b)}B_{2}^{(c)}.
\] 
We also set $B^{(a,b,c)}_{i}=0$ if one of $a,b,c$ is negative. 

 \begin{lemma}
 \label{lem:BBl}
 The following identity holds,  for $l\geq 0$ and $d\in\Z$:
 \begin{align}\notag
     B_1 B_{1}^{(1,l+1+d,l)}
     =& \;  [l+1+d]  B_{1}^{(1,l+2+d,l)}+[l+1] B_{1}^{(0,l+2+d,l+1)}
     \\
     &- [l+1+d] B_{1}^{(0,l+1+d,l)} \{k_1;l-3/2-2d\}.
     \label{eq:BBld}
 \end{align}
 \end{lemma}

 \begin{proof} 
 We fix $l\geq 0$ and prove \eqref{eq:BBld} by induction on $d$. Denote \eqref{eq:BBld} by $\eqref{eq:BBld}_{d,l}$ in the proof. The base cases $\eqref{eq:BBld}_{-l-1,l}$ and $\eqref{eq:BBld}_{-l-2,l}$ follow by a simple computation while the case $\eqref{eq:BBld}_{d,l}$ for $l+2+d<0$ is trivial. 

 It remains to show that $\eqref{eq:BBld}_{d-1,l}\Rightarrow \eqref{eq:BBld}_{d,l}$ for $d>-l-1$. Indeed, we have
 \begin{align*}
    &\quad [l+1+d]  B_1 B_{1}^{(1,l+1+d,l)}
    \\
    &=(B_1 B_2 B_1) B_1^{(l+d)} B_2^{(l)}
    \\
    &= B_1^{(2)} B_2 B_1^{(l+d)} B_2^{(l)} +  B_2 B_1^{(2)} B_1^{(l+d)} B_2^{(l)}-B_1 \{k_1;-3/2\} B_1^{(l+d)} B_2^{(l)}
    \\
    &=[2]^{-1} B_1 \big( B_1 B_1^{(1,l+d,l)}\big) + \qbinom{l+d+2}{2} B_1^{(1,l+2+d,l)} 
    \\
    &\quad - [l+1+d] B_1^{(0,l+1+d,l)}\{k_1;-3/2-3d\}
    \\
    &\overset{(*)}{=}[2]^{-1}[l+d] B_1 B_{1}^{(1,l+1+d,l)}+[2]^{-1} [l+1] [l+2+d] B_{1}^{(0,l+2+d,l+1)}
     \\
     &\quad - \qbinom{l+1+d}{2} B_{1}^{(0,l+1+d,l)} \{k_1;l+1/2-2d\}- [l+1+d] B_1^{(0,l+1+d,l)}\{k_1;-3/2-3d\}
    \\
    &\quad + \qbinom{l+2+d}{2} B_1^{(1,l+2+d,l)} 
    \\
    &=[2]^{-1}[l+d] B_1 B_{1}^{(1,l+1+d,l)}+[2]^{-1} [l+1] [l+2+d] B_{1}^{(0,l+2+d,l+1)}
     \\
     &\quad - \qbinom{l+2+d}{2} B_{1}^{(0,l+1+d,l)} \{k_1;l-3/2-2d\} 
     + \qbinom{l+2+d}{2} B_1^{(1,l+2+d,l)},
 \end{align*}
 where the equality (*) follows by the induction hypothesis $\eqref{eq:BBld}_{d-1,l}$. Now, moving the first summand on the right-hand side to the left and using $[2][l+1+d]=[l+2+d]+[l+d]$, we obtain $\eqref{eq:BBld}_{d,l}$ from the above computation as desired.
 \end{proof}
 
 \begin{proposition}
 The following identity holds,  for $a,l \geq 0$ and $d\in\Z$:
 \begin{align}\notag
     B_1 B_{1}^{(a,l+a+d,l)}
     =& \; [l+1+d]  B_{1}^{(a,l+a+d+1,l)}+[l+1] B_{1}^{(a-1,l+a+d+1,l+1)}
     \\
     &- [l+1+d] B_{1}^{(a-1,l+a+d,l)} \{ k_1;l-a-1/2-2d \}.
     \label{eq:BBad}
 \end{align}
 In particular, when $d=0$, we have, for $t,l\geq 0$,
  \begin{align} 
     B_1 B_{1}^{(t,l+t,l)}
     &= [l+1] \big( B_{1}^{(t,l+t+1,l)}+ B_{1}^{(t-1,l+t+1,l+1)}
      - B_{1}^{(t-1,l+t,l)} \{ k_1;l-t-1/2 \} \big).
     \label{eq:BBa}
 \end{align} 
 \end{proposition}

 \begin{proof}
We prove \eqref{eq:BBad} by induction on $a$. Denote \eqref{eq:BBad} by $\eqref{eq:BBad}_{a,d,l}$ in this proof. The base case $\eqref{eq:BBad}_{1,d,l}$ is proved in Lemma~\ref{lem:BBl} while $\eqref{eq:BBad}_{0,d,l}$ is trivial.

It remains to show that $\eqref{eq:BBad}_{a-1,d+1,l}+\eqref{eq:BBad}_{a-2,d+2,l}\Rightarrow \eqref{eq:BBad}_{a,d,l}$.
Indeed, we have
\begin{align*}
&\quad \qbinom{a}{2} B_1 B_{1}^{(a,l+a+d,l)} 
\\
&=B_1 B_2^{(2)} B_2^{(a-2)} B_1^{(l+a+d)} B_2^{(l)}
\\
&= B_2 B_1 B_2  B_2^{(a-2)} B_1^{(l+a+d)} B_2^{(l)} - B_2^{(2)} B_1 B_2^{(a-2)} B_1^{(l+a+d)} B_2^{(l)}\\
 &\quad+ B_2 \{k_1; 3/2\}  B_2^{(a-2)} B_1^{(l+a+d)} B_2^{(l)}
 \\
&= [a-1] B_2 B_1 B_{1}^{(a-1,l+a+d,l)} -B_2^{(2)} B_1 B_{1}^{(a-2,l+a+d,l)}\\
 &\quad+ [a-1] \{k_1; -3/2\}  B_2^{(a-1)} B_1^{(l+a+d)} B_2^{(l)}
 \\
&\overset{(*)}{=} [a-1][a][l+2+d]  B_{1}^{(a ,l+a+d+1,l)}+ [a-1]^2 [l+1] B_{1}^{(a-1,l+a+d+1,l+1)}
     \\
     &\quad- [a-1]^2 [l+2+d] B_{1}^{(a-1,l+a+d,l)} \{k_1;l-a-3/2-2d\}
     \\
 &\quad- [l+3+d]\qbinom{a}{2}  B_{1}^{(a ,l+a+d+1,l)}- [l+1] \qbinom{a-1}{2}B_{1}^{(a-1,l+a+d+1,l+1)}
     \\
     & \quad+  [l+3+d] \qbinom{a-1}{2} B_{1}^{(a-1,l+a+d,l)} \{k_1;l-a-5/2-2d\}
     \\
     &\quad+ [a-1] B_{1}^{(a-1,l+a+d,l)} \{k_1; -3d-9/2\} 
     \\
&= [l+1+d] \qbinom{a}{2}  B_{1}^{(a ,l+a+d+1,l)} -[l+1] \qbinom{a}{2} B_{1}^{(a-1,l+a+d+1,l+1)}
     \\
 & \quad-[l+1+d] \qbinom{a}{2}B_{1}^{(a-1,l+a+d,l)}  \{k_1;l-a-1/2-2d\},
\end{align*}
 where the equality (*) follows by applying the inductive assumptions $\eqref{eq:BBad}_{a-1,d+1,l}$ and $\eqref{eq:BBad}_{a-2,d+2,l}$. Hence, we have proved $\eqref{eq:BBad}_{a,d,l}$.
\end{proof}
 
 \begin{remark}
The following multiplicative formula in $\U^-(\sl_3)$, for $i=1, 2$, 
 \begin{align*}
 F_i F_{i}^{(a,l+a+d,l)}
     =& \; [l+1+d]  F_{i}^{(a,l+a+d+1,l)}+[l+1] F_{i}^{(a-1,l+a+d+1,l+1)}
     \end{align*}
can be obtained from \eqref{eq:BBad} by taking the leading terms. Here we have denoted 
\[
F^{(a,b,c)}_{2}= F_{1}^{(a)}F_{2}^{(b)}F_{1}^{(c)}, \qquad
F^{(a,b,c)}_{1}= F_{2}^{(a)}F_{1}^{(b)}F_{2}^{(c)}.
\]
 \end{remark}
 
 \begin{corollary}\label{cor:B2B1}
 The following identity holds, for $t,l\geq 0$:
 \begin{align}\label{eq:B2B1} 
     B_{1}^{(t,l+t,l)} =B_{2}^{(l,l+t,t)}.
 \end{align}
 \end{corollary}
 
 \begin{proof}
 We prove \eqref{eq:B2B1} by induction on $t$. For $t=0$, \eqref{eq:B2B1} is trivial. Denote $\eqref{eq:B2B1}$ by $\eqref{eq:B2B1}_{t,l}$ in this proof. For $t>0$, we show that $\eqref{eq:B2B1}_{t,l}+\eqref{eq:B2B1}_{t,l+1}\Rightarrow \eqref{eq:B2B1}_{t+1,l}$ for any $l\geq 0$. By \eqref{eq:BBa}, we have
 \begin{align*}
     B_2 B_1 B_{1}^{(t,l+t,l)}
     &= [l+1][t+1]  B_{1}^{(t+1,l+t+1,l)}+[l+1][t] B_{1}^{(t,l+t+1,l+1)}
     \\
     &\quad - [l+1][t] B_{1}^{(t,l+t,l)} \{ k_1;l-t-1/2 \}.
 \end{align*}
 On the other hand, by \eqref{eq:BBad}, we have
 \begin{align*}
     B_2 B_1 B_{2}^{(l,l+t,t)}&=[l+1] B_2 B_{2}^{(l+1,l+t,t)}
     \\
     &= [l+1][t]  B_{2}^{(l+1,l+t+1,t)}+[l+1][t+1] B_{2}^{(l,l+t+1,t+1)}
     \\
     &\quad - [l+1][t] B_{2}^{(l,l+t,t)} \{ k_1;l-t-1/2 \}.
 \end{align*}
 By the induction hypothesis $\eqref{eq:B2B1}_{t,l} $ and $\eqref{eq:B2B1}_{t,l+1}$, the above two identities clearly imply $\eqref{eq:B2B1}_{t+1,l} $, as desired.
 \end{proof}

\begin{corollary} 
The following identities hold, for $t, l \ge 0$:
 \begin{align}
     B_{2}^{(l+1,l+t,t)}
     &= B_{1}^{(t,l+t+1,l)} +B_{1}^{(t-1,l+t+1,l+1)}
     - B_{1}^{(t-1,l+t,l)} \{ k_1;l-t-1/2 \},
     \label{eq:BBb} 
     \\
     B_{1}^{(t,l+t,l)} B_2
     &= [l+1] \big( B_{2}^{(l,l+t+1,t)} + B_{2}^{(l+1,l+t+1,t-1)}
     - \{ k_1;l-t-1/2 \} B_{2}^{(l,l+t,t-1)} \big).
     \label{eq:BBc} 
 \end{align}
\end{corollary}

\begin{proof}
Substituting $B_{1}^{(t,l+t,l)} =B_{2}^{(l,l+t,t)}$ (see \cref{cor:B2B1}) into the left-hand side of the identity \eqref{eq:BBa}, we immediately obtain \eqref{eq:BBb}.
Applying the anti-involution $\sigma_\tau$ in Lemma~\ref{lem:inv} to both sides of \eqref{eq:BBb}, we then obtain \eqref{eq:BBc}. 
\end{proof}

\begin{corollary}
The following identity holds,  for $t, l \ge 0$:
\begin{align}
B_{1}^{(t,l+t,l)} B_1
     &= [t+1] \big(B_{1}^{(t,l+t+1,l)} + B_{1}^{(t+1,l+t+1,l-1)}
     - \{ k_1^{-1};t-l-1/2 \} B_{1}^{(t,l+t,l-1)} \big).
     \label{eq:BBe} 
\end{align}
\end{corollary}

\begin{proof}
Applying the involution $\widehat\tau$ (which exists only for balanced parameters) defined in \cref{lem:inv} to \eqref{eq:BBc}, we obtain 
\begin{align} 
B_{2}^{(t,l+t,l)} B_1
     &= [l+1] \big( B_{1}^{(l,l+t+1,t)} + B_{1}^{(l+1,l+t+1,t-1)}
      -  \{ k_1^{-1};l-t-1/2 \} B_{1}^{(l,l+t,t-1)} \big).
     \label{eq:BBd} 
\end{align}

Thanks to $B_{2}^{(t,l+t,l)} =B_{1}^{(l,l+t,t)}$, we obtain \eqref{eq:BBe} from \eqref{eq:BBd} by switching $t,l$.

The formula \eqref{eq:BBd} can be alternatively obtained by applying $\sigma_\imath$ to \eqref{eq:BBa}.
\end{proof}

\subsection{An inversion formula}

In this subsection, we consider the quantum symmetric pair $\big(\U(\sl_3),\Ui(\sl_3)\big)$ with a general parameter $(\vs_1,\vs_2)$. Denote by $\omega_1,\omega_2$ the fundamental weights for $\sl_3$, and denote the irreducible finite-dimensional $\U(\sl_3)$-module $L(m\omega_1+n\omega_2)$ by $L(m,n)$, for $m,n\in \N$.
 

The canonical basis of $\U^-(\sl_3)$ is given by
\begin{align*}
  \cB(\infty)=  \{F_1^{(a)}F_{2}^{(b)}F_1^{(c)},F_2^{(a)}F_1^{(b)}F_2^{(c)}\mid b\geq a+c\},
\end{align*}
where 
\begin{align}
\label{FFF}
    F_1^{(a)}F_{2}^{(b)}F_1^{(c)}=F_2^{(c)}F_{1}^{(b)}F_2^{(a)}
\end{align}
are identified if $b=a+c$, cf. \cite[14.5.4]{Lus93}. 

\begin{lemma}
\label{lem:qBF}
 We have, for $j=1,2, n\in \N$,
 \begin{align}
     B_j^{(n)}=\sum_{t=0}^{n} q_j^{t(n-t)} \vs_j^{n-t}  F_j^{(t)} (E_{\tau j}K_j^{-1})^{(n-t)},
     \qquad
     B_j^{(n)}\eta=F_j^{(n)}\eta.
 \end{align}
\end{lemma}

\begin{proof}
 Note that $B_j=F_j +\vs_j  E_{\tau j} K_j'$ and 
 $F_j (\vs_jE_{\tau j} K_j^{-1}) =q_j^{-2} (\vs_jE_{\tau j} K_j^{-1}) F_j $. The first formula now follows by the $q$-binomial identity (cf. \cite[1.3.5]{Lus93}), and the second formula follows from the first one. 
\end{proof}

\begin{lemma}
\label{lem:BBFF}
For $b\geq c$, we have
\begin{align*}
    B_2^{(b)} B_1^{(c)}
    =& \; \sum_{l=0}^{b}\sum_{t=0}^c \sum_{s=0}^{\min(b-l,t)}\vs_2^{b-l}q^{-(b-l)(b-3l+1+2t-2s)/2}q^{t(c-t)}\frac{1}{[s]!}
    F_2^{(l)}  F_1^{(t-s)}K_2^{l-b}\times
    \\
   & \times \big(\prod_{j=1}^s [K_1; s+l-b-t+j]\big)E_1^{(b-l-s)} (\vs_1 E_2 K_1^{-1})^{(c-t)},
   \\
    B_2^{(b)} B_1^{(c)}\eta
    =& \; \sum_{l=b-c}^{b}  \vs_2^{b-l} q^{-(b-l)(-b-l+1+2c+2n)/2}  \qbinom{m-c+b-l}{b-l}
    F_2^{(l)}  F_1^{(c-b+l)} \eta.
\end{align*}
\end{lemma}

\begin{proof}
    Follows by applying \cref{lem:qBF} and Lemma~\ref{lem:EF}.
\end{proof}

\begin{proposition}
\label{prop:BBBFFF}
 For $b\geq a+c$, we have
 \begin{align}
 \notag
   B_1^{(a)} B_2^{(b)} B_1^{(c)}\eta
    =& \; \sum_{t=0}^a\sum_{l=b-c}^{ b } \vs_1^{a-t} \vs_2^{b-l} q^{-(b-l)(-b-l+1+2c+2n)/2 -(a-t)(2m+4b-4c-2l-a-t+1)/2}
    \\ \label{eq:BBBFFF}
    &\qbinom{n+a-t-b+c}{a-t}\qbinom{m-c+b-l}{b-l}
    F_1^{(t)} F_2^{(l-a+t)}  F_1^{(c-b+l)} \eta.
 \end{align} 
 Equivalently, we have
 \begin{align}
 \notag
   B_1^{(a)} B_2^{(b)} B_1^{(c)}\eta
    =& \; \sum_{x=0}^a\sum_{y=0}^{c} \vs_1^{x} \vs_2^{y} q^{-y(y+1+2n+2c-2b)/2 -x(x+1+2m+2b-4c-2a+2y)/2}
    \\  \label{eq:BBBFFFxy}
    &\qbinom{n-b+c+x}{x}\qbinom{m-c+y}{y}
    F_1^{(a-x)} F_2^{(b-x-y)}  F_1^{(c-y)} \eta.
 \end{align} 
\end{proposition}

\begin{proof}
Apply $B_1^{(a)}$ to the second formula in \cref{lem:BBFF}. Then apply \cref{lem:qBF} and \cref{lem:EF} and simplify.
\end{proof}


It follows by \eqref{eq:BBBFFF} that $B_1^{(a)} B_2^{(b)} B_1^{(c)}\eta$ is a linear combination of $F_1^{(a')} F_2^{(b')} F_1^{(c')}\eta$ for some $a', b', c'$ such that $b'-a'-c'=b-a-c, c'\leq c, a'\leq a$. Therefore, for $b-a-c=r$ fixed, we can reformulate \eqref{eq:BBBFFF} as
\begin{align}
 \notag
   B_1^{(a)} B_2^{(a+c+r)} B_1^{(c)}\eta
    =& \; \sum_{t=0}^a\sum_{l=0}^{ c } \vs_1^{a-t} \vs_2^{c-l} q^{-(c-l)(c-l+1+2n-2a-2r)/2 -(a-t)(2m+2r-2l+a-t+1)/2}
    \\ \label{eq:BBBFFFr}
    &\qbinom{n-r-t}{a-t}\qbinom{m -l}{c-l}
    F_1^{(t)} F_2^{(t+l+r )}  F_1^{(l)} \eta.
 \end{align} 

It is remarkable that the formula \eqref{eq:BBBFFFr} admits a close inverse formula below.

\begin{proposition}
\label{prop:FB}
We have, for $a,c,r\in \N$,
\begin{align}\notag 
    & F_1^{(a)} F_2^{(a+c+r)} F_1^{(c)}\eta
    \\\notag
    &= \;  \sum_{t=0}^a\sum_{l=0}^{ c }(-1)^{a-t+c-l} 
     q^{ (a-t)(-2m-3-2r+2l+a-t)/2+(c-l)(-2n -3+2r+2a +c -l )/2}
    \\
    &\qquad \times \vs_1^{a-t} \vs_2^{c-l}\qbinom{n-r-t}{a-t}\qbinom{m -l}{c-l}
    B_1^{(t)} B_2^{(t+l+r )} B_1^{(l)} \eta.
    \label{eq:FFFBBB}
\end{align}
Equivalently, we have, for $a,c,r\in \N$,
\begin{align}
\notag 
    &F_1^{(a)} F_2^{(a+c+r)} F_1^{(c)}\eta
    \\
    \notag
    &= \;  \sum_{x=0}^a\sum_{y=0}^{ c }(-1)^{x+y} 
     q^{x(-2m-3-2r+2c-2y+x)/2 +y(-2n -3+2r+2a +y )/2}
    \\
    &\quad \times \vs_1^{x} \vs_2^{y}\qbinom{n-r-a+x}{x}\qbinom{m -c+y}{y}
    B_1^{(a-x)} B_2^{(a+c-x-y+r )} B_1^{(c-y)} \eta.
    \label{eq:FFFBBBxy}
\end{align}
\end{proposition}

\begin{proof}
Recall a standard $q$-binomial identity 
\begin{align}
\label{eq:qbinom2}
\sum_{t=0}^x (-1)^t q^{t-tx} \qbinom{x}{t}=\sum_{t=0}^x (-1)^t q^{-t+tx} \qbinom{x}{t}=0,\qquad \forall x\geq 1.
\end{align} 
Plugging \eqref{eq:FFFBBB} into RHS \eqref{eq:BBBFFFr}, we obtain
\begin{align*}
& \text{RHS }\eqref{eq:BBBFFFr}
 \\
 = & \sum_{t=0}^a\sum_{l=0}^{ c } \sum_{i=0}^t\sum_{j=0}^{l } (-1)^{t-i+l-j} \vs_1^{a-i} \vs_2^{c-j} q^{-(c-l)(c-l+1+2n-2a-2r)/2}
    \\
    & \quad \times q^{-(a-t)(2m+2r-2l+a-t+1)/2}q^{ (t-i)(-2m-3-2r+2j+t-i)/2}    q^{(l-j)(-2n -3+2r+2t +l -j )/2}
    \\
    &\quad \times \qbinom{n-r-i}{t-i}\qbinom{m -j}{l-j}\qbinom{n-r-t}{a-t}\qbinom{m -l}{c-l}     B_1^{(i)} B_2^{(i+j+r )} B_1^{(j)} \eta
    \\
 = &\sum_{i=0}^a\sum_{j=0}^{ c } q^{-(a-i)(m+r)-(c-j)(n-r)+ac-ij}\vs_1^{a-i} \vs_2^{c-j}\times
 \\
    &\times\Big(\sum_{t=i}^a (-1)^{t-i }  q^{-(a-t)( a-t+1)/2}q^{ (t-i)( -3  +t-i)/2}\qbinom{a-i}{t-i}\Big)\times
    \\
    &\times\Big(\sum_{l=j}^{c}(-1)^{l-j} q^{-(c-l)(c-l+1 )/2}   q^{(l-j)( -3 +l -j )/2}\qbinom{c -j}{l-j}\Big)\times
    \\
    &\times \qbinom{n-r-i}{a-i}\qbinom{m -j}{c-j}     B_1^{(i)} B_2^{(i+j+r )} B_1^{(j)} \eta.
\end{align*}
Using \eqref{eq:qbinom2}, we have
\begin{align}
\begin{split}
\sum_{t=i}^a (-1)^{t-i }  q^{-(a-t)( a-t+1)/2}q^{ (t-i)( -3  +t-i)/2}\qbinom{a-i}{t-i}=\delta_{a,i},
\\
\sum_{l=j}^{c}(-1)^{l-j} q^{-(c-l)(c-l+1 )/2}   q^{(l-j)( -3 +l -j )/2}\qbinom{c -j}{l-j}=\delta_{c,j}.
\end{split}
\end{align}
Then, by these two identities, the only nonzero term in the above formula of RHS \eqref{eq:BBBFFFr} is $ B_1^{(a)} B_2^{(a+c+r )} B_1^{(c)} \eta$, which is exactly  LHS \eqref{eq:BBBFFFr}. 
\end{proof}
  

By symmetry and \eqref{eq:FFFBBBxy}, we have
\begin{align}
\notag 
    & F_2^{(a)} F_1^{(a+c+r)} F_2^{(c)}\eta
    \\ \notag 
    &= \;  \sum_{x=0}^a\sum_{y=0}^{ c }(-1)^{x+y} 
     q^{x(-2n-3-2r+2c-2y+x)/2 +y(-2m -3+2r+2a +y )/2}
    \\
    &\quad \times \vs_2^{x} \vs_1^{y}\qbinom{m-r-a+x}{x}\qbinom{n-c+y}{y}
    B_2^{(a-x)} B_1^{(a+c-x-y+r )} B_2^{(c-y)} \eta. 
    \label{eq:FB2}
\end{align}

\subsection{A monomial basis}

\begin{lemma}
\label{lem:CBA2}
The canonical basis of the simple $\U(\sl_3)$-module $L(m,n)$ is 
\[
\{ F_1^{(a)} F_2^{(b)} F_1^{(c)}\eta \mid
 c \le m, \;
 a\le b-c \le n
\}
\cup\{
F_2^{(c)} F_1^{(b)} F_2^{(a)}\eta  \mid
 a \le n, \;
 c\le b-a \le m \},
 \]
modulo the identification $F_1^{(a)} F_2^{(b)} F_1^{(c)}\eta =F_2^{(c)} F_1^{(b)} F_1^{(a)} \eta$ when $b=a+c$.
\end{lemma}

\begin{proof}
This should be well known; since we did not find a good reference to the explicit indexing set as formulated, we give a self-contained proof. 
For $j=1,2$ and $k\in \Z$, we set $\cB_{j,k}=\cB(\infty)\cap \U^-F_j^{(k)}$. By \cite[14.4.11]{Lus93}, it suffices to show that: for $b\geq a+c$,
\begin{align}
\label{CB-mod}
& F_1^{(a)} F_2^{(b)} F_1^{(c)}\in \cB_{1,m+1}\cup \cB_{2,n+1} \text{ if and only if }
c>m \text{ or } b-c>n ,
\\
&F_2^{(c)} F_1^{(b)} F_2^{(a)}\in \cB_{1,m+1}\cup \cB_{2,n+1} \text{ if and only if }
a>n \text{ or } b-a>m .
\end{align}
We prove the statement \eqref{CB-mod}; the second statement can be obtained by symmetry.

Clearly, if $c>m$, then $F_1^{(a)} F_2^{(b)} F_1^{(c)}\in \cB_{1,m+1}$. If $b-c>n$, using the higher order Serre relation in \cite[Corollary 7.1.7]{Lus93}, we can write  $F_2^{(b)} F_1^{(c)}$ as a linear combination of $F_2^{(s)} F_1^{(c)} F_2^{(b-s)} $ for $0\leq s\leq c$ and this implies that $F_2^{(b)} F_1^{(c)} \in \cB_{2,n+1}$.

It remains to show the ``only if'' direction in \eqref{CB-mod}. Suppose that $F_1^{(a)} F_2^{(b)} F_1^{(c)}\in \cB_{1,m+1}\cup \cB_{2,n+1} $, for some $a,b,c$ with $c\leq m$ and $a\leq b-c\leq n$. Since $ \cB_{j,k}\subset  \cB_{j,l}$ for $k\leq l$, we have $F_1^{(a)} F_2^{(b)} F_1^{(c)}\in \cB_{1,c+1}\cup \cB_{2,b-c+1} $. By \cite[14.4.11]{Lus93}, this implies that $F_1^{(a)} F_2^{(b)} F_1^{(c)}\eta'=0$, where $\eta'$ denotes the highest weight vector for $L(c,b-c)$.

However, we claim that $F_1^{(a)} F_2^{(b)} F_1^{(c)}\eta' \neq 0$ in $L(c,b-c)$ and this leads to a contradiction. We write $T_{w}$ for Lusztig symmetry $T'_{w,-1}$. Note that $F_2^{(b)} F_1^{(c)}\eta'$ is the extremal weight vector $ T_{s_2 s_1}\eta'$ in $L(c,b-c)$ (cf. \cite{Lus93}); in particular, $F_2^{(b)} F_1^{(c)}\eta'$ is nonzero. Moreover, we have 
\[
E_1 T_{s_2 s_1}\eta'=T_{s_2 s_1}\big( T_{s_2 s_1}^{-1}(E_1)\eta'\big)=0.
\]
i.e., $T_{s_2 s_1}\eta'$ is a highest weight vector for the action of $\langle E_1, F_1, K_1^{\pm1}\rangle\cong \U(\sl_2)$. Since $K_1  T_{s_2 s_1}\eta'=q^{b-c} T_{s_2 s_1}\eta'$, we conclude that $F_1^{(a)} F_2^{(b)} F_1^{(c)}\eta'$ is nonzero for any $a\leq b-c$, as desired.
\end{proof}

\begin{theorem}
\label{thm:qsibasis}
Let $\eta$ be a highest weight vector of $L(m,n)$. For any $b\ge a+c$, $B_1^{(a)}B_{2}^{(b)}B_1^{(c)}\eta=0$ if and only if $F_1^{(a)}F_{2}^{(b)}F_1^{(c)}\eta=0$. Furthermore, the set
\begin{align*}
\{ B_1^{(a)} B_2^{(b)} B_1^{(c)}\eta \mid
 c \le m, \;
 a\le b-c \le n
\}
\cup\{
B_2^{(c)} B_1^{(b)} B_2^{(a)}\eta  \mid
 a \le n, \;
 c\le b-a \le m \},
 %
\end{align*}
forms a basis for $L(m,n)$, modulo the identification $B_1^{(a)}B_{2}^{(b)}B_1^{(c)}\eta =B_2^{(c)}B_1^{(b)}B_2^{(a)}\eta$ for $b= a+c.$ 
\end{theorem}

\begin{proof}
A nonzero term $F_1^{(a)}F_{2}^{(b)}F_1^{(c)}\eta$ appears as the leading term (of lowest weight) in the expression \eqref{eq:BBBFFF} of $B_1^{(a)}B_{2}^{(b)}B_1^{(c)}\eta$, forcing the latter to be nonzero. By \cref{lem:CBA2}, $F_1^{(a)}F_{2}^{(b)}F_1^{(c)}\eta\neq0$ if and only if $c\leq m$ and $ b-c\leq n $. Hence, it remains to show that, if $c>m$ or $b-c> n$, then $B_1^{(a)}B_{2}^{(b)}B_1^{(c)}\eta=0$.

Assume first that $c>m$.  Since $F_1^{(k)}\eta=0$ for $k>m$, any nontrivial summand on the right-hand side \eqref{eq:BBBFFF} of the expansion of $B_1^{(a)}B_{2}^{(b)}B_1^{(c)}\eta$ must be of the form
\begin{align*}
    q^{*}\qbinom{n+a-t-b+c}{a-t}\qbinom{m-c+b-l}{b-l}
    F_1^{(t)} F_2^{(l-a+t)}  F_1^{(c-b+l)} \eta
\end{align*}
for some $l$ such that $ m-c+b-l\geq 0$. This inequality together with $m-c+1\leq 0$ implies that $\qbinom{m-c+b-l}{b-l}=0$. Hence $B_1^{(a)}B_{2}^{(b)}B_1^{(c)}\eta =0$.

Assume now that $b-c> n$. Since $F_2^{(k_1)}F_1^{(k_2)}\eta=0$ for $k_1-k_2>n$,  any nontrivial summand on the right-hand side \eqref{eq:BBBFFF} of the expansion of $B_1^{(a)}B_{2}^{(b)}B_1^{(c)}\eta$ must be of the form \begin{align*}
    q^{*}\qbinom{n+a-t-b+c}{a-t}\qbinom{m-c+b-l}{b-l}
    F_1^{(t)} F_2^{(l-a+t)}  F_1^{(c-b+l)} \eta
\end{align*}
for some $t$ such that $n-b+c+a-t\geq 0$. This inequality together with $n-b+c+1\leq0$ implies that $\qbinom{n+a-t-b+c}{a-t}=0$. Hence $B_1^{(a)}B_{2}^{(b)}B_1^{(c)}\eta =0$.

By \cref{lem:BBFF,prop:BBBFFF}, the transition matrix between the set specified in the theorem and the canonical basis for $L(m,n)$ is upper triangular with all diagonal entries equal to 1. Hence this set forms a basis for $L(m,n)$.

Finally, the identity $B_1^{(a)}B_{2}^{(b)}B_1^{(c)}\eta =B_2^{(c)}B_1^{(b)}B_2^{(a)}\eta$ for $b= a+c$ follows by \eqref{eq:B2B1}. 
\end{proof}

We refer to the basis in \cref{thm:qsibasis} {\em a monomial basis} for $L(m,n)$. In case $\vs_1\vs_2=q$, there exists a bar involution $\psi^\imath$ on $\Ui$ and $L(m,n)$ compatible with each other, and the monomial basis elements are $\psi^\imath$-invariant. 

\section{Braid group operators of quasi-split type on integrable $\Ui$-modules}
\label{sec:quasisplit}

In this section, we fix $i\in \wI$ such that $c_{i,\tau i}=-1$. We introduce linear operators  $\TT'_{i,-1},\TT''_{i,+1}$ (which are shown to be mutually inverse) via explicit formulas on any integrable $\Ui$-modules and show that they are compatible with corresponding symmetries on $\Ui$ given in \cite{WZ23, Z23}. We further show that $\TT'_{i,-1},\TT''_{i,+1}$ specialize to the corresponding ones {\em loc. cit.} on integrable $\U$-modules by relating to Lusztig symmetries.

\subsection{Rank one formulas of quasi-split type}

Recall $c_{i,\tau i}=-1$. Let $\bvs=(\vs_i)_{i\in \wI}$ be an arbitrary parameter and $M\in \mathcal{C}$ be an integrable $\Ui$-module. 

We define linear operators $\TT'_{i,-1},\TT''_{i,+1}$ on $M$ by letting 
\begin{align} 
\label{eq:ibraid3}
    \TT'_{1,-1}(v) &= q_i^{\la_{i,\tau}^2/2} k_i^{-\la_{i,\tau}}\sum_{t\ge 0,l\ge 0} (-1)^{t+l}  q_i^{ -\frac{(t-l)^2}{2}-\frac{t+l}{2}} k_i^{t-l} \vs_i^{\frac{\la_{i,\tau}}{2}-t } \vs_{\tau i}^{-\frac{\la_{i,\tau}}{2}-l} B^{(t)}_{\tau i} B^{(t+l)}_{i}B^{(l)}_{\tau i}v,
    \\
\label{eq:ibraid3'}
    \TT''_{1,+1}(v) &=q_i^{-\la_{i,\tau}^2/2} k_i^{\la_{i,\tau}}\sum_{t\ge 0,l\ge 0} (-1)^{t+l}  q_i^{ \frac{(t-l)^2}{2}+\frac{3(t+l)}{2}} k_i^{l-t} \vs_i^{-\frac{\la_{i,\tau}}{2}-l} \vs_{\tau i}^{\frac{\la_{i,\tau}}{2}-t} B^{(t)}_{\tau i} B^{(t+l)}_{i}B^{(l)}_{\tau i}v,
\end{align}
for any $v\in M_{\ov{\lambda}}$, where $\la_{i,\tau}=\langle h_i-h_{\tau i},\ov{\lambda}\rangle$. The summations on the right-hand sides of \eqref{eq:ibraid3}--\eqref{eq:ibraid3'} are well-defined, as only finitely many terms are nonzero since $M$ is integrable. 

We introduce the shorthand notation, for $a,b,c \in \N$,  
\[
B^{(a,b,c)}_{i}:= B_{\tau i}^{(a)}B_{ i}^{(b)}B_{\tau i}^{(c)},
\qquad
F^{(a,b,c)}_{i}:= F_{\tau i}^{(a)}F_{ i}^{(b)}F_{\tau i}^{(c)}.
\]
In addition, we set $B^{(a,b,c)}_{i}=0$ if one of $a,b,c$ is negative. 
We introduce the following element in some completion of $\Ui$, for $\nu\in \Z$:
\begin{align}
     \label{eq:gnu}
    g_{\nu,\bvs}(k_i, B_i, B_{\tau i}) 
    &= q_i^{\nu^2/2} k_i^{-\nu}\sum_{t\ge 0,l\ge 0} (-1)^{t+l}  q_i^{ -\frac{(t-l)^2}{2}-\frac{t+l}{2}} k_i^{t-l} \vs_i^{\frac{\nu}{2}-t } \vs_{\tau i}^{-\frac{\nu}{2}-l} B^{(t,t+l,l)}_{i},
\\
\label{eq:hnu}
    h_{\nu,\bvs} (k_i,B_i,B_{\tau i}) 
    &=   q_i^{-\nu^2/2} k_i^{\nu}\sum_{t\ge 0,l\ge 0} (-1)^{t+l}  q_i^{ \frac{(t-l)^2}{2}+\frac{3(t+l)}{2}} k_i^{l-t} \vs_i^{-\frac{\nu}{2}-l} \vs_{\tau i}^{\frac{\nu}{2}-t} B^{(t,t+l,l)}_{i}.
\end{align}
We sometimes drop the subscript $\bvs$  in \eqref{eq:gnu}--\eqref{eq:hnu} to write $g_{\nu }, h_{\nu }$ when there is no ambiguity.

By \eqref{eq:ibraid3}, for $v\in M_{\ov{\lambda}}$ and $\nu= \langle h_i-h_{\tau i},\ov{\lambda}\rangle $, we have 
\[
\TT'_{i,-1}(v)=g_\nu(k_i, B_i, B_{\tau i})  v,\qquad \TT''_{i,+1}(v) =h_\nu(k_i, B_i, B_{\tau i})  v.
\]

Recall from Lemma~\ref{lem:inv} the anti-involution $\sigma_\tau$ and the involution $\htau$ of $\Ui$.

\begin{lemma}
\label{lem:sigmatau}
For an arbitrary parameter $\bvs$, the element $g_\nu(k_i, B_i, B_{\tau i}) $ is fixed by $\sigma_\tau$.
For any balanced parameter $\bvs$, we have
$\htau  \big(g_\nu(k_i, B_i, B_{\tau i})\big)=  g_{-\nu}(k_i, B_i, B_{\tau i}).$ 
\end{lemma}

\begin{proof} 
By \cref{lem:inv}, $\sigma_\tau$ fixes $k_i$ and sends $B^{(t,t+l,l)}_{i}$ to $B^{(l,t+l,t)}_{\tau i}$. By \cref{cor:B2B1}, $B^{(l,t+l,t)}_{\tau i}=B^{(t,t+l,l)}_{i}$. Hence, $\sigma_\tau$ fixes $B^{(t,t+l,l)}_{i}$ and $k_i$. Note also that $k_i$ commutes with $B^{(t,t+l,l)}_{i}$. Thus, using the formula \eqref{eq:gnu}, it is clear that $\sigma_\tau$ fixes $g_\nu(k_i, B_i, B_{\tau i}) $.

For any balanced parameter $\bvs$, the formula \eqref{eq:gnu} simplifies to
\begin{align*} 
    g_{\nu}(k_i, B_i, B_{\tau i}) 
    &= q_i^{\nu^2/2} k_i^{-\nu}\sum_{t\ge 0,l\ge 0} (-1)^{t+l}  q_i^{ -\frac{(t-l)^2}{2}-\frac{t+l}{2}} k_i^{t-l} \vs_i^{ -t-l }   B^{(t,t+l,l)}_{i}.
\end{align*}
By \cref{lem:inv} and \cref{cor:B2B1}, $\htau (k_i)=k_i^{-1}$ and $\htau\big(B^{(t,t+l,l)}_{i}\big)=B^{(t,t+l,l)}_{\tau i}=B^{(l,t+l,t)}_{i}$. Applying $\htau$ to the above formula of $g_{\nu}(k_i, B_i, B_{\tau i})$, we obtain
\begin{align*}
 \htau\big( g_{\nu}(k_i, B_i, B_{\tau i}) \big)
    &= q_i^{\nu^2/2} k_i^{\nu}\sum_{t\ge 0,l\ge 0} (-1)^{t+l}  q_i^{ -\frac{(t-l)^2}{2}-\frac{t+l}{2}} k_i^{-t+l} \vs_i^{ -t-l }   B^{(l,t+l,t)}_{i}
    \\
    &=q_i^{\nu^2/2} k_i^{\nu}\sum_{t\ge 0,l\ge 0} (-1)^{t+l}  q_i^{ -\frac{(t-l)^2}{2}-\frac{t+l}{2}} k_i^{t-l} \vs_i^{ -t-l }   B^{(t,t+l,l)}_{i}
    \\
    &=g_{-\nu}(k_i, B_i, B_{\tau i}).
\end{align*}
The lemma is proved. 
\end{proof}

\subsection{Compatibility with relative braid group symmetries}

Recall $c_{i,\tau i}=-1$. 
\begin{lemma}
 \label{lem:qsmu}
For $\mu\in Y^\imath$, we have
\begin{align*}
    \langle \mu,\alpha_i+\alpha_{\tau i}\rangle &=0 ,
    \qquad\text{ and }
    \quad
\bs_i\mu =\mu.
\end{align*}
\end{lemma}

\begin{proof}
The proof for $\langle \mu,\alpha_i+\alpha_{\tau i}\rangle=0$ is the same as Lemma~\ref{lem:diagmu}. The second identity follows by a direct computation: 
$\bs_i\mu=s_i s_{\tau i} s_i(\mu)=\mu-\langle \mu,\alpha_i+\alpha_{\tau i}\rangle(h_i+h_{\tau i}) =\mu.$
\end{proof}

The proof of the next theorem will be reduced to two key identities, which will then be verified in the subsequent subsections (see \cref{thm:qs-rk1,thm:qs-rktwo}).

\begin{theorem}
\label{thm:qs-split}
Let $i\in \wI$ be such that $c_{i,\tau i}=-1$. We have, for any $x\in \Ui,v\in M,$
\begin{align}\label{eq:qs-split}
&\TT'_{i,-1}(x v)=\TT'_{i,-1}(x) \TT'_{i,-1}(v).
\end{align}
\end{theorem}

\begin{proof}
By similar arguments as in the proof of Theorem~\ref{thm:split}, it suffices to verify \eqref{eq:qs-split} when $x$ runs over a generating set of $\Ui$ and for $v\in M_{\ov{\lambda}}$ with $\ov{\lambda}\in X_\imath$. Set $\nu= \langle h_i-h_{\tau i},\ov{\lambda}\rangle$.

We prove \eqref{eq:qs-split} for $x=K_\mu\; (\mu\in Y^\imath)$. By  \cref{thm:braid-iQG}, \cref{prop:rkone} and Lemma~\ref{lem:qsmu}, $\TT'_{i,\bvs ,-1}(K_\mu)=K_\mu$. 
It remains to show that $\TT'_{i,-1}(K_\mu v)=K_\mu \TT'_{i,-1}(v) $. Since $\mu\in Y^\imath$, by \cref{lem:qsmu} we have
\[
K_\mu B_{\tau i} K_{\mu}^{-1}
= q^{-\langle \mu,\alpha_{\tau i}\rangle}B_{\tau i} 
=q^{\langle \mu,\alpha_{ i}\rangle}B_{\tau i},
\qquad
K_\mu B_{ i} K_{\mu}^{-1}=q^{-\langle \mu,\alpha_{ i}\rangle}B_{ i}. 
\]
Thus, $K_\mu$ commutes with $ B^{(t)}_{\tau i} B^{(t+l)}_{i}B^{(l)}_{\tau i}$, and by \eqref{eq:ibraid3}, $K_\mu$ commutes with $\TT'_{i,-1}$ when acting on $M$ as desired.

We prove \eqref{eq:qs-split} for $x=B_j \, (j\in \I)$. Since $B_jv\in M_{\ov{\lambda-\alpha_j}}$, we have
\[
\TT'_{i,\bvs,-1} (B_j v)=g_{\nu-c_{ij}+c_{\tau i,j},\bvs}(k_i, B_i, B_{\tau i}) B_j v.
\]
Hence, it suffices to show that
\begin{align}
\label{eq:Tgnu}
 \TT'_{i,\bvs,-1}(B_j) g_{\nu,\bvs}(k_i, B_i,B_{\tau i})=g_{\nu-c_{ij}+c_{\tau i,j},\bvs}(k_i,B_i,B_{\tau i}) B_j.
\end{align}
Let $\bvs,\bvs'$ be two parameters. By Lemma~\ref{lem:iso-parameter} and the definition \eqref{eq:gnu}, we have 
\[
g_{\nu,\bvs' }(k_i, B_i,B_{\tau i})=\phi_{\bvs ,\bvs'} g_{\nu,\bvs}(k_i, B_i,B_{\tau i}),
\]
for any $\nu\in \Z$. By \cref{thm:braid-iQG}(b), we have $\TT'_{i,\bvs',-1}\phi_{\bvs ,\bvs'}(B_j)=\phi_{\bvs,\bvs'} \TT'_{i,\bvs ,-1}(B_j)$. Therefore, it suffices to verify \eqref{eq:Tgnu} for a special parameter. 

For $j=i$ or $\tau i$, \eqref{eq:Tgnu} will be established in Theorem~\ref{thm:qs-rk1} below for the parameter $\bvs$ such that $\vs_i=\vs_{\tau i}=-q_i^{1/2}$. For $j\neq i,\tau i$, \eqref{eq:Tgnu} will be proved for the distinguished parameter $\bvs_\dm$ in Theorem~\ref{thm:qs-rktwo} and \cref{cor:qs-rktwo}. 

This completes the proof of the desired identity \eqref{eq:qs-split}. 
\end{proof}

Recall $g_{\nu,\bvs_\star}$ and $h_{\nu,\bvs_\star}$ from  \eqref{eq:gnu}--\eqref{eq:hnu}.

 \begin{lemma}\label{lem:qs-bar}
 Set the parameter $\bvs =\bvs_\star$. For any $\nu\in\Z ,j\in \I$, we have 
 \begin{align*}
 \psi^\imath\big(g_{\nu,\bvs_\star}(k_i, B_i, B_{\tau i})\big)&= h_{\nu,\bvs_\star}(k_i, B_i, B_{\tau i}),
 \\
 \psi^\imath\big(\TT'_{i, -1}(B_j)\big)&=\TT''_{i, +1}(B_j).
 \end{align*}
 \end{lemma}
 
 \begin{proof}
 Recall from \S\ref{subsec:ibar} that $\vs_{\star,i}=q_i^{1/2}$ (since $c_{i,\tau i}=-1$) and the bar involution $\psi^\imath$ exists on $\Ui_{\bvs_\star}$. The first identity follows directly by the formulas \eqref{eq:gnu}--\eqref{eq:hnu}.

 The second identity follows from the explicit formulas for $\TT'_{i, -1}(B_j)$ and $\TT''_{i,+1}(B_j)$; see Proposition~\ref{prop:rkone} for $j=i$ and \cite[Theorem 3.7(iii)]{Z23} for $j\neq i$.
 \end{proof}

\begin{proposition}
\label{prop:quasisplit2}
Let $i\in \wI$ be such that $c_{i,\tau i}=-1$. We have, for any $x\in \Ui,v\in M,$
\begin{align}
\label{eq:qs-split2}
&\TT''_{i,+1}(x v)=\TT''_{i,+1}(x) \TT''_{i,+1}(v).
\end{align}
\end{proposition}

\begin{proof}
By the same arguments as in the first paragraph in the proof of Theorem~\ref{thm:split}, it suffices to prove \eqref{eq:qs-split2} when $x$ runs a generating set of $\Ui$ and for $v\in M_{\ov{\lambda}}$ with $\ov{\lambda}\in X_\imath$. 

For $x=K_\mu \, (\mu\in Y^\imath)$, the proof of \eqref{eq:qs-split2} is essentially the same as the proof of \eqref{eq:qs-split} in Theorem~\ref{thm:qs-split}.

We prove \eqref{eq:qs-split2} for $x=B_j \, (j\in \I)$. Using $\phi_{\bvs,\bvs'}$ and similar arguments in the proof of Theorem~\ref{thm:qs-split}, we can set the parameter to be  $\bvs_\star$. Applying $\psi^\imath$ to \eqref{eq:Tgnu} and using Lemma~\ref{lem:qs-bar}, we obtain that  
\begin{align}
\label{eq:Thnu}
 \TT''_{i,\bvs_\star,+1}(B_j) h_{\lambda_{i,\tau},\bvs_\star}(k_i, B_i,B_{\tau i})=h_{\lambda_{i,\tau}-c_{ij}+c_{\tau i,j},\bvs_\star}(k_i,B_i,B_{\tau i}) B_j.
\end{align}
By \eqref{eq:ibraid3'}, this implies that 
$\TT''_{i,+1}(B_j v)=\TT''_{i,+1}(B_j) \TT''_{i,+1}(v)$ as desired.
\end{proof}

\subsection{Identities in quasi-split rank one}

Recall $c_{i,\tau i}=-1$. In this subsection,  fixing a parameter $\bvs$ with $\vs_i=\vs_{\tau i}=-q_i^{1/2}$, we shall verify a key identity in the proof of Theorem~\ref{thm:qs-split}, for $x=B_{i},$ or $B_{\tau i}$. Set 
\begin{align}
     \label{eq:g}
\begin{split}
    g(k_i, B_i, B_{\tau i}) 
    &= \sum_{t\ge 0,l\ge 0}   q_i^{ -\frac{(t-l)^2}{2}-t-l } k_i^{t-l}  B^{(t,t+l,l)}_{i}.
    \end{split}
\end{align}

By definition in \eqref{eq:gnu}, we have \begin{align}
 \label{ggnu}
    g_\nu(k_i, B_i, B_{\tau i}) =q_i^{\nu^2/2} k_i^{-\nu}g(k_i, B_i, B_{\tau i}).
\end{align}


Recall from Proposition~\ref{prop:rkone} that
\begin{align*}
    \TT'_{i,-1}(B_i) =  q_i^{3/2} B_i k_i^{-1},
    \qquad
    \TT'_{i,-1}(B_{\tau i}) =  q_i^{3/2}  B_{\tau i} k_i.
\end{align*}

Let $\Ui_{i,\tau i}$ be the subalgebra of $\Ui$ generated by $B_i,B_{\tau i},k_i$. Since $c_{i,\tau i}=-1$, there is an algebra isomorphism $\Ui_{i,\tau i}\cong \Ui(\sl_3)$ such that $q_i\mapsto q, B_i\mapsto B_1,B_{\tau i}\mapsto B_2, k_i\mapsto k_1$. In this way, we are free to apply the results on the quantum symmetric pair $(\U(\sl_3), \Ui(\sl_3))$ in Section~\ref{sec:quasi basis}.

\begin{theorem}
\label{thm:qs-rk1}
Assume $c_{i,\tau i}=-1$. The element $g_\nu(k_i, B_i, B_{\tau i})$ in \eqref{eq:gnu} satisfies the following two identities for any $\nu\in \Z$,
\begin{align}
\label{eq:gnuB1}
    g_{\nu-3}(k_i, B_i, B_{\tau i})B_i &=\TT'_{i,-1}(B_i) g_\nu(k_i, B_i, B_{\tau i}),
    \\
    g_{\nu+3}(k_i, B_i, B_{\tau i})B_{\tau i} &=\TT'_{i,-1}(B_{\tau i}) g_\nu(k_i, B_i, B_{\tau i}).
\label{eq:gnuB2}
\end{align} 
\end{theorem}

\begin{proof}
 We prove \eqref{eq:gnuB1}; the other formula \eqref{eq:gnuB2} is obtained by applying $\sigma_\tau$ to \eqref{eq:gnuB1} and using Lemma~\ref{lem:sigmatau}. 

Using \eqref{eq:BBe}, we compute the LHS \eqref{eq:gnuB1} up to a simple factor as follows:
\begin{align*}
&k_i^3 q_i^{9/2}  g(k_i, B_i, B_{\tau i})  B_i
\\
=& \;  \sum_{t\ge 0,l\ge 0}   q_i^{ -\frac{(t-l)^2}{2}-t-l+9/2} k_i^{t-l+3} B^{(t,t+l,l)}_{i} B_i
\\
=& \;  \sum_{t\ge 0,l\ge 0} [t+1]_i  q_i^{ -\frac{(t-l)^2}{2}-t-l+9/2} k_i^{t-l+3} \big( B^{(t,t+l+1,l)}_{i} + B^{(t+1,t+l+1,l-1)}_{i} \big)
  \\
 &-\sum_{t\ge 0,l\ge 0} [t+1]_i  q_i^{ -\frac{(t-l)^2}{2}-t-l+9/2} k_i^{t-l+3}
 \{ k_i^{-1};t-l-1/2 \} B^{(t,l+t,l-1)}_{i}
 \\
=& \;  \sum_{t\ge 0,l\ge 0}  q_i^{ -\frac{(t-l-1)^2}{2}-t-2l+1}   k_i^{t-l+1} B^{(t,t+l+1,l)}_{i}
\\
=& \;  \sum_{t\ge 0,l\ge 0}  q_i^{ -\frac{(t-l+1)^2}{2}+t-4l+1}   k_i^{t-l+1} B^{(t,t+l+1,l)}_{i}.
\end{align*}

On the other side, using \eqref{eq:BBa}, we compute the RHS \eqref{eq:gnuB1} up to a simple factor as follows:
\begin{align*}
& q_i^{3/2}  B_i k_i^{-1} g(k_i, B_i, B_{\tau i})
\\
=& \;  q_i^{3/2} \sum_{t\ge 0,l\ge 0}   q_i^{ -\frac{(t-l)^2}{2}-t-l }  B_i k_i^{t-l-1} B^{(t,t+l,l)}_{i}
\\
=& \; q_i^{-3/2} \sum_{t\ge 0,l\ge 0}   q_i^{ -\frac{(t-l)^2}{2}+2t-4l }  k_i^{t-l-1}B_i B^{(t,t+l,l)}_{i}
\\
=& \; q_i^{-3/2} \sum_{t\ge 0,l\ge 0} [l+1]_i q_i^{ -\frac{(t-l)^2}{2}+2t-4l }  k_i^{t-l-1} \big( B^{(t,t+l+1,l)}_{i}+B^{(t-1,t+l+1,l+1)}_{i} \big)
\\
&-q_i^{-3/2} \sum_{t\ge 0,l\ge 0} [l+1]_i q_i^{ -\frac{(t-l)^2}{2}+2t-4l }  k_i^{t-l-1}\{ k_i;l-t+5/2 \} B_{i}^{(t-1,l+t,l)}
\\
=& \; \sum_{t\ge 0,l\ge 0}  q_i^{ -\frac{(t-l+1)^2}{2}+t-4l+1}  k_i^{t-l+1}
  B^{(t,t+l+1,l)}_{i}.
\end{align*}
Hence, combining the above computations gives us
\begin{align}
    \label{gB=Bg}
    k_i^3 q_i^{9/2} g(k_i,B_i, B_{\tau i})B_i - q_i^{3/2} B_i k_i^{-1 } g(k_i,B_i, B_{\tau i}) =0.
\end{align}
Using \eqref{ggnu} and \eqref{gB=Bg}, we have
\begin{align*}
    g_{\nu-3}& (k_i,B_i, B_{\tau i})B_i -q_i^{3/2} B_i k_i^{-1 } g_\nu(k_i,B_i, B_{\tau i})
    \\
    =& \; q_i^{(\nu-3)^2/2} k_i^{-\nu+3} g(k_i,B_i, B_{\tau i}) B_i -q_i^{\nu^2/2-3\nu} k_i^{-\nu} q_i^{3/2} B_i k_i^{-1 } g(k_i,B_i, B_{\tau i})
    \\
    =& \; q_i^{\nu^2/2-3\nu} k_i^{-\nu} \big(k_i^3 q_i^{9/2} g(k_i,B_i, B_{\tau i})B_i - q_i^{3/2} B_i k_i^{-1 } g(k_i,B_i, B_{\tau i})\big)
    =   0,
\end{align*}
whence \eqref{eq:gnuB1}.
\end{proof}


\begin{proposition}
The element $h_\nu(k_i, B_i, B_{\tau i})$ satisfies the following two identities, for any $\nu\in \Z$:
\begin{align}
\label{eq:hnuB1}
   h_{\nu-3}(k_i, B_i, B_{\tau i})B_i &=\TT''_{i,+1}(B_i) h_\nu(k_i, B_i, B_{\tau i}),
    \\
   h_{\nu+3}(k_i, B_i, B_{\tau i})B_{\tau i} &=\TT''_{i,+1}(B_{\tau i}) h_\nu(k_i, B_i, B_{\tau i}).
\label{eq:hnuB2}
\end{align}  
\end{proposition}

\begin{proof}
Follows by similar calculations and arguments as in the proof of Theorem~\ref{thm:qs-rk1}. We skip the details.
\end{proof}
\subsection{Identities in quasi-split rank two with parameter $\bvs_\dm$}

Recall $c_{i,\tau i}=-1$. In this subsection, we set the parameter to be the distinguished parameter $\bvs_\dm$, and thus $\vs_i=\vs_{\tau i}=-q_i^{-1/2}$. We shall verify a key identity needed for the proof of Theorem~\ref{thm:qs-split}.

For the parameter $\bvs_\dm$, the element $g_\nu(k_i, B_i, B_{\tau i})$ defined in \eqref{eq:gnu} specializes to 
\begin{align}
     \label{eq:gnupar}
\begin{split}
    g_\nu(k_i, B_i, B_{\tau i}) 
    &= q_i^{\nu^2/2}  k_i^{-\nu}\sum_{t\ge 0,l\ge 0} q_i^{ -\frac{(t-l)^2}{2} } k_i^{t-l} B^{(t,t+l,l)}_{i}
    \\
    &= q_i^{\nu^2/2} k_i^{-\nu}\sum_{t\ge 0,l\ge 0} q_i^{ -\frac{(t-l)^2}{2} } k_i^{t-l}   B^{(l,t+l,t)}_{\tau i}.
    \end{split}
\end{align}

Recall from \eqref{def:qsqsBij1}--\eqref{def:qsqsBij2} the elements $b_{i,\tau i, j;n,a,b,c}$; we shall denote them by a shorthand $b_{n,a,b,c}$ (with $j \neq i,\tau i$ fixed) in this subsection.
Write $\alpha=-c_{ij}$ and $\beta =-c_{\tau i,j}$. 

\begin{lemma}\label{lem:qsqsKM1}
Let $k\geq 0$. The following identities hold: 
\begin{align}
    b_{n,n\beta-x,n\alpha+n\beta,n\alpha} B_{i}^{(k)}
    &=\sum_{y=0}^x \sum_{u=0}^{n\alpha+n\beta} (-1)^y q_i^{(k-u)(n\beta-2x+y)-\frac{u(u+2k)}{2}-y} \qbinom{n\beta-x+y}{y}_i \times
    \notag \\
    &\qquad \times B_i^{(k-y-u)} b_{n,n\beta-x+y,n\alpha+n\beta-u,n\alpha} k_i^u, 
    \qquad \text{for } 0\leq x\leq n\beta,
\\
    b_{n,n\beta,n\alpha+n\beta ,n\alpha}B_{\tau i}^{(k)}&= 
    \sum_{r=0}^{k} q_i^{(k-r)n\alpha-kr-\frac{1}{2}r^2} 
    B_{\tau i}^{(k-r)}b_{n,n\beta-r,n\alpha+n\beta,n\alpha} k_i^{-r},
\\
    b_{n,0,n\alpha+n\beta-u,n\alpha} B_{\tau i}^{(k)}
    &=\sum_{v=0}^u \sum_{r=0}^{n\alpha} (-1)^v q_i^{(k-r)(n\alpha +n\beta-2u+v)-\frac{r(r+2k)}{2}-v}  \times
    \notag \\
    & \quad \times\qbinom{n\alpha+n\beta-u+v}{v}_i B_{\tau i}^{(k-v-r)} b_{n,0,n\alpha+n\beta-u+v,n\alpha-r} k_i^{-r},
    \notag\\
    &\qquad \text{for } 0\leq u\leq n\alpha+n\beta.
\end{align}
\end{lemma}

\begin{proof}
The proof for these identities using Lemmas~\ref{lem:qsrecur}--\ref{lem:qsvanish} is similar to the proof of \cref{lem:qsbij}, and hence will be skipped.
\end{proof}

Set, for $0\leq d \le a$ with $a>0$, 
\begin{align}
\label{Had}
H(a,d)=\sum_{t,l\geq 0} q_i^{-\frac{(t-l)^2}{2} +(l-d)a+(l-2t)d} k_i^{t-l+2d} \frac{ 1}{[l+1]_i\cdots[l+a]_i} B_{\tau i}^{(l,l+t+d,t)} B_{\tau i}^{a-d}.
\end{align}

\begin{lemma}
\label{lem:qsqsKM2}
The following identity holds:  
\begin{align}\label{eq:qsqsKM5}
H(a,d)=[a-d-1]_i H(a,d+1),
\qquad \text{for } 0\leq d< a.
\end{align}
\end{lemma}

\begin{proof}
By \eqref{eq:BBad}, we have the following identity, for $t,l\geq 0$:
\begin{align*} 
    B_{\tau i}^{(l,t+l+d,t)}B_{\tau i}
    =& \; [l+d+1]_i B_{\tau i}^{(l,l+t+d+1,t)}+[l+1]_i B_{\tau i}^{(l+1,l+t+d+1,t-1)}
    \\
    &-[l+d+1]_i q_i^{-1}\{k_i;l-t-\frac{1}{2}-2d\}B_{\tau i}^{(l,l+t+d,t-1)}.
\end{align*}
Using this identity, we rewrite $H(a,d)$ in \eqref{Had} as
\begin{align*} 
     H(a,d) =&\sum_{t,l\geq 0} q_i^{-\frac{(t-l)^2}{2} +(l-d)a+(l-2t)d} k_i^{t-l+2d} \frac{ 1}{[l+1]_i\cdots[l+a]_i} B_{\tau i}^{(l,l+t+d+1,t)} B_{\tau i}^{a-d}
     \\
     =& \; \sum_{t,l\geq 0} q_i^{-\frac{(t-l)^2}{2} +(l-d)a+(l-2t)d+2(l-t)-2-3d} k_i^{t-l+2d} \times
     \\
     &\qquad \times \frac{ q_i^{-a}[l+a]_i-q_i^{-d-1}[l+d+1]_i}{[l+1]_i\cdots[l+a]_i} B_{\tau i}^{(l,l+t+d+1,t)} B_{\tau i}^{a-d-1} 
     \\
     =& \; [a-d-1]_i H(a,d+1),
\end{align*}
as desired.
\end{proof}

\begin{lemma}
\label{lem:qsqsKM3}
The following identity holds, for $a> 0$:
\begin{align*} 
     &\sum_{t,l\geq 0} q_i^{-\frac{(t-l)^2}{2} + l a } k_i^{t-l } \frac{ 1}{[l+1]_i\cdots[l+a]_i} B_{\tau i}^{(l,l+t ,t)} B_{\tau i}^{a }
     = 0.
\end{align*}
\end{lemma}

\begin{proof}
To prove the lemma amounts to showing that $H(a,0)=0$; for notation $H(a,d)$ see \eqref{Had}. 
Applying \cref{lem:qsqsKM2} repeatedly, we have
\begin{align*}
H(a,0) =[a-1]_i H(a,1) &=[a-1]_i [a-2]_i H(a,2)\\
&=\cdots = \prod_{m=1}^{a} [a-m]_i \cdot H(a,a)=0,
\end{align*}
proving the lemma.
\end{proof}

\begin{theorem}\label{thm:qs-rktwo}
For any $j\neq i,\tau i$, we have
\begin{align}\label{eq:qqsgb}
    g_{\nu+n\alpha-n\beta}(k_i,B_i,B_{\tau i}) b_{n,0,0,0} =
    b_{n,n\beta,n\alpha+n\beta,n\alpha} g_{\nu }(k_i,B_i,B_{\tau i}).
\end{align}
\end{theorem}

\begin{proof}
We compute the RHS \eqref{eq:qqsgb} (up to a simple factor) using \cref{lem:qsqsKM1} as follows:
\begin{align*}
  &q_i^{-\nu^2/2-(n\alpha-n\beta)\nu} k_i^{\nu}  b_{n,n\beta,n\alpha+n\beta,n\alpha} g_{\nu }(k_i,B_i,B_{\tau i})
   \\
   =& \; \sum_{t,l\geq 0} q_i^{-\frac{(t-l)^2}{2}-(n\alpha-n\beta)(t-l)} k_i^{t-l} b_{n,n\beta,n\alpha+n\beta,n\alpha} B_{\tau i}^{(t)} B_i^{(t+l)}B_{\tau i}^{(l)}
   \\
   =& \;  \sum_{t,l\geq 0}\sum_{x=0}^{n\beta} q_i^{-\frac{(t-l)^2}{2}-n(\alpha-\beta)(t-l)}q_i^{(t-x)n\alpha-tx-\frac{1}{2}x^2} k_i^{t-l} B_{\tau i}^{(t-x)}b_{n,n\beta-x,n\alpha+n\beta,n\alpha} k_i^{-x} B_i^{(t+l)}B_{\tau i}^{(l)}
   \\
   =& \;  \sum_{t,l\geq 0}\sum_{x=0}^{n\beta} \sum_{y=0}^x \sum_{u=0}^{n\alpha+n\beta}(-1)^y q_i^{-\frac{(t-l)^2}{2}}q_i^{(t-x)n\alpha-tx-\frac{1}{2}x^2+3(t+l)x+(t+l-u)(n\beta-2x+y)-\frac{u(u+2t+2l)}{2}-y}
   \times 
   \\
   &\qquad \times \qbinom{n\beta-x+y}{y}_i k_i^{t-l} B_{\tau i}^{(t-x)}B_i^{(t+l-y-u)}b_{n,n\beta-x+y,n\alpha+n\beta-u,n\alpha}  k_i^{u-x}B_{\tau i}^{(l)}
   \\
   =& \;  \sum_{t,l\geq 0}  \sum_{u=0}^{n\alpha+n\beta}\sum_{y=0}^{n\beta}\sum_{x=y}^{n\beta} (-1)^{x-y}q_i^{-\frac{(t-l+x)^2+x^2}{2}-(n\alpha-n\beta)(t-l)+tn\alpha+3(t+l+x)x+(t+l+x-u)(n\beta-x-y)}
   \times 
   \\
   & \times q_i^{-\frac{u(u+2t+2l+2x)}{2}-x+y-3lx-(t+x)x}\qbinom{n\beta-y}{x-y}_i k_i^{t-l} 
   B_{\tau i}^{(t)} B_i^{(t+l+y-u)}b_{n,n\beta-y,n\alpha+n\beta-u,n\alpha} k_i^u  B_{\tau i}^{(l)}
   \\
   =& \;  \sum_{t,l\geq 0}  \sum_{u=0}^{n\alpha+n\beta}\sum_{y=0}^{n\beta}\sum_{x=y}^{n\beta} q_i^{-\frac{(t-l)^2}{2}-(n\alpha-n\beta)(t-l)}q_i^{tn\alpha +(t+l-u)(n\beta-y)-\frac{u(u+2t+2l)}{2}+y+x(n\beta-y-1) }
   \times 
   \\
   &\qquad \times (-1)^{x-y}\qbinom{n\beta-y}{x-y}_i k_i^{t-l} 
   B_{\tau i}^{(t)}B_i^{(t+l+y-u)}b_{n,n\beta-y,n\alpha+n\beta-u,n\alpha} k_i^u  B_{\tau i}^{(l)}.
\end{align*}

Using the following standard $q$-binomial identity
\begin{align*}
    \sum_{x=y}^{n\beta} (-1)^{x-y} q_i^{(x-y)(n\beta-y-1)}\qbinom{n\beta-y}{x-y}_i=\delta_{y,n\beta},
\end{align*}
we simplify the above formula of $ b_{n,n\beta,n\alpha+n\beta,n\alpha} g_{\nu }(k_i,B_i,B_{\tau i})$ as
\begin{align}\notag
     & q_i^{-\nu^2/2-n(\alpha-\beta)\nu} k_i^{\nu}  b_{n,n\beta,n\alpha+n\beta,n\alpha} g_{\nu }(k_i,B_i,B_{\tau i})
     \\\label{eq:qsqsKM1}
      =& \;  \sum_{t,l\geq 0}  \sum_{u=0}^{n\alpha+n\beta} q_i^{-\frac{(t-l)^2}{2}+
     tn\alpha -\frac{u(u+2t+2l)}{2}-n(\alpha-\beta)(t-l)} k_i^{t-l} 
   B_{\tau i}^{(t)}B_i^{(t+l+n\beta-u)}b_{n,0,n\alpha+n\beta-u,n\alpha} k_i^u  B_{\tau i}^{(l)}.
\end{align}
We next compute the right-hand side of \eqref{eq:qsqsKM1} using the third identity in \cref{lem:qsqsKM1}: 
\begin{align*}
    &q_i^{-\nu^2/2-n(\alpha-\beta)\nu} k_i^{\nu}  b_{n,n\beta,n\alpha+n\beta,n\alpha} g_{\nu }(k_i,B_i,B_{\tau i})
    \\
    =& \; \sum_{t,l\geq 0}  \sum_{u=0}^{n\alpha+n\beta} \sum_{v=0}^u\sum_{r=0}^{n\alpha} q_i^{-\frac{(t-l)^2}{2}-n(\alpha-\beta)(t-l)+tn\alpha -\frac{u(u+2t)}{2}+(l-r)(n\alpha+n\beta-2u+v)-\frac{r(r+2l)}{2}-v+2lu}\times 
    \\
   & \times (-1)^v\qbinom{n\alpha+n\beta-u+v}{v}_i  k_i^{t-l} B_{\tau i}^{(t)}B_i^{(t+l+n\beta-u)} B_{\tau i}^{(l-v-r)}b_{n,0,n\alpha+n\beta-u+v,n\alpha-r} k_i^{u-r}
   \\
   =& \; \sum_{t,l\geq 0} \sum_{r=0}^{n\alpha} \sum_{v=0}^{n\alpha+n\beta} \sum_{u=v}^{n\alpha+n\beta} q_i^{-\frac{(t-l)^2}{2}-n(\alpha-\beta)(t-l)+tn\alpha +(l-r)(n\alpha+n\beta -v)-\frac{r(r+2l)}{2}+u(n\alpha+n\beta-v-1)} \times
    \\
   & \times (-1)^{u-v} q_i^{v}\qbinom{n\alpha+n\beta-v}{u-v}_i k_i^{t-l}  B_{\tau i}^{(t)}B_i^{(t+l+n\beta)} B_{\tau i}^{(l+v-r)}b_{n,0,n\alpha+n\beta-v,n\alpha-r} k_i^{-r}.
\end{align*}

Using the following $q$-binomial identity
\begin{align*}
    \sum_{u=v}^{n\alpha+n\beta} (-1)^{u-v} q_i^{(u-v)(n\alpha+n\beta-v-1)}\qbinom{n\alpha+n\beta-v}{u-v}_i=\delta_{v,n\alpha+n\beta},
\end{align*}
we further simplify the above formula as
\begin{align}\notag
    &q_i^{-\nu^2/2-n(\alpha-\beta)\nu} k_i^{\nu}  b_{n,n\beta,n\alpha+n\beta,n\alpha} g_{\nu }(k_i,B_i,B_{\tau i})
    \\\notag
   =& \; \sum_{t,l\geq 0} \sum_{r=0}^{n\alpha} q_i^{-\frac{(t-l)^2}{2}-n(\alpha-\beta)(t-l)+tn\alpha -\frac{r(r+2l)}{2} } k_i^{t-l}  B_{\tau i}^{(t)}B_i^{(t+l+n\beta)} B_{\tau i}^{(l+n\alpha+n\beta-r)}b_{n,0,0,n\alpha-r} k_i^{-r}
   \\\notag
   =& \; \sum_{r=0}^{n\alpha} \sum_{t,l\geq 0} q_i^{-\frac{(t-l+n\beta)^2}{2}-n(\alpha-\beta)(t-l+n\beta)+(t+r)n\alpha -\frac{r(r+2l)}{2} } \times
   \\\label{eq:qsqsKM2}
   &\qquad\quad\times k_i^{t-l+n\beta-r}  B_{\tau i}^{(t)}B_i^{(t+l)} B_{\tau i}^{(l+n\alpha-r)}b_{n,0,0,n\alpha-r}.
\end{align}

We claim that, for any $0\leq r<\alpha$,
\begin{align} \label{eq:qsqsKM3}
    \sum_{t,l\geq 0} q_i^{-\frac{(t-l+n\beta)^2}{2}-n(\alpha-\beta)(t-l+n\beta)+(t+r)n\alpha -\frac{r(r+2l)}{2} } k_i^{t-l+n\beta-r}  B_{\tau i}^{(t)}B_i^{(t+l)} B_{\tau i}^{(l+n\alpha-r)}=0,
\end{align}
which is equivalent to
\begin{align}\label{eq:qsqsKM4}
    \sum_{t,l\geq 0} q_i^{-\frac{(t-l)^2}{2} +l(n\alpha-r)} k_i^{t-l} \frac{1}{[l+1]_i\cdots[l+n\alpha-r]_i} B_i^{(l)} B_{\tau i}^{(t+l )}B_{i}^{(t)} B_{\tau i}^{n\alpha-r}=0.
\end{align} 
Indeed, the identity \eqref{eq:qsqsKM4} follows by the identity in Lemma~\ref{lem:qsqsKM3} for $a=n\alpha-r$.

Using \eqref{eq:qsqsKM2}--\eqref{eq:qsqsKM3}, we finally obtain that 
\begin{align*}
    &b_{n,n\beta,n\alpha+n\beta,n\alpha} g_{\nu }(k_i,B_i,B_{\tau i})
    \\
    =& \;  q_i^{(\nu+n\alpha-n\beta)^2/2}  k_i^{-\nu-n\alpha+n\beta}\sum_{t\ge 0,l\ge 0} q_i^{ -\frac{(t-l)^2}{2} } k_i^{t-l} B^{(t,t+l,l)}_{i} b_{n,0,0,0}
    \\
    =& \; g_{\nu+n\alpha-n\beta }(k_i,B_i,B_{\tau i})b_{n,0,0,0},
\end{align*}
proving the desired identity \eqref{eq:qqsgb}.
\end{proof}

\begin{corollary}
\label{cor:qs-rktwo}
 For any $j\in \wI,j\neq i,\tau i$, the identity \eqref{eq:Tgnu} holds.
\end{corollary}

\begin{proof}
Setting $n=1$ in \cref{thm:qs-rktwo}, we obtain
\[
    g_{\nu+\alpha-\beta}(k_i,B_i,B_{\tau i}) b_{1,0,0,0} =
    b_{1,\beta,\alpha+\beta,\alpha} g_{\nu }(k_i,B_i,B_{\tau i}).
\]
Recall from \S\ref{sec:braid-formula} that $\TT_{i,-1}'(B_j)=b_{1,\beta,\alpha+\beta,\alpha}$ and $b_{1,0,0,0} =B_j$. Then the desired identity \eqref{eq:Tgnu} follows.
\end{proof}

\subsection{Relation with Lusztig symmetries} 
\label{sec:qsrk1-mod}

We can view any $\U$-module as a $\Ui$-module by restriction. Recall from \eqref{braid_rescale} the normalized Lusztig symmetries $T'_{\bs_i,\bvs, -1}$ on $\U$.

\begin{theorem}
\label{thm:res-qssplit}
Assume $c_{i,\tau i}=-1$ and let $\bvs$ be a balanced parameter. Let $M$ be an integrable $\U$-module with weights bounded above. For any $v\in M$, we have
\begin{align}
\label{eq:qssTmod0}
  \TT'_{i,-1}(v)=\fX_{i,\bvs} T'_{\bs_i,\bvs,-1}(v),\qquad
  \TT''_{i,+1}(v)=T''_{\bs_i,\bvs,+1}(\fX_{i,\bvs }^{-1} v).
\end{align}
\end{theorem}

\begin{proof}

Recall that the isomorphism $\phi_{\bvs ,\bvs'}$ is the restriction of $\Phi_{\ov{\bvs }\bvs'}$ when $\bvs,\bvs'$ are balanced parameters. Note that $\Phi_{\ov{\bvs}\bvs'} (\fX_{i,\bvs}) = \fX_{i,\bvs'}$. Therefore, by \cref{thm:braid-iQG}, it suffices to prove the identities \eqref{eq:qssTmod0} for the parameter $\bvs_\star$. In this case, $\vs_{i,\star}=\vs_{\tau i,\star}=q_i^{1/2}$.

Let $\U_{i,\tau i}$ be the subalgebra of $\U$ generated by $E_j,F_j,K_j^{\pm1}$ for $j\in \{i,\tau i\}$ and $\Ui_{i,\tau i}$ be the subalgebra of $\Ui$ generated by $B_i,B_{\tau i},k_i$. There is a natural isomorphism $\U_{i,\tau i}\cong \U(\sl_3)$ such that
$q_i\mapsto q,E_i\mapsto E_1,E_{\tau i}\mapsto E_2,F_i\mapsto F_1,F_{\tau i}\mapsto F_2$. Moreover, $(\U_{i,\tau i},\Ui_{i,\tau i})$ is itself a quantum symmetric pair of quasi-split type AIII. Since the identity \eqref{eq:qssTmod0} only involves elements in $\U_{i,\tau i}$, it suffices to assume that $\U=\U(\sl_3)$ and $\Ui$ is the subalgebra of $\U$ generated by 
\[
B_1=F_1+q^{1/2} E_2 K_1^{-1},\quad B_2=F_2+q^{1/2} E_1 K_2^{-1},\quad k_1=K_1 K_2^{-1}.
\]
Then it also suffices to assume that $M$ is an irreducible $\U(\sl_3)$-module, $L(m,n)$. Under this setting, we prove the desired identity \eqref{eq:qssTmod0} in Propositions~\ref{prop:TTeta}--\ref{prop:TTv} below.
\end{proof}

\begin{corollary}
\label{cor:qssmod-inv}
For $i\in \I$ such that $c_{i,\tau i}=-1$, and any integrable $\Ui$-module $M$, $\TT'_{i,-1}$ and $\TT''_{i,+1}$ are mutually inverse linear operators on $M$.
\end{corollary}

\begin{proof}
Assume that $\bvs$ is a balanced parameter for now. It follows from \eqref{eq:qssTmod0} that $\TT'_{i,-1}\TT''_{i,+1}v$ $=\TT''_{i,+1}\TT'_{i,-1}v=v$, for any $v\in M$. This implies that 
\[
 g_{\nu,\bvs}(k_i, B_i, B_{\tau i})  h_{\nu,\bvs}(k_i, B_i, B_{\tau i})v =v,
\]
for any $v\in M_{\ov{\lambda}}$, where $\nu=\langle h_i-h_{\tau i}, \ov{\lambda}\rangle$. Since $M$ can be taken to be an arbitrary integrable $\U$-module, we have for any $\nu\in \Z$,
\[
 g_{\nu,\bvs}(k_i, B_i, B_{\tau i})  h_{\nu,\bvs}(k_i, B_i, B_{\tau i})=1.
\]
Let $\bvs'$ be an arbitrary parameter. Applying $\phi_{\bvs,\bvs'}$ to the above identity, we obtain
\[
 g_{\nu,\bvs'}(k_i, B_i, B_{\tau i})  h_{\nu,\bvs'}(k_i, B_i, B_{\tau i})=1.
\]
i.e., $\TT'_{i,-1}\TT''_{i,+1}v=v$ for any $v\in M_{\ov{\lambda}}$, where $\nu=\langle h_i-h_{\tau i}, \lambda\rangle$. The argument for $\TT''_{i,+1}\TT'_{i,-1}v=v$ is entirely similar and omitted.
\end{proof}

In the remainder of this subsection, we shall prove \eqref{eq:qssTmod0} when $M=L(m,n)$ is an irreducible $\U(\sl_3)$-module  with highest weight vector $\eta$ (see Propositions~\ref{prop:TTeta}--\ref{prop:TTv}), completing the proof of Theorem~\ref{thm:res-qssplit}.

\begin{proposition}
\label{prop:TTeta}
Set the parameter to be $\bvs_\star=(q^{1/2},q^{1/2})$. Set $\nu=m-n$. Then we have 
\begin{align}
 \label{TT:qs}
    \fX_{1,\bvs_\star} T'_{\bs_1,\bvs_\star,-1} \eta &= g_\nu(k_1, B_1, B_2)\eta,
    \\ 
    T''_{\bs_1,\bvs_\star,+1}(\fX_{1,\bvs_\star }^{-1} \eta) &= h_\nu(k_1, B_1, B_2)\eta.
\end{align}
\end{proposition} 

\begin{proof}
By definition \eqref{eq:gnu}, we have
\begin{align}
\label{eq:gnu-star}
    g_\nu(k_1, B_1, B_2) \eta=q^{-\nu^2/2}\sum_{t\ge 0,l\ge 0} (-1)^{t+l}  q^{-\frac{(t-l)^2}{2}-t-l} q^{(t-l)\nu} B^{(t,t+l,l)}_{1} \eta.
\end{align}

Since $c_{1,2}=-1$, we have $\bs_1=s_1 s_{2} s_1=s_{2} s_1 s_2$. Lusztig's braid group symmetry $T'_{\bs_1,-1}$ acts on $\eta$ via the composition
\begin{align}
\label{eq:qsTmod}
    \eta \overset{T'_{1,-1}}{\longrightarrow}
    F_1^{(m)}\eta \overset{T'_{2,-1}}{\longrightarrow } 
    F_2^{(m+n)}F_1^{(m)}\eta \overset{T'_{1,-1}}{\longrightarrow} F_2^{(n,m+n,m)} \eta.
\end{align}
Indeed, one can show that $F_1^{(m)}\eta$ is annihilated by $E_2$ and  $F_2^{(m+n)}F_1^{(m)}\eta$ is annihilated by $E_1$. Then this action \eqref{eq:qsTmod} follows by \cite[Proposition 5.2.2]{Lus93}. Similarly, the action of $T''_{\bs_1,+1}$ on $\eta$ is given by the composition
\begin{align*} 
    \eta \overset{T''_{1,+1}}{\longrightarrow}
    (-q)^m F_1^{(m)}\eta \overset{T''_{2,+1}}{\longrightarrow } 
    (-q)^{2m+n} F_2^{(m+n)}F_1^{(m)}\eta \overset{T''_{1,+1}}{\longrightarrow} 
    (-q)^{2m+2n} F_2^{(n,m+n,m)} \eta.
\end{align*}

Recall from \cref{thm:braid-iQG} that $T'_{\bs_1,\bvs_\star,-1},T_{\bs_1,\bvs_\star,+1}''$ are obtained from $T'_{\bs_1,-1},T''_{\bs_1,+1}$ using the rescaling map $\Phi_{\ov{\bvs_\dm}\bvs}$, and we have
\begin{align}
\label{eq:qsimod1} 
T'_{\bs_1,\bvs_\star,-1}(\eta)= (-q)^{-m-n} F_2^{(n,m+n,m)} \eta,
\qquad 
T''_{\bs_1,\bvs_\star,+1}(\eta)= (-q)^{m+n} F_2^{(n,m+n,m)} \eta.
\end{align}
 
Using Proposition~\ref{prop:FB} and Theorem~\ref{thm:qsibasis}, we have
\begin{align}
\label{eq:qsimod2} 
F_2^{(n,m+n,m)}\eta= \sum_{\substack{t\ge 0,\\ l\ge 0}} (-1)^{n-t+m-l}q^{(n-t)(-2m-2+2l+n-t)/2+(m-l)(m-l-2)/2} B^{(t,t+l,l)}_{2} \eta,
\end{align}
where $B^{(t,t+l,l)}_{2} \eta =0$ if $l> m$ or $t> n$.
Note that the left-hand side of \eqref{eq:qsimod2} is $\psi$-invariant. Applying $\psi$ to the above identity, we have
\begin{align}
\label{eq:qsimod2'} 
F_2^{(n,m+n,m)}\eta= \sum_{\substack{t\ge 0,\\ l\ge 0}} (-1)^{n-t+m-l} q^{-(n-t)(-2m-2+2l+n-t)/2-(m-l)(m-l-2)/2} \ov{B^{(t,t+l,l)}_{2}} \eta.
\end{align}

On the other hand, since the elements $B^{(t,t+l,l)}_{2}\eta$ are $\psi^\imath$-invariant and $\psi^\imath =\fX_{1,\bvs_\star} \circ \ov{{\,\cdot\,}}$ (cf. \cite{BW18a}), we have 
\begin{align}
\label{eq:qsimod3} 
\fX_{1,\bvs_\star}\ov{B^{(t,t+l,l)}_{2}}\eta
=\psi^\imath(B^{(t,t+l,l)}_{2}\eta)=B^{(t,t+l,l)}_{2}\eta.
\end{align}
 
Using \eqref{eq:qsimod1}, \eqref{eq:qsimod2'} and \eqref{eq:qsimod3}, we obtain that
\begin{align*}
&\quad\fX_{1,\bvs_\star} T'_{\bs_1,\bvs_\star,-1}\eta 
\\
&=(-q)^{-m-n} \fX_{1,\bvs_\star} F_2^{(n,m+n,m)}\eta
   \\
   &=(-q)^{-m-n} \sum_{t\ge 0,l\ge 0} (-1)^{n-t+m-l} q^{-(n-t)(-2m-2+2l+n-t)/2-(m-l)(m-l-2)/2} \fX_{1,\bvs_\star}\ov{B^{(t,t+l,l)}_{2}} \eta
   \\
   &=q^{-m-n} \sum_{t\ge 0,l\ge 0} (-1)^{t+l} q^{-(n-t)(-2m-2+2l+n-t)/2-(m-l)(m-l-2)/2}  B^{(t,t+l,l)}_{2} \eta 
   \\
   &=  \sum_{t\ge 0,l\ge 0} (-1)^{t+l} q^{-\frac{(m-l-n+t)^2}{2}}q^{-t-l}  B^{(t,t+l,l)}_{2} \eta
   \\
   &=  \sum_{t\ge 0,l\ge 0} (-1)^{t+l} q^{-\frac{(\nu-l +t)^2}{2}}q^{-t-l}  B^{(t,t+l,l)}_{2} \eta=g_\nu(k_1, B_1, B_2) \eta.
\end{align*}
Similarly, using \eqref{eq:qsimod1}, \eqref{eq:qsimod2} and \eqref{eq:qsimod3}, we obtain that
\begin{align*}
   &\quad  T''_{\bs_1,\bvs_\star,+1}(\fX_{1,\bvs_\star}^{-1} \eta)
    \\
    & = (-q)^{m+n} F_2^{(n,m+n,m)}\eta
   \\
   &= (-q)^{m+n} \sum_{t\ge 0,l\ge 0} (-1)^{n-t+m-l}q^{(n-t)(-2m-2+2l+n-t)/2+(m-l)(m-l-2)/2} B^{(t,t+l,l)}_{2} \eta
   \\
   &=\sum_{t\ge 0,l\ge 0} (-1)^{t+l}  q^{\frac{(m-n+t-l)^2}{2}}q^{t+l}  B^{(t,t+l,l)}_{2} \eta
   \\
   &=h_\nu(k_1, B_1, B_2) \eta.
\end{align*}
Comparing the above formulas, we have proved the proposition. 
\end{proof}

\begin{proposition}
\label{prop:TTv}
Set the parameter $\bvs_\star=(q^{1/2},q^{1/2})$. For any $v\in L(m,n)$, we have 
\begin{align}
 \label{TT:qsv}
    \TT'_{1,-1}v &= \fX_{1,\bvs_\star} T'_{\bs_1,\bvs_\star,-1} (v ),
    \\ 
 \label{TT:qsv'}
    \TT''_{1,+1}v &= T''_{\bs_1,\bvs_\star,+1}(\fX_{1,\bvs_\star}^{-1} v)  .
\end{align} 
\end{proposition}

\begin{proof}
We prove the first identity \eqref{TT:qsv} only, skipping a similar argument for \eqref{TT:qsv'}.

By \cite[Definition 10.4 and Theorem 10.5]{WZ23}, we have
\begin{align}
\label{eq:fXTT}
 \fX_{1,\bvs_\star} T'_{\bs_1,\bvs_\star,-1} ( B_j v )= \TT'_{1,-1}(B_j)\fX_{1,\bvs_\star} T'_{\bs_1,\bvs_\star,-1} (v ),\qquad j=1,2.
\end{align}

Thanks to \eqref{eq:fXTT} and Theorem~\ref{thm:qs-split}, if the identity \eqref{TT:qsv} holds for $v\in L(m,n) $ then it also holds for $B_j v,j=1,2$. 
When $v$ is the highest weight vector $\eta$ of $L(m,n)$, the identity \eqref{TT:qsv} is proved in Proposition~\ref{prop:TTeta}. Therefore, the identity \eqref{TT:qsv} holds for any $v\in \Ui\eta$. By \cite[Lemma~4.1]{BW18b}, we have $L(m,n)=\Ui\eta$, and hence \eqref{TT:qsv} holds for any $v\in L(m,n)$.
\end{proof}

\subsection{Formulas of $\TT'_{i,-1}$ on divided powers}

Set the parameter $\bvs=\bvs_\dm$ in this subsection. Then we have $\vs_i=\vs_{\tau i}=-q_i^{-1/2}$. Recall $X_{j,n,\ov{t}}$ from \eqref{Xjn}.

\begin{theorem}
\label{thm:qsplit-T1DP}
Assume $c_{i,\tau i}=-1$. 
Let $j\in\I$ such that $j\neq i,\tau i$. Set $\alpha=-c_{ij},\beta=-c_{\tau i,j}$. Then we have $\TT'_{i,-1}(X_{j,n,\ov{t}})=b_{i,\tau i, j;n,n\beta,n\alpha+n\beta,n\alpha}$, for $n\ge 0,\ov{t}\in \Z_2$; that is,   
\begin{align}\notag
\TT'_{i,-1}(X_{j,n,\ov{t}})&=\sum_{u,w\geq 0} \sum_{t=0}^{n\beta-w} \sum_{s=0}^{n\beta+n\alpha-w-u} \sum_{r=0}^{n\alpha-u} (-1)^{t+r+s}
 q_i^{t(-2w+1) + r(u+1)+s(w-2u+1)+uw+3ut} 
 \\
 \label{eq:qs-Tdiv1}
&\quad\times q_i^{-\frac{u^2+w^2}{2}} B_i^{(n\beta-w-t)}B_{\tau i}^{(n\alpha+n\beta-w-u-s)}B_i^{(n\alpha-u-r)} X_{j,n,\ov{t}} B_i^{(r)} B_{\tau i}^{(s)} B_i^{(t)} k_i^{w-u}.
\end{align}
Moreover, the formula of $\TT''_{i,+1}(X_{j,n,\ov{t}})$ is given by
\begin{align}\notag
\TT''_{i,+1}(X_{j,n,\ov{t}})&=\sum_{u,w\geq 0} \sum_{t=0}^{n\beta-w} \sum_{s=0}^{n\beta+n\alpha-w-u} \sum_{r=0}^{n\alpha-u} (-1)^{t+r+s}
 q_i^{t(-2w+1) + r(u+1)+s(w-2u+1)+uw+3ut} 
 \\
 \label{eq:qs-Tdiv2}
&\quad\times q_i^{-\frac{u^2+w^2}{2}} k_i^{u-w} B_i^{(t)} B_{\tau i}^{(s)} B_i^{(r)} X_{j,n,\ov{t}} B_i^{(n\alpha-u-r)} B_{\tau i}^{(n\alpha+n\beta-w-u-s)} B_i^{(n\beta-w-t)} .
\end{align}
\end{theorem}

\begin{proof}
The proof is the same as for \cref{thm:split-T1DP} and is hence omitted.
\end{proof}

\section{Integral relative braid group symmetries}
\label{sec:integral}

In this section we rescale naturally the symmetries on integrable $\Ui$-modules to become integral and then formulate symmetries on $\dUi$ which preserve its integral form. We show these symmetries on integral forms of modules and $\dUi$ satisfy the relative braid group relations.

\subsection{Modified iquantum groups and completions}

Set $\A=\bZ[q,q^{-1}]$. Let $\dot{\U}$ be the modified quantum group in \cite[23.1]{Lus93} and $\dot{\U}_\A$ be the $\A$-subalgebra of $\dot{\U}$ generated by $F_i^{(m)}\one_\lambda,E_i^{(m)}\one_\lambda$ for $\lambda\in X,m\in \N,i\in \I$. Let $\widehat{\U}$ be the $\A$-algebra consisting of all $\A$-linear combinations
of the canonical basis on $\dot\U$ \cite[1.11]{Lus09}. The algebra $\dot{\U}$ is naturally a subalgebra of $\widehat{\U}$. Following \cite[1.12]{Lus09}, any integrable $\U$-module admits a natural $\widehat{\U}$-module structure.

Let $\dUi$ be the modified iquantum group defined in \cite{BW18b} and \cite[\S3.5]{BW21}, which contains the idempotents $\one_{\ov{\lambda}}$, for $\ov{\lambda}\in X_\imath$. Moreover, any $M\in \mathcal{\mathcal{C}}$ admits a natural $\dUi$-module structure. The modified quantum group $\dot{\U}$ is naturally a $(\dUi,\dUi)$-bimodule by \cite[3.7]{BW18b}.  Following \cite[Definition 3.19]{BW18b}, we define 
\[
\dUiA:=\{u\in \dUi \mid u m\in \dot{\U}_{\A} , \forall m\in \dot{\U}_{\A} \}.
\]

Let $\dot{\bB}^\imath$ be the icanonical basis of $\dUi$ \cite{BW18b, BW21}, which forms an $\A$-basis of $\dUiA$ and an $\F$-basis of $\dUi$. For $a,b,c\in  \dot{\bB}^\imath$, we write $ab=\sum_{c\in \dot{\bB}^\imath} m_{a,b}^c c$ where $m_{a,b}^c \in \A$ are $0$ except finitely many $c$.  

\begin{lemma} \cite[Corollary 3.2]{BS22}
\label{lem:BS}
For any given $c\in \dot{\bB}^\imath$, the set $\{(a,b)\in \dot{\bB}^\imath\times \dot{\bB}^\imath \mid m_{a,b}^c\neq 0\}$ is finite.
\end{lemma}

Let $\widehat{\U}_\A^\imath$ be the $\A$-module consisting of all formal linear combinations
$\sum_{a\in \dot{\bB}^\imath} n_a a$
with $n_a\in \A$. 
By Lemma~\ref{lem:BS}, we have a well-defined $\A$-algebra structure on $\widehat{\U}_\A^\imath$ extending the algebra structure on $\dUiA$, 
and $\dUiA \subset \widehat{\U}_\A^\imath$ as a subalgebra. The algebra $\widehat{\U}_\A^\imath$ contains $\sum_{\ov{\lambda}\in X_\imath} \one_{\ov{\lambda}}$ as the unit. Then $\widehat{\U}^\imath:= \widehat{\U}^\imath_{\A} \otimes_{\A} \F$ is naturally an $\F$-algebra. 
Following \cite[\S3.1]{BS22}, there is an algebra embedding 
\[
\widehat{\U}_\A^\imath \longrightarrow\widehat{\U}_\A,
\qquad 
x\one_{\ov{\lambda}}\mapsto x \sum_{\mu\in X,\ov{\mu}=\ov{\lambda}}\one_{\mu}.
\]

In the remainder of this section, we shall fix a set $\fwItau$ of representatives for $\tau$-orbits on $\Ifin$ and choose the parameter $\bvs $ such that 
\begin{align} \label{parameter_iCB}
\vs_i\in q_i^\Z, \qquad \vs_{\tau i}\vs_i =q_i^{-c_{i,\tau i}}  ,\qquad \forall i\in \I.
\end{align}
Note that $\vs_i =q_i^{-1}$  if $\tau i=i.$ (More generally, we could have allowed $\vs_i\in \pm q_i^\Z$ in \eqref{parameter_iCB}, though we choose to use the positive signs throughout.) The $\dUi$ with these parameters admits the bar involution \cite{CLW21a} and icanonical basis \cite{BW18b, BW21}.

Following \cite[Proposition 3.17]{BW18b}, there is a bar involution $\psi^\imath$ on $\dot{\U}^\imath$ such that $\psi^\imath(q)=q^{-1},\psi^\imath(\one_{\ov{\lambda}})=\one_{\ov{\lambda}},
\psi^\imath(B_i\one_{\ov{\lambda}})=B_i\one_{ \ov{\lambda}}$, for $i\in \I,\ov{\lambda}\in X_\imath$. By definition, $\dot{\bB}^\imath$ is $\psi^\imath$-invariant. 
We extend $\psi^\imath$ to a bar involution, denoted again by $\psi^\imath$, on $\widehat{\U}^\imath$ via $\psi^\imath(\sum_{a\in \dot{\bB}^\imath} n_a a)=\sum_{a\in \dot{\bB}^\imath} \ov{n}_a a$.

\subsection{Integral braid group symmetries on modules}
 \label{subsec:split_mod}
 
We separate $i\in \Ifin$ into 3 cases (i)--(iii) with distinguished choices of parameters, and will often do so in the subsequent subsections.

(i) $\boxed{c_{i,\tau i}=2 \text{ and } \vs_i=q_i^{-1}}$. 
In this case, we have $i=\tau i$. For $\la\in X$, set $\la_i=\langle h_i,\lambda\rangle \in \Z.$ We define 
\begin{align}
{\bf t}_i' &:= \sum_{m\geq 0} \sum_{\substack{\ov{\lambda}\in X_\imath\\\ov{\la_i}=\ov{0}}} (-1)^{m}  q_i ^{-m} \dvi{\ov{0} }{2m}\one_{\ov{\lambda}}
+\sum_{m\geq 0} \sum_{\substack{\ov{\lambda}\in X_\imath\\\ov{\la_i}=\ov{1}}} (-1)^{m} q_i^{-m} \dvi{\ov{1} }{2m+1}\one_{\ov{\lambda}}\in \widehat{\U}^\imath_\A,
\label{ti'}
\\
{\bf t}_i'' &:= \sum_{m\geq 0} \sum_{\substack{\ov{\lambda}\in X_\imath\\\ov{\la_i}=\ov{0}}} (-1)^{m}  q_i ^{m} \dvi{\ov{0} }{2m}\one_{\ov{\lambda}}
+\sum_{m\geq 0} \sum_{\substack{\ov{\lambda}\in X_\imath\\\ov{\la_i}=\ov{1}}} (-1)^{m} q_i^{m} \dvi{\ov{1} }{2m+1}\one_{\ov{\lambda}}\in \widehat{\U}^\imath_\A.
\label{ti''}
\end{align}
The above formulas are related to $f_{\ov{p},\bvs_\star}(B_i)$ in \eqref{def:fbvs} as follows: 
\begin{align}
\label{ti'f}
\begin{split}
{\bf t}_i' &=\sum_{\substack{\ov{\lambda}\in X_\imath\\\ov{\la_i}=\ov{0}}} f_{\ov{0},\bvs_\star}(B_i) \one_{\ov{\lambda}} + (-q_i)^{1/2} \sum_{\substack{\ov{\lambda}\in X_\imath\\\ov{\la_i}=\ov{1}}} f_{\ov{1},\bvs_\star}(B_i) \one_{\ov{\lambda}},
\\
{\bf t}_i'' &=\sum_{\substack{\ov{\lambda}\in X_\imath\\\ov{\la_i}=\ov{0}}} \tilde{f}_{\ov{0},\bvs_\star}(B_i) \one_{\ov{\lambda}} + (-q_i)^{-1/2} \sum_{\substack{\ov{\lambda}\in X_\imath\\\ov{\la_i}=\ov{1}}} \tilde{f}_{\ov{1},\bvs_\star}(B_i) \one_{\ov{\lambda}}.
\end{split}
\end{align}

(ii) 
$\boxed{c_{i,\tau i}=0 \text{ and } \vs_i=\vs_{\tau i}=1}.$ Set $\la_{i,\tau}=\langle h_i-h_{\tau i},\ov{\lambda}\rangle$ for $\ov{\lambda}\in X_\imath$. We define integral elements in $\widehat{\U}_\A^\imath$ as follows:
\begin{align}
{\bf t}_i' &:=   \sum_{ \ov{\lambda}\in X_\imath }  \sum_{a-b=\la_{i,\tau}} (-1)^b q_i^{-b}  B_{i}^{(a)}  B_{\tau i}^{(b)} \one_{\ov{\lambda}},
\label{diagti'}
\\
{\bf t}_i'' &:= \sum_{ \ov{\lambda}\in X_\imath }  \sum_{b-a
=\la_{i,\tau}}  (-1)^b q_i^{b} B_{\tau i}^{(a)}  B_{ i}^{(b)}\one_{\ov{\lambda}}. 
\label{diagti''}
\end{align}
These elements are related to $ z_{m,\bvs}( B_i,B_{\tau i}), z'_{m,\bvs}( B_i,B_{\tau i})$ in \eqref{eq:zm}--\eqref{eq:zm'} as follows:
\begin{align}
  \label{ti:diag}
\begin{split}
{\bf t}_i' &= \sum_{ \ov{\lambda}\in X_\imath } (-1)^{-\la_{i,\tau}/2} q_i^{\la_{i,\tau}/2}z_{\la_{i,\tau},\bvs}( B_i,B_{\tau i}) \one_{\ov{\lambda}},
\\
{\bf t}_i'' &= \sum_{ \ov{\lambda}\in X_\imath } (-1)^{-\la_{i,\tau}/2} q_i^{\la_{i,\tau}/2}z'_{\la_{i,\tau},\bvs}( B_i,B_{\tau i}) \one_{\ov{\lambda}}.
\end{split}
\end{align}

(iii) $\boxed{c_{i,\tau i}=-1 \text{ and } \vs_i, \vs_{\tau i}\in q_i^\Z \text{ with } \vs_i \vs_{\tau i}=q_i}$. Set $\la_{i,\tau}=\langle h_i-h_{\tau i},\ov{\lambda}\rangle$ for $\ov{\lambda}\in X_\imath$.
We define 
\begin{align}
{\bf t}_i' &:=   \sum_{ \ov{\lambda}\in X_\imath }  q_i^{-\la_{i,\tau}(\la_{i,\tau}-1)/2}  \sum_{t\ge 0,l\ge 0} (-1)^{t+l}  q_i^{ -\frac{(t-l)(t-l+1)}{2}-t-l+\la_{i,\tau}(t-l)}  \vs_{\tau i}^{-\la_{i,\tau}+t-l}   B^{(t,t+l,l)}_{i} \one_{\ov{\lambda}},
\label{qss-ti'}
\\
{\bf t}_i'' &:=   \sum_{ \ov{\lambda}\in X_\imath }   q_i^{\la_{i,\tau}(\la_{i,\tau}-1)/2}  \sum_{t\ge 0,l\ge 0} (-1)^{t+l}  q_i^{ \frac{(t-l)(t-l+1)}{2}+t+l+\la_{i,\tau}(l-t)} \vs_{\tau i}^{\la_{i,\tau}-t+l}  B^{(t,t+l,l)}_{i}\one_{\ov{\lambda}}.
\label{qss-ti''}
\end{align}
Recall the elements $ g_{\nu,\bvs}, h_{\nu,\bvs}$ from \eqref{eq:gnu}--\eqref{eq:hnu}. We can rewrite \eqref{qss-ti'}--\eqref{qss-ti''} as 
\begin{align}\label{eq:ttgh}
{\bf t}_i' &=   \sum_{ \ov{\lambda}\in X_\imath } g_{\la_{i,\tau},\bvs}(k_i, B_i, B_{\tau i})\one_{\ov{\lambda}},
\qquad 
{\bf t}_i'' =   \sum_{ \ov{\lambda}\in X_\imath } h_{\la_{i,\tau},\bvs}(k_i, B_i, B_{\tau i})\one_{\ov{\lambda}},
\end{align}

For any integrable $\Ui$-module $M$, we have well-defined linear maps, for $i\in \Ifin$:
\begin{align}
\label{TTdM}
\begin{split}
    \TTd'_{i,-1}:  & M \longrightarrow M, \quad v\mapsto {\bf t}_i' v, 
    \qquad
    \TTd'_{i,+1}:   M \longrightarrow M, \quad v\mapsto \psi^\imath({\bf t}_i') v, 
    \\
    \TTd''_{i,+1}:  & M \longrightarrow M, \quad v\mapsto {\bf t}_i'' v,
    \qquad
    \TTd''_{i,-1}:  M \longrightarrow M, \quad v\mapsto  \psi^\imath({\bf t}_i'') v.
\end{split}
\end{align}

\begin{remark}
When $c_{i,\tau i}=2$ or $-1$, we have ${\bf t}_i'=\psi^\imath({\bf t}_i'')$ and thus $\TTd''_{i,e}=\TTd'_{i,e}$, for $ e=\pm 1$. These elements ${\bf t}_i', {\bf t}_i''$ in 
\eqref{ti'}--\eqref{ti''}, \eqref{diagti'}--\eqref{diagti''}, and 
\eqref{qss-ti'}--\eqref{qss-ti''} should be viewed as three $\mathrm{i}$-counterparts of the elements $s_{i,-1}', s_{i,+1}''$ in \eqref{si'si''} or \cite[2.2(b)]{Lus09}.
\end{remark}

\begin{lemma}
\label{lem:inverse}
Let $i\in \wI$ and $e =\pm 1$. Then we have ${\bf t}_i'{\bf t}_i''={\bf t}_i''{\bf t}_i'=1$, and the linear maps $\TTd'_{i,-e},\TTd''_{i,e}$ in \eqref{TTdM} are mutual inverses on any integrable $\Ui$-module $M$.
\end{lemma}

\begin{proof}
Let $M$ be an integrable $\Ui$-module and let $v\in M_{\ov\lambda}$, for $\ov{\lambda} \in X_\imath$. We divide the considerations into 3 cases for $i\in \wI$.

(i) $\boxed{c_{i,\tau i}=2}$. 
Recall that $B_i v$ has iweight $\ov{\lambda-\alpha_i}$ and note that $\ov{\lambda-2r\alpha_i}=\ov{\lambda}$, for $r\in \Z$. Then by \eqref{def:idv}, we have
\begin{align}\label{eq:slambda}
 f_{\ov{0},\bvs_\star}(B_i)v, \ \tilde{f}_{\ov{0},\bvs_\star}(B_i)v\in M_{\ov{\lambda}};\qquad
 f_{\ov{1},\bvs_\star}(B_i)v, \ \tilde{f}_{\ov{1},\bvs_\star}(B_i)v\in M_{\ov{\lambda-\alpha_i}}.
\end{align}
If $\ov{\la_i}$ is even, using Corollary~\ref{cor:mod-inv}, we have
\begin{align}
{\bf t}_i'{\bf t}_i''v=  f_{\ov{0},\bvs_\star}(B_i) \tilde{f}_{\ov{0},\bvs_\star}(B_i)v=\TT'_{i,-1}\TT''_{i,+1}(v) =v.
\end{align}
If $\ov{\la_i}$ is odd, using Corollary~\ref{cor:mod-inv}, we similarly have
\begin{align}
{\bf t}_i'{\bf t}_i''v= (-q_i)^{1/2} f_{\ov{1},\bvs_\star}(B_i) (-q_i)^{-1/2} \tilde{f}_{\ov{1},\bvs_\star}(B_i)v=\TT'_{i,-1}\TT''_{i,+1}(v) =v.
\end{align}
Therefore, we have $\TTd'_{i,-1} \TTd''_{i,+1} =\text{Id}_M$.
As this holds for any $M$, we have ${\bf t}_i'{\bf t}_i''=1$. Similarly, we have $\TTd''_{i,+1} \TTd'_{i,-1} =\text{Id}_M$ and ${\bf t}_i''{\bf t}_i'=1$.

(ii) $\boxed{c_{i,\tau i}=0}$. 
Note that ${\bf t}_i''v\in M_{\ov{\lambda-\la_{i,\tau}\alpha_i}}$ and $\langle h_i-h_{\tau i},\ov{\lambda-\la_{i,\tau}\alpha_i}\rangle=-\la_{i,\tau}$. By \eqref{ti:diag} and \eqref{ibraid:diag}, we have
\[
{\bf t}_i'{\bf t}_i''v=(-1)^{-\la_{i,\tau}/2} q_i^{\la_{i,\tau}/2} {\bf t}_i'( z'_{\la_{i,\tau},\bvs}v)=z_{-\la_{i,\tau},\bvs} z'_{\la_{i,\tau},\bvs}v=\TT'_{i,-1}\TT''_{i,+1}v=v,
\]
where the last equality follows from Corollary~\ref{cor:qsmod-inv}. Following the same arguments as in Case (i), we conclude that $\TTd'_{i,-1} \TTd''_{i,+1} = \TTd''_{i,+1} \TTd'_{i,-1} =\text{Id}_M$ and ${\bf t}_i'{\bf t}_i''={\bf t}_i''{\bf t}_i'=1$.

Applying $\psi^\imath$, we have $\psi^\imath({\bf t}_i')\psi^\imath({\bf t}_i'')=\psi^\imath({\bf t}_i'')\psi^\imath({\bf t}_i')=1$. Hence, we have $\TTd'_{i,-1} \TTd''_{i,+1}=\TTd''_{i,+1} \TTd'_{i,-1}  =\text{Id}_M$.

(iii) $\boxed{c_{i,\tau i}=-1}$. 
Note that ${\bf t}_i''v\in M_{\ov{\lambda}}$. By \eqref{eq:ttgh} and \eqref{eq:ibraid3}--\eqref{eq:ibraid3'}, we have
\[
{\bf t}_i'{\bf t}_i''v=g_{\la_{i,\tau},\bvs}(k_i, B_i, B_{\tau i}) h_{\la_{i,\tau},\bvs}(k_i, B_i, B_{\tau i})v= \TT'_{i,-1}\TT''_{i,+1}v=v,
\]
where the last equality follows from Corollary~\ref{cor:qssmod-inv}.

Therefore, we have $\TTd'_{i,-1} \TTd''_{i,+1} =\text{Id}_M$.
As this holds for any integrable $\Ui$-module $M$, we have ${\bf t}_i'{\bf t}_i''=1$. Similarly, we have $\TTd''_{i,+1} \TTd'_{i,-1} =\text{Id}_M$ and ${\bf t}_i''{\bf t}_i'=1$. The same arguments as in (i)--(ii) finishes the proof in this case. 
\end{proof}

\begin{lemma}\label{lem:iweight}
Let $M$ be an integrable $\Ui$-module and let $v\in M_{\ov{\lambda}}$, for $\ov{\lambda}\in X_\imath$. Then we have $\TTd'_{i,-1}(v), \TTd''_{i,+1}(v) \in M_{\ov{\bs_i\lambda}}$. 
\end{lemma}

\begin{proof}
(i) $\boxed{c_{i,\tau i}=2}$. 
In this case, $\bs_i=s_i$, and then we have
\begin{align} \label{rilambda_i}
\ov{\bs_i\lambda}=\ov{s_i\lambda}=\ov{\lambda-\la_i\alpha_i}
=
\begin{cases}
\ov{\lambda}, &\text{ if } \ov{\la_i}=\ov{0},
\\
\ov{\lambda-\alpha_i}, &\text{ if } \ov{\la_i}=\ov{1}.
\end{cases}
\end{align}
If $\ov{\la_i}=\ov{0}$, then we have $\TTd'_{i,-1}(v) ={\bf t}_i' v= f_{\ov{0},\bvs_\star}(B_i)v \in M_{\ov{\lambda}}$, by \eqref{eq:slambda}. If $\ov{\la_i}=\ov{1}$, then $\TTd'_{i,-1}(v) ={\bf t}_i' v= (-q_i)^{1/2} f_{\ov{0},\bvs_\star}(B_i)v\in M_{\ov{\lambda-\alpha_i}}$, by \eqref{eq:slambda}. Similar for $\TTd''_{i,+1}(v)$.

(ii) $\boxed{c_{i,\tau i}=0}$. 
In this case, $\bs_i=s_i s_{\tau i}$, and then by \cref{lem:diagmu} we have
\begin{align} \label{rilambda_ii}
\begin{split}
\bs_i\lambda &=\lambda-\langle h_i ,\lambda\rangle \alpha_i-\langle h_{\tau i} ,\lambda\rangle \alpha_{\tau i},
\\
\ov{\bs_i\lambda} &=\ov{\lambda-\la_{i,\tau}\alpha_i}. 
\end{split}
\end{align}
On the other hand, by definition \eqref{diagti'}--\eqref{diagti''}, we have ${\bf t}_i'v,{\bf t}_i''v\in M_{\ov{\lambda-\la_{i,\tau}\alpha_i}}$ and hence the desired statement follows.

(iii) $\boxed{c_{i,\tau i}=-1}$. 
In this case, $\bs_i=s_i s_{\tau i} s_i$, and then by \cref{lem:qsmu} we have
\begin{align} \label{rilambda_iii}
\begin{split}
\bs_i\lambda &=\lambda-(\langle h_i ,\lambda\rangle +\langle h_{\tau i} ,\lambda\rangle)( \alpha_i+\alpha_{\tau i}),
\\
\ov{\bs_i\lambda} &=\ov{\lambda}.  
\end{split}
\end{align}
On the other hand, by definition \eqref{qss-ti'}--\eqref{qss-ti''}, we have ${\bf t}_i'v,{\bf t}_i''v\in M_{\ov{\lambda }}$ and hence the desired statement follows.
\end{proof}


\subsection{Integral braid group symmetries on $\dot{\U}^\imath$}
\label{subsec:split_dUi}

Let $M\in \mathcal{C}$. A free $\A$-submodule $M_\A$ of $M$ is called an $\A$-form if $\dUiA M_\A=M_\A$ and $M_\A \otimes_\A \F=M$. 

For any integrable weight $\U$-module $N$ with an $\A$-form $N_\A$, by \cref{lem:integrable}, $N$ viewed as a $\Ui$-module by restriction lies in $\mathcal{C}$; moreover, by \cite[Lemma 3.20]{BW18b}, $N_\A$ is also an $\A$-form for the $\Ui$-module $N$. In this way, we have obtained many (though not all) integrable $\Ui$-modules with $\A$-forms. 

\begin{lemma}\label{int-mod}
Let $M\in \mathcal{C}$ with an $\A$-form $M_\A$. Then the linear operators $\TTd'_{i,\pm 1},\TTd''_{i,\pm 1}$ preserve $M_\A$.
\end{lemma}

\begin{proof}
By definition \eqref{ti'}-\eqref{qss-ti''}, ${\bf t}_i',{\bf t}_i''\in \widehat{\U}^\imath_\A$ and then the statement follows.
\end{proof}

We now lift the operators $\TTd'_{i,-1} $ and $\TTd''_{i,+1}$ to the algebra $\dUi$.

\begin{theorem}
\label{thm:integralTi}
Let $i\in \Ifin$. 
\begin{enumerate}
    \item 
For any $x\in \dUi$, there exist unique $x',x''\in \dUi$ such that
\begin{align}\label{eq:x'x''}
x' {\bf t}_i' ={\bf t}_i' x,\qquad  x'' {\bf t}_i'' ={\bf t}_i'' x.
\end{align}
\item
The maps $\TTd'_{i,-1} : \dUi \rightarrow \dUi\; (x\mapsto x')$ and $\TTd''_{i,+1}: \dUi \rightarrow \dUi \; (x\mapsto x'')$ are automorphisms of the algebra $\dUi$; see \eqref{dotTi:idemp}--\eqref{dotTi''iii} or Table~\ref{table:rkone-int}--\ref{table:rktwo-int2} for explicit formulas. Moreover, $\TTd'_{i,-1} \TTd''_{i,+1} = \TTd''_{i,+1} \TTd'_{i,-1} =\mathrm{Id}.$
\item 
Let $M$ be an integrable $\Ui$-module. For $v\in M$ and $x\in \dUi$, we have
$\TTd'_{i,-1} (xv) =\TTd'_{i,-1}(x) \TTd'_{i,-1}(v)$ and $\TTd''_{i,+1}(xv) =\TTd''_{i,+1}(x) \TTd''_{i,+1}(v)$.
\item 
The automorphisms $\TTd'_{i,-1} $ and $\TTd''_{i,+1}$ of $\dUi$ preserve the $\A$-subalgebra $\dUiA$. 
\end{enumerate}
\end{theorem}

While the formulation of \cref{thm:integralTi} is uniform for $i\in \Ifin$ in all 3 types, the explicit formulas can only be presented case-by-case. 
The automorphisms $\TTd'_{i,-1} : \dUi \rightarrow \dUi\; (x\mapsto x')$ and $\TTd''_{i,+1}: \dUi \rightarrow \dUi \; (x\mapsto x'')$ are given by 
\begin{align}
\label{dotTi:idemp}
\TTd'_{i,-1} (\one_{\ov{\lambda}})
& =\one_{\ov{\bs_i\lambda}},
\qquad
\TTd''_{i,+1} (\one_{\ov{\lambda}}) 
=\one_{\ov{\bs_i\lambda}},
\end{align}
and
$\boxed{\text{(i) } c_{i,\tau i}=2}:$ 
\begin{align}
\TTd'_{i,-1}(B_j\one_{\ov{\lambda}}) 
&=
\begin{dcases}
\TT'_{i,-1}(B_j)\one_{\ov{\bs_i\lambda}}, & \text{for } c_{ij}\text{ even},
\\
(-q_i)^{1/2}\TT'_{i,-1}(B_j)\one_{\ov{\bs_i\lambda}}, & \text{for } c_{ij}\text{ odd } \& \  \ov{\la_i}=\ov{0},
\\
(-q_i)^{-1/2}\TT'_{i,-1}(B_j)\one_{\ov{\bs_i\lambda}}, & \text{for } c_{ij}\text{ odd } \& \  \ov{\la_i}=\ov{1}; 
\end{dcases}
\label{dotTi'}
\\
\TTd''_{i,+1} (B_j\one_{\ov{\lambda}}) 
&=
\begin{dcases}
\TT''_{i,+1}(B_j)\one_{\ov{\bs_i\lambda}}, & \text{for } c_{ij}\text{ even},
\\
(-q_i)^{-1/2}\TT''_{i,+1}(B_j)\one_{\ov{\bs_i\lambda}}, & \text{for } c_{ij}\text{ odd } \& \  \ov{\la_i}=\ov{0},
\\
(-q_i)^{1/2}\TT''_{i,+1}(B_j)\one_{\ov{\bs_i\lambda}}, & \text{for } c_{ij}\text{ odd } \& \ \ov{\la_i}=\ov{1}.
\end{dcases}
\label{dotTi''}
\end{align}

$\boxed{\text{(ii) } c_{i,\tau i}=0}:$  
\begin{align}
\TTd'_{i,-1} (B_j \one_{\ov{\lambda}}) = (-1)^{(c_{ij}-c_{\tau i,j})/2} q_i^{(c_{\tau i, j}-c_{ij})/2}\TT'_{i,-1} (B_j) \one_{\ov{\bs_i\lambda}},
\\
\TTd''_{i,+1} (B_j \one_{\ov{\lambda}}) = (-1)^{(c_{ij}-c_{\tau i,j})/2} q_i^{(c_{\tau i,j}-c_{ij})/2}\TT''_{i,+1} (B_j) \one_{\ov{\bs_i\lambda}}.
\end{align}

$\boxed{\text{(iii) } c_{i,\tau i}=-1}:$
\begin{align}
\TTd'_{i,-1} (B_j \one_{\ov{\lambda}})=\TT'_{i,-1} (B_j) \one_{\ov{\bs_i\lambda}},
\qquad
\TTd''_{i,+1} (B_j \one_{\ov{\lambda}})=\TT''_{i,+1} (B_j) \one_{\ov{\bs_i\lambda}}.
\label{dotTi''iii}
\end{align}

\begin{proof}
Let us focus on Case (i) when $c_{i,\tau i}=2$, i.e., $i=\tau i$.

(1). Since ${\bf t}_i',{\bf t}_i''$ are invertible in $\widehat{\U}^\imath$, it is clear that $x',x'' $ are unique if they exist.

Suppose that $x',y'$ satisfying the first identity in \eqref{eq:x'x''} exist, for $x,y\in \dUi$. Then we have $x'y'{\bf t}_i' ={\bf t}_i'xy$. i.e., $(xy)':=x'y'$ satisfies \eqref{eq:x'x''}. Hence $\TTd'_{i,-1}: \dUi \rightarrow \dUi\; (x\mapsto x')$ is a well-defined endomorphism of $\dUi$, once we show that $x'$ satisfying the first identity in \eqref{eq:x'x''} exists for all generators $x$ of $\dUi$. 

(The reduction from proving $\TTd''_{i,+1}: \dUi \rightarrow \dUi \; (x\mapsto x'')$ is a well-defined endomorphism of $\dUi$ to the existence for $x''$ satisfying the second identity in \eqref{eq:x'x''} to $x$ being generators is entirely similar.)

Recall that $\dUi$ are generated by $\{\one_{\ov{\lambda}}, \  B_j\one_{\ov{\lambda} } \mid \ov{\lambda}\in X_\imath, j\in \I \}$. For  $x$ being any of these generators, $x',x''$ are given by the right-hand sides of the formulas \eqref{dotTi:idemp}--\eqref{dotTi''}. 

Using \eqref{eq:fsplit} and \eqref{eq:fsplit'}, one directly check that these elements satisfy the identities \eqref{eq:x'x''}. This proves (1). 

(2). We already showed above that $\TTd'_{i,-1}$ and $\TTd''_{i,+1}$ are endomorphisms of $\dUi$. Combining the identities \eqref{eq:x'x''} and  Lemma~\ref{lem:inverse}, we readily see that the maps $\TTd'_{i,-1}$ and $\TTd''_{i,+1}$ are mutual inverses; in particular, they are automorphisms of $\dUi$. The formulas \eqref{dotTi:idemp}--\eqref{dotTi''} were already obtained in the process of proving (1) above.

(3). The proof of the two formulas are similar, and let us verify the first one. One clearly reduces the verification to the cases when $x$ is any of the generators of $\dUi$ and for $v\in M_{\ov\la}$, that is, we shall verify the following identities:
\begin{align}
    \TTd'_{i,-1}(\one_{\ov\mu} v) &= \TTd'_{i,-1}(\one_{\ov\mu}) \, \TTd'_{i,-1}(v), \quad \text{for } \ov\mu\in X_\imath,
    \label{Tonev} \\
    \TTd'_{i,-1}(B_j v) &= \TTd'_{i,-1}(B_j \one_{\ov\la}) \, \TTd'_{i,-1}(v).
    \label{TBjv}
\end{align}
If $\ov\mu \neq \ov\la$, the identity \eqref{Tonev} holds trivially as both sides are $0$. If $\ov\mu = \ov\la$, both sides of \eqref{Tonev} are equal to $\TTd'_{i,-1}(v)$ (since this vector has iweight $\ov{\bs_i\la}$ by Lemma~\ref{lem:iweight}).

By Theorem~\ref{thm:split}, we have $\TT'_{i,-1}(B_j v) = \TT'_{i,-1}(B_j) \, \TT'_{i,-1}(v).$ The difference of the three items in this identity from their counterparts in \eqref{TBjv} is some suitable integer powers of $(-q_i)^{1/2}$. Hence, the proof of \eqref{TBjv} reduces to checking the compatibility of these powers of $(-q_i)^{1/2}$. Write $\ov{\la_i}=\ov{\langle h_i,\la\rangle}\in \Z_2$. By \eqref{ti'f} and \eqref{dotTi'}, LHS \eqref{TBjv} contributes to a $(-q_i)^{1/2}$-power $\delta_{\ov{\la_i+c_{ij}},\ov{1}}$ while RHS \eqref{TBjv} contributes to a $(-q_i)^{1/2}$-power $(-1)^{\delta_{\ov{\la_i},\ov{1}}} \delta_{\ov{c_{ij}},\ov{1}} +\delta_{\ov{\la_i},\ov{1}}$. One checks case-by-case that $\delta_{\ov{\la_i+c_{ij}},\ov{1}} =(-1)^{\delta_{\ov{\la_i},\ov{1}}} \delta_{\ov{c_{ij}},\ov{1}} +\delta_{\ov{\la_i},\ov{1}}$. This proves \eqref{TBjv}. 

(4). 
We only prove that  $\TTd'_{i,-1}$ preserves the $\A$-subalgebra $\dUiA$.
Let $x\in \dUiA$ and let $\lambda\in X$.
By \cite[Lemma~ 3.20]{BW18b}, it suffices to show that $\TTd'_{i,-1}(x)\one_\lambda\in \dot{\U}_\A$. Let $M$ be an integrable $\U$-module with $\A$-form $M_\A$. Then $M_\A$ is preserved by $\TTd'_{i,-1},\TTd''_{i,+1}$ thanks to  \cref{lem:integrable} and \cref{int-mod}. By \cref{lem:inverse}, for any $v\in M_\A$, we have
\begin{align}\label{Tx-mod}
\TTd'_{i,-1}(x)(v)=\TTd'_{i,-1}\big(x\TTd''_{i,+1}(v)\big).
\end{align}
By \cref{int-mod}, both $\TTd'_{i,-1},\TTd''_{i,+1}$ preserve $M_\A$. Hence, since $x M_\A \subset M_\A$, the right-hand side of \eqref{Tx-mod} lies in $M_\A$ and so  $\TTd'_{i,-1}(x)(v)\in M_\A$. Now take $M={}^\omega L(\mu) \otimes L(\nu)$ with $\nu-\mu=\la$ and $v=\xi_{-\mu} \otimes \eta_\nu$; here ${}^\omega L(\mu)$ denotes the lowest weight $\U$-module with lowest weight vector $\xi_{-\mu}$. Since $\TTd'_{i,-1}(x)(\xi_{-\mu} \otimes \eta_\nu)\in {}^\omega L(\mu)_\A \otimes_\A L(\nu)_\A$ with $\nu-\mu=\la$ for $\mu,\nu\gg 0$,  we conclude that $\TTd'_{i,-1}(x)\one_\lambda\in \dot{\U}_\A$.

Therefore, we have completed the proof of the theorem in Case (i). 

The proofs of (1)--(4) for Case (ii)--(iii) are entirely similar. In place of \eqref{eq:fsplit} and \eqref{eq:fsplit'} in Case (i), we shall use \eqref{eq:Tzm} and \eqref{eq:Tzm'} in Case (ii) and use \eqref{eq:Tgnu} and \eqref{eq:Thnu} in Case (iii). We skip the details. 
\end{proof}

Recall the idivided powers $X_{j,n,\ov{t}}$ from \eqref{Xjn}, and closed formulas for $\TTd_{i,-1} (X_{j,n,\ov{t}})$ in $\Ui$ are given in \eqref{eq:splitT1}, \eqref{TiBjn:diag}, \eqref{eq:qs-Tdiv1}.  

\begin{proposition}
\label{prop:dotTiDP}
For any $n\ge 0$, we have

(i) $\boxed{c_{i,\tau i}=2}:$ 
\begin{align}
\TTd'_{i,-1}(X_{j,n,\ov{\lambda}_j}\one_{\ov{\lambda}}) 
&=
\begin{dcases}
\TT'_{i,-1}(X_{j,n,\ov{\lambda}_j})\one_{\ov{\bs_i\lambda}}, & \text{ for } nc_{ij} \text{ even},
\\
(-q_i)^{1/2}\TT'_{i,-1}(X_{j,n,\ov{\lambda}_j})\one_{\ov{\bs_i\lambda}}, & \text{ for } nc_{ij} \text{ odd } \& \  \ov{\la_i}=\ov{0},
\\
(-q_i)^{-1/2}\TT'_{i,-1}(X_{j,n,\ov{\lambda}_j})\one_{\ov{\bs_i\lambda}}, & \text{ for } nc_{ij}\text{ odd } \& \  \ov{\la_i}=\ov{1}; 
\end{dcases}
\label{dotTi'n}
\end{align}

(ii) $\boxed{c_{i,\tau i}=0}:$  
\begin{align}
\TTd'_{i,-1} (X_{j,n,\ov{\lambda}_j} \one_{\ov{\lambda}}) = (-1)^{(nc_{ij}-nc_{\tau i,j})/2} q_i^{(nc_{\tau i, j}-nc_{ij})/2}\TT'_{i,-1} (X_{j,n,\ov{\lambda}_j}) \one_{\ov{\bs_i\lambda}},
\end{align}

(iii) $\boxed{c_{i,\tau i}=-1}:$
\begin{align}
\TTd'_{i,-1} (X_{j,n,\ov{\lambda}_j} \one_{\ov{\lambda}})=\TT'_{i,-1} (X_{j,n,\ov{\lambda}_j}) \one_{\ov{\bs_i\lambda}}.
\label{dotTi'iiin}
\end{align}
\end{proposition}

\begin{proof}
The arguments for all 3 cases are similar, and we only provide details in verifying \eqref{dotTi'n}. If $\tau j\neq j$, then 
\[
X_{j,n,\ov{\lambda}_j}\one_{\lambda}=\dv{j}{n}\one_{\lambda}=\frac{1}{[n]_j!} B_j \one_{\lambda-(n-1)\alpha_j}\cdots B_j \one_{\lambda-\alpha_j}B_j \one_{\lambda}.
\]
Note that $\langle h_i, \lambda-k\alpha_j \rangle=\lambda_i-k c_{ij}$. Hence, \eqref{dotTi'n} can be directly derived from \eqref{dotTi'}.

If $\tau j=j$, then $X_{j,n,\ov{\lambda}_j}$ is a linear combination of $\dv{j}{n-2t}$ for $0\le t\le \lfloor n/2\rfloor$. Let us assume $\ov{\lambda}_j=\ev$ (the other case for $\ov{\lambda}_j=\odd$ is entirely similar). In this case, using \eqref{dotTi'}, we have for any $0\le t\le \lfloor n/2\rfloor$
\begin{align}\label{eq:TTdBjn}
\TTd'_{i,-1}(\dv{j}{n-2t}\one_{\ov{\lambda}}) =
\begin{dcases}
\TT'_{i,-1}(\dv{j}{n-2t})\one_{\ov{\bs_i\lambda}}, & \text{for } nc_{ij} \text{ even},
\\
(-q_i)^{1/2}\TT'_{i,-1}(\dv{j}{n-2t})\one_{\ov{\bs_i\lambda}}, & \text{for } nc_{ij}\text{ odd}.
\end{dcases}
\end{align} 
In particular, the scalar multiple between $\TTd'_{i,-1}(\dv{j}{n-2t} \one_{\ov{\lambda}})$ and $\TT'_{i,-1}(\dv{j}{n-2t} )\one_{\ov{\bs_i\lambda}}$ does not depend on $t\ge 0$. The formula \eqref{eq:TTdBjn} together with this observation implies \eqref{dotTi'n}.
\end{proof}

\begin{remark}
It is not manifestly clear that the formulas in \cref{prop:dotTiDP} are integral, i.e., $\TT'_{i,-1}(X_{j,n,\ov{\lambda}_j}\one_{\lambda})\in \dUiA$, but this indeed holds by \cref{thm:integralTi}(4). Closed formulas for $\TT''_{i,+1}(X_{j,n,\ov{\lambda}_j}\one_{\lambda})$ can be obtained similarly.
\end{remark}
\subsection{Relative braid group relations}

Let $i, j\in \fwItau$ (see \eqref{def:Ifinite}) be such that $j\neq i,\tau i$. We use $m_{ij}$ to denote the order of $\bs_i \bs_j$ in $W^\circ$, where $m_{ij}\in \{2,3,4,6,\infty\}$. Recall from \eqref{eq:ovc} the Cartan integers $\oc_{ij}$ of the relative root system. Then $m_{ij}<\infty$ if and only if $0\leq \oc_{ij} \oc_{ji}\leq 3$, and
\begin{align}
\label{eq:mij}
    m_{ij} =
    \begin{dcases}
        2, & \text{ if } \oc_{ij} \oc_{ji} =0
        \\
        3, & \text{ if } \oc_{ij} \oc_{ji} =1
        \\
        4, & \text{ if } \oc_{ij} \oc_{ji} =2
        \\
        6, & \text{ if } \oc_{ij} \oc_{ji} =3.
    \end{dcases}
\end{align}

For $i\neq j\in \wI$ such that $i=\tau i$ and $j=\tau j$, we have $\oc_{ij} =c_{ij}$. We first verify the braid relations in the most involved (split) case separately.

\begin{theorem}
\label{thm:TTd_braid_split}
    Let $i\neq j\in \wI$ such that $i=\tau i$, $j=\tau j$ and $m_{ij} <\infty$. When acting on $\dUi$ or on any integrable $\Ui$-module $M$ the following braid relations hold:
    \begin{align*}
      \small  \underbrace{\TTd'_{i,-1} \TTd'_{j,-1} \TTd'_{i,-1} \ldots}_{m_{ij}} 
        =
        \underbrace{\TTd'_{j,-1} \TTd'_{i,-1} \TTd'_{j,-1} \ldots}_{m_{ij}} \ ,
        \qquad
        \underbrace{\TTd''_{i,+1} \TTd''_{j,+1} \TTd''_{i,+1} \ldots}_{m_{ij}} 
        =
        \underbrace{\TTd''_{j,+1} \TTd''_{i,+1} \TTd''_{j,+1} \ldots}_{m_{ij}} \ .
    \end{align*}
\end{theorem}

\begin{proof}
    The proofs for the two braid relations above are similar, and we will prove only the first one. 

We start with showing that the first braid group relation when acting on all integrable $\Ui$-modules imply the first braid group relations when acting on $\dUi$. Indeed, for any $v\in M$ and $x\in \dUi$, by Theorem~\ref{thm:integralTi}(3), we have 
    \begin{align*}
       \TTd'_{i,-1} \TTd'_{j,-1} \TTd'_{i,-1} \ldots (xv)
        &=  \TTd'_{i,-1} \TTd'_{j,-1} \TTd'_{i,-1} \ldots (x) \cdot \TTd'_{i,-1} \TTd'_{j,-1} \TTd'_{i,-1} \ldots (v),
        \\
        \TTd'_{j,-1} \TTd'_{i,-1} \TTd'_{j,-1} \ldots (xv)
        &= \TTd'_{j,-1} \TTd'_{i,-1} \TTd'_{j,-1} \ldots (x) \cdot \TTd'_{j,-1} \TTd'_{i,-1} \TTd'_{j,-1} \ldots (v).
    \end{align*}
Implicitly, we have assumed the number of $\TTd$'s is $m_{ij}$ in each product in the above formulas. By the assumption, the left-hand sides above are equal and the two expressions acting on $v$ on the right-hand sides are also equal. This implies that 
\[
\underbrace{\TTd'_{i,-1} \TTd'_{j,-1} \TTd'_{i,-1} \ldots}_{m_{ij}} (x)=\underbrace{\TTd'_{j,-1} \TTd'_{i,-1} \TTd'_{j,-1} \ldots}_{m_{ij}} (x)
\] when acting on an arbitrary integrable $\Ui$-module, and so this identity must hold in $\dUi$. 

To show that the first braid relation holds for all integrable $\Ui$-modules, it suffices to show that it holds for all integrable $\U$-modules. Indeed, assuming the first braid relation holds for all integrable $\U$-modules, we have that 
\begin{align}\label{eq:titj}
\underbrace{\bt'_{i} \bt'_{j} \bt'_{i} \ldots}_{m_{ij}} 
=\underbrace{\bt'_{j} \bt'_{i} \bt'_{j} \ldots }_{m_{ij}}
\end{align}
when acting on any integrable $\U$-modules. This forces that the identity \eqref{eq:titj} holds in $\widehat{\U}^\imath$ and hence the identity \eqref{eq:titj} holds on any integrable $\Ui$-module, i.e., the first braid relation holds on all integrable $\Ui$-modules by definition \eqref{TTdM}.

It remains to verify the first braid relation holds for any integrable $\U$-module $M$. Let $v\in M_{\ov\la}$. 

Let $i,j \in \wI$ such that $c_{ij}=-1 \ \& \ c_{ji}=-3$. Set $\alpha :=c_{ij} c_{ji}$, and $m_{ij}=6$. The braid group relation
\begin{align}
\label{TTG2}
    \TT'_{i,-1} \TT'_{j,-1}\TT'_{i,-1} \TT'_{j,-1}\TT'_{i,-1} \TT'_{j,-1}(v) = 
    \TT'_{j,-1} \TT'_{i,-1}\TT'_{j,-1} \TT'_{i,-1}\TT'_{j,-1} \TT'_{i,-1}(v)
\end{align}
already holds by \cite[Theorem~10.6]{WZ23}. We claim that the dot version of the braid group relation \eqref{TTG2} holds: 
\begin{align}
\label{TTdG2}
    \TTd'_{i,-1} \TTd'_{j,-1}\TTd'_{i,-1} \TTd'_{j,-1}\TTd'_{i,-1} \TTd'_{j,-1}(v) = 
    \TTd'_{j,-1} \TTd'_{i,-1}\TTd'_{j,-1} \TTd'_{i,-1}\TTd'_{j,-1} \TTd'_{i,-1}(v).
\end{align}
For either $k=i$ or $k=j$, the actions of $\TTd'_{k,-1}$ and $\TT'_{k,-1}$ on a vector of iweight ${\ov\mu}$ differ by a scalar multiple $(-q_k)^{1/2 \cdot \delta_{\ov\mu_k,\ov{1}}}$ due to \eqref{ti'f}. 

The parity sequence $\ov\mu_k$ (where $k$ alternates between $j,i$) for the iweights $\ov\mu$ of the 6 vectors $v, \TTd'_{j,-1}(v), \TTd'_{i,-1} \TTd'_{j,-1}(v)$, $\TTd'_{j,-1}\TTd'_{i,-1} \TTd'_{j,-1}(v)$, $\TTd'_{i,-1} \TTd'_{j,-1}\TTd'_{i,-1} \TTd'_{j,-1}(v)$, and the vector $\TTd'_{j,-1}\TTd'_{i,-1} \TTd'_{j,-1}\TTd'_{i,-1} \TTd'_{j,-1}(v)$ is given by 
\begin{align}
    \label{paritysequence}
\ov{\la_j}, \ \brown{\ov{\la_i+\la_j}}, \ \ov{\la_i}, \ \brown{\ov{\la_j}}, \ \ov{\la_i+\la_j},\text{ and }   \brown{\ov{\la_i}}. 
\end{align}
This follows by a direct computation using the following formulas:
\begin{align}
\label{pairingG2}
\begin{split}
    \langle h_j,\la \rangle &=\la_j, 
    \\
\langle \brown{h_i},s_j\la \rangle &=\langle h_i-c_{ij}h_j,\la \rangle, \\ 
\langle h_j,s_is_j\la \rangle &=\langle (\alpha-1)h_j-c_{ji}h_i,\la \rangle, \\ 
\langle \brown{h_i},s_{jij}\la \rangle &=\langle (\alpha-1)h_i+(2-\alpha)c_{ij}h_j,\la \rangle, \\ 
\langle h_j,s_{ijij}\la \rangle &=\langle h_j-c_{ji}h_i,\la \rangle, \\ 
\langle \brown{h_i},s_{jijij}\la \rangle &= \langle h_i,\la \rangle. 
\end{split}
\end{align}
Here we have use the shorthand $s_{jij} =s_js_is_j$ and so on.

It follows from \eqref{paritysequence} that LHS \eqref{TTdG2} is equal to LHS \eqref{TTG2} times the following scalar
\begin{align}
\label{scaleG2}
(-q_i)^{1/2 \cdot (\delta_{\brown{\ov{\la_i+\la_j}},\ov{1}} +\delta_{\brown{\ov{\la_j}},\ov{1}} +\delta_{\brown{\ov{\la_i}},\ov{1}})}
(-q_j)^{1/2 \cdot (\delta_{\ov{\la_j},\ov{1}} +\delta_{\ov{\la_i},\ov{1}} +\delta_{\ov{\la_i+\la_j},\ov{1}})}. 
\end{align}
Note this scalar is fixed upon interchange of indices $i,j$, and this scalar (with $i,j$ interchanged) is exactly the scalar difference between 
the RHS  of \eqref{TTdG2} and \eqref{TTG2}. This proves \eqref{TTdG2} in case when $c_{ij}=-1\ \& \ c_{ji}=-3$.

The proof of the first braid relation for an integrable $\Ui$-module $M$ is similar and easier for the remaining 3 cases:
\begin{enumerate}
    \item[(0)] $c_{ij}=c_{ji}=0$; 
\item[(1)] $c_{ij}=c_{ji}=-1$;   
\item[(2)] $c_{ij}=-1, c_{ji}=-2$. 
\end{enumerate}
The counterpart of the scalar \eqref{scaleG2} in Case (0) is computed using the first two formulas in \eqref{pairingG2} to be $(-q_i)^{1/2 \cdot \delta_{\ov{\la_i},\ov{1}}} (-q_j)^{1/2 \cdot \delta_{\ov{\la_j},\ov{1}}}$. The counterpart of the scalar \eqref{scaleG2} in Case (1) is computed using the first 3 formulas with $\alpha=1$ in \eqref{pairingG2} to be $(-q_i)^{1/2 \cdot (\delta_{\ov{\la_i+\la_j},\ov{1}} +\delta_{\ov{\la_j},\ov{1}} +\delta_{\ov{\la_i},\ov{1}})}$. The counterpart of the scalar \eqref{scaleG2} in Case (2) is computed using the first four formulas with $\alpha=2$ in \eqref{pairingG2} to be 
$ 
(-q_i)^{1/2 \cdot (\delta_{\ov{\la_i+\la_j},\ov{1}}  +\delta_{\ov{\la_i},\ov{1}}) }
(-q_j)^{1/2 \cdot 2\delta_{\ov{\la_j},\ov{1}}}.
$ 
All these scalar multiples in Case (0)-(2) are invariant under the interchange of $i,j$.

This completes the proof of \cref{thm:TTd_braid_split}. 
\end{proof}


Now we prove the braid group relations in general. 

\begin{theorem}
\label{thm:TTd_braid}
    Let $i,j\in \fwItau$ such that $j\neq i,\tau i$ and $m_{ij}<\infty$. When acting on $\dUi$ or on any integrable $\Ui$-module $M$ the following braid relations hold:
    \begin{align}
      \small  \underbrace{\TTd'_{i,-1} \TTd'_{j,-1} \TTd'_{i,-1} \ldots}_{m_{ij}} 
        &\small =
        \underbrace{\TTd'_{j,-1} \TTd'_{i,-1} \TTd'_{j,-1} \ldots}_{m_{ij}} \ ,
        \label{braidTdot1} \\
    \small    \underbrace{\TTd''_{i,+1} \TTd''_{j,+1} \TTd''_{i,+1} \ldots}_{m_{ij}} 
        &\small =
        \underbrace{\TTd''_{j,+1} \TTd''_{i,+1} \TTd''_{j,+1} \ldots}_{m_{ij}} \ .
        \label{braidTdot2}
    \end{align}
\end{theorem}

\begin{proof}
 The proofs for the two braid relations are similar, and we will prove only the first one \eqref{braidTdot1}. By similar arguments in the proof of Theorem~\ref{thm:TTd_braid_split}, it suffices to show that 
\begin{align}\label{eq:qsTTdbraid}
\underbrace{\TTd'_{i,-1} \TTd'_{j,-1} \TTd'_{i,-1} \ldots}_{m_{ij}}(v) 
        &=
\underbrace{\TTd'_{j,-1} \TTd'_{i,-1} \TTd'_{j,-1} \ldots}_{m_{ij}}(v).
\end{align}
for any integrable $\U$-module $M$ and any iweight vector $v\in M_{\ov{\lambda}}$.

Recall from \cite[Theorem~10.6]{WZ23} the following version of the braid relation:
\begin{align}\label{eq:qsTTbraid}
\underbrace{\TT'_{i,-1} \TT'_{j,-1} \TT'_{i,-1} \ldots}_{m_{ij}}(v) 
        &=
\underbrace{\TT'_{j,-1} \TT'_{i,-1} \TT'_{j,-1} \ldots}_{m_{ij}}(v).
\end{align}
We separate the proof of the identity \eqref{eq:qsTTdbraid} into several cases.
\begin{itemize}
\item[(a)] $c_{i,\tau i}=c_{j,\tau j}=2$. i.e., $\tau i=i,\tau j=j$. This case was already treated in Theorem~\ref{thm:TTd_braid_split}.

\item[(b)] $c_{i,\tau i}=2,c_{j,\tau j}=-1$. Then $\tau i=i$. In this case, by \eqref{eq:ovc} and \eqref{eq:mij}, the only possibility is $m_{ij}=2$. This implies that $c_{ij}=c_{i,\tau j}=0$ and then $B_i$ commutes with $B_j,B_{\tau j},k_j$. By \eqref{ti'} and \eqref{qss-ti'}, we have ${\bf t}_i' {\bf t}_j'v={\bf t}_j'{\bf t}_i'v$ as desired. 

\item[(c)] $c_{i,\tau i}=2,c_{j,\tau j}=0$. In this case, by \eqref{eq:ovc} and \eqref{eq:mij}, the only possibilities are $m_{ij}=2,4$. If $m_{ij}=2$, then $c_{ij}=c_{\tau i,j}=0$ and hence \eqref{eq:qsTTdbraid} can be easily verified. 

Let $m=4$. Then $c_{ij}=c_{\tau i,j}=-1$. The LHS of \eqref{eq:qsTTdbraid} differs from the LHS of \eqref{eq:qsTTbraid} by the scalar
\begin{align}\label{eq:qsdot1}
(-q_i)^{\frac{1}{2}\delta_{\ov{\langle h_i, \bs_j \lambda\rangle},\ov{1}}
+\frac{1}{2}\delta_{\ov{\langle h_i,\bs_{j}\bs_i\bs_j\lambda\rangle},\ov{1}}}
(-q_j^{-1})^{-\frac{1}{2}\langle h_j-h_{\tau j},\lambda+\bs_i\bs_j\lambda\rangle},
\end{align}
while the RHS of \eqref{eq:qsTTdbraid} differs from the RHS of \eqref{eq:qsTTbraid} by the scalar
\begin{align}\label{eq:qsdot2}
(-q_i)^{\frac{1}{2}\delta_{\ov{\langle h_i,  \lambda\rangle},\ov{1}}
+\frac{1}{2}\delta_{\ov{\langle h_i,\bs_j\bs_i\lambda\rangle},\ov{1}}}
(-q_j^{-1})^{-\frac{1}{2}\langle h_j-h_{\tau j},\bs_i\lambda+\bs_i\bs_j\bs_i\lambda\rangle}.
\end{align}
By a direct computation, one can show that $\langle h_i, \lambda\rangle$ and $\langle h_i,\bs_{j}\bs_i\bs_j\lambda\rangle$ have the same parity; $\langle h_i, \bs_{j}\lambda\rangle$ and $\langle h_i, \bs_j\bs_i\lambda\rangle$ have the same parity. Moreover, since $\bs_i(h_j-h_{\tau j})=h_j-h_{\tau j}$, we have
\[
\langle h_j-h_{\tau j},\bs_i\lambda+\bs_i\bs_j\bs_i\lambda\rangle
=
\langle h_j-h_{\tau j}, \lambda+ \bs_j\bs_i\lambda\rangle. 
\]
Using these formulas, we see that the two scalars in \eqref{eq:qsdot1} and \eqref{eq:qsdot2} are equal and hence \eqref{eq:qsTTdbraid} is proved.

\item[(d)] $c_{i,\tau i}=0,c_{j,\tau j}=0$. In this case, \eqref{eq:qsTTdbraid} is proved in Lemma~\ref{lem:diagonal} below.

\item[(e)] $c_{i,\tau i}=0,c_{j,\tau j}=-1$. In this case, by \eqref{eq:ovc} and \eqref{eq:mij}, the only possibilities are $m_{ij}=2,4$. If $m_{ij}=2$, then $c_{ij}=c_{\tau i,j}=0$ and hence \eqref{eq:qsTTdbraid} is easily verified. 

Let $m_{ij}=4$. By \eqref{eq:ttgh}, we have $\TTd'_{j,-1}v=\TT'_{j,-1}v$ for any $v\in M$. Hence, both sides of \eqref{eq:qsTTdbraid} differ from those of \eqref{eq:qsTTbraid} by an integer power of $(-q_i^{-1})^{1/2}$. The power arising from the difference of LHS of these two relations is 
\[
-\langle h_i-h_{\tau i},\ov{\bs_j\lambda}+\ov{\bs_j\bs_i\bs_j\lambda}\rangle,
\]
while the power arising from RHS is  
\[
-\langle h_i-h_{\tau i},\ov{\lambda}+\ov{\bs_j\bs_i\lambda}\rangle.
\]
Since $\ov{\bs_j\lambda}=\ov{\lambda}$ (see \eqref{rilambda_iii} in the proof of Lemma~\ref{lem:iweight}), these two powers are equal and hence the relation \eqref{eq:qsTTdbraid} holds in this case.

\item[(f)] $c_{i,\tau i}=-1,c_{j,\tau j}=-1$. In this case, by \eqref{eq:ttgh}, we have
$\TTd'_{i,-1}v=\TT'_{i,-1}v$ and $\TTd'_{j,-1}v=\TT'_{j,-1}v$ for any $v\in M$. Then the desired statement follows from \cite[Theorem~10.6]{WZ23}.
\end{itemize}
This completes the proof of \cref{thm:TTd_braid} (see \cref{lem:diagonal} below for Case (d)). 
\end{proof}

Below we complete Case (d) in the proof of \cref{thm:TTd_braid}.

\begin{lemma}
\label{lem:diagonal}
Let $i,j\in \fwItau$ such that $j\neq i,\tau i,c_{i,\tau i}=c_{j,\tau j}=0$ and $m_{ij}<\infty$. When acting on any integrable $\U$-module $M$, the braid relations \eqref{braidTdot1}--\eqref{braidTdot2} hold.
\end{lemma}

\begin{proof}
We prove the first relation \eqref{braidTdot1}, omitting a similar proof for the second relation \eqref{braidTdot2}. By \eqref{eq:ovc}, the condition $m_{ij}<\infty$ forces either $c_{ij}=c_{\tau i,\tau j}=0$ or $c_{\tau i,j}=c_{i,\tau j}=0$. Without loss of generality, we assume $c_{\tau i,j}=c_{i,\tau j}=0$.

Let $\Ui_{i,j}$ be the subalgebra of $\Ui$ generated by $B_i,B_{\tau i},k_i,B_j,B_{\tau j},k_j$. Let $(\{i,j\},\cdot)$ be the Cartan datum such that $c_{ij}=2\frac{i\cdot j}{i\cdot i},c_{ji}=2\frac{i\cdot j}{j\cdot j}$. Let $\un{\U}$ be the quantum group  associated with $(\{i,j\},\cdot)$. Then we have an algebra isomorphism $\Ui_{i,j}\overset{\sim}{\longrightarrow} \un{\U}$ given by
\begin{align*}
&B_i\mapsto F_1,\qquad B_{\tau i}\mapsto E_1,\qquad k_i\mapsto K_1,
\\
&B_j\mapsto F_2,\qquad B_{\tau j}\mapsto E_2,\qquad k_j\mapsto K_2.
\end{align*}

Under this isomorphism, an integrable $\Ui$-module $M$ can be viewed as an integrable $\un{\U}$-module; moreover, by \eqref{diagti'}, the actions of $\TTd'_{i,-1} ,\TTd'_{j,-1}$ on $M$ are identified with the actions of Lusztig's symmetries $T'_{i,-1}, T'_{j,-1}$ \cite[\S2.1-2.3]{Lus09}. Since $T'_{i,-1}, T'_{j,-1}$ satisfy the braid relation when acting on $M$, $\TTd'_{i,-1} ,\TTd'_{j,-1}$ also satisfy the braid relation \eqref{braidTdot1}.
\end{proof}

The following relative braid relations when acting on integrable $\U$-modules with weights bounded above were established in \cite{WZ23}.

\begin{corollary}
\label{cor:braid}
    Let $i,j\in \fwItau$ be such that $j\neq i,\tau i$ and $m_{ij}<\infty$. When acting on any integrable $\Ui$-module $M$ the following braid relations hold:
    \begin{align*}
    \small
        \underbrace{\TT'_{i,-1} \TT'_{j,-1} \TT'_{i,-1} \ldots}_{m_{ij}} 
        =
        \underbrace{\TT'_{j,-1} \TT'_{i,-1} \TT'_{j,-1} \ldots}_{m_{ij}} \ ,
        \qquad
        \underbrace{\TT''_{i,+1} \TT''_{j,+1} \TT''_{i,+1} \ldots}_{m_{ij}} 
        =
        \underbrace{\TT''_{j,+1} \TT''_{i,+1} \TT''_{j,+1} \ldots}_{m_{ij}} \ .
    \end{align*}
\end{corollary}

\begin{proof}
    From the proofs of \cref{thm:TTd_braid_split,thm:TTd_braid} (which goes through the identity \eqref{eq:titj}), we can read off the following equivalence: when acting on integrable $\Ui$-modules the braid relation between $\TT'_{i,-1}$ and $\TT'_{j,-1}$ is equivalent to the braid relation between $\TTd'_{i,-1}$ and $\TTd'_{j,-1}$. The latter holds by \cref{thm:TTd_braid}, and we are done.
\end{proof}
By \cref{thm:TTd_braid,cor:braid}, for any $w\in W^\circ$, we have well-defined relative braid group symmetries $\TT'_{w,-1},\TT''_{w,+1}$ on $M$ as well as $\TTd'_{w,-1},\TTd''_{w,+1}$
on $M$ and $\dUi$.

%
\section{Tables on formulas for relative braid group action} \label{sec:tables}

We list the formulas for $\TTd'_{i,-1}$, $\TTd''_{i,+1}$ when acting on (1) $B_{i}\one_{\ov{\lambda}}, B_{\tau i}\one_{\ov{\lambda}}$, (2) $B_{j}\one_{\ov{\lambda}}$, for $j \neq i, \tau i$ in the 3 tables.

\begin{table}[H]
\caption{Rank one formulas for $\TTd'_{i,-1}$, $\TTd''_{i,+1}$,
\text{ for } $i \in \Ifin$\qquad\qquad ${}$}
     \label{table:rkone-int}
\resizebox{5.5in}{!}{%
\begin{tabular}{| c | c | c |c|c|}
\hline
&
\begin{tikzpicture}[baseline=0]
\node at (0, -0.15) {
\qquad$\TTd'_{i,-1}(B_i\one_{\ov{\lambda}})$\qquad};
\end{tikzpicture}
&
\begin{tikzpicture}[baseline=0]
\node at (0, -0.15) {
\qquad $\TTd'_{i,-1}(B_{\tau i}\one_{\ov{\lambda}})$\qquad};
\end{tikzpicture}
&
\begin{tikzpicture}[baseline=0]
\node at (0, -0.15) {
\qquad $\TTd''_{i,+1}(B_i\one_{\ov{\lambda}})$\qquad};
\end{tikzpicture}
&
\begin{tikzpicture}[baseline=0]
\node at (0, -0.15) {
\qquad $\TTd''_{i,+1}(B_{\tau i}\one_{\ov{\lambda}})$\qquad};
\end{tikzpicture}
\\
\hline 

\begin{tikzpicture}[baseline=0]
\node at (0, -0.15) {$c_{i,\tau i}=2$ };
\end{tikzpicture}
&
\begin{tikzpicture}[baseline=0]
\node at (0, -0.15) {$ B_i \one_{\ov{\bs_i\lambda}}$ };
\end{tikzpicture}
&
\begin{tikzpicture}[baseline=0]
\node at (0, -0.15) {$ B_i \one_{\ov{\bs_i\lambda}}$ };
\end{tikzpicture}
&
\begin{tikzpicture}[baseline=0]
\node at (0, -0.15) {$ B_i \one_{\ov{\bs_i\lambda}}$ };
\end{tikzpicture}
&
\begin{tikzpicture}[baseline=0]
\node at (0, -0.15) {$ B_i \one_{\ov{\bs_i\lambda}}$ };
\end{tikzpicture}
\\
\hline
\begin{tikzpicture}[baseline=0]
\node at (0, -0.15) {\large  $c_{i,\tau i}=0$ };
\end{tikzpicture}
&
\begin{tikzpicture}[baseline=0]
\node at (0, -0.15) {
$-q_i^{-\la_{i,\tau}} B_{\tau i}\one_{\ov{\bs_i\lambda}}$ 
};
\end{tikzpicture}
&
\begin{tikzpicture}[baseline=0]
\node at (0, -0.15) {
$ -q_i^{\la_{i,\tau}+2} B_{i}\one_{\ov{\bs_i\lambda}}$ };
\end{tikzpicture}
&
\begin{tikzpicture}[baseline=0]
\node at (0, -0.15) {
$-q_i^{\la_{i,\tau}-2} B_{\tau i}\one_{\ov{\bs_i\lambda}}$ 
};
\end{tikzpicture}
&
\begin{tikzpicture}[baseline=0]
\node at (0, -0.15) {
$ -q_i^{-\la_{i,\tau}} B_{i}\one_{\ov{\bs_i\lambda}}$ };
\end{tikzpicture}
\\
\hline
\begin{tikzpicture}[baseline=0]
\node at (0, -0.15) { $c_{i,\tau i}=-1$};
\end{tikzpicture}
&
\begin{tikzpicture}[baseline=0]
\node at (0, -0.15) {
$ q_i^{-\la_{i,\tau}+1}\vs_{i}  B_i \one_{\ov{\bs_i\lambda}}$ 
};
\end{tikzpicture}
&
\begin{tikzpicture}[baseline=0]
\node at (0, -0.15) {$ q_i^{\la_{i,\tau}+1}\vs_{\tau i} B_{\tau i} \one_{\ov{\bs_i\lambda}}$
};
\end{tikzpicture}
&

\begin{tikzpicture}[baseline=0]
\node at (0, -0.15) {
$ q_i^{\la_{i,\tau}-2}\vs_{\tau i}  B_i \one_{\ov{\bs_i\lambda}}$ 
};
\end{tikzpicture}
&
\begin{tikzpicture}[baseline=0]
\node at (0, -0.15) {$ q_i^{-\la_{i,\tau}-2}\vs_{i} B_{\tau i} \one_{\ov{\bs_i\lambda}}$
};
\end{tikzpicture}
\\
\hline
\end{tabular}
}
\end{table}


\begin{table}[H]
\caption{Rank two formulas for $\TTd'_{i,-1}(B_j\one_{\ov{\lambda}})$,  for  $i \in \Ifin$ and $j \neq i, \tau i$
\\
${}$ \qquad\qquad \qquad (Set $\alpha=-c_{ij}$, $\beta= -c_{\tau i,j}$)
}
     \label{table:rktwo-int}
\resizebox{5.5 in}{!}{%
\begin{tabular}{| c | c|}
\hline
\begin{tikzpicture}[baseline=0]
\node at (0, -0.15) {
$c_{i,\tau i}$};
\end{tikzpicture}
&
\begin{tikzpicture}[baseline=0]
\node at (0, -0.15) {
$\TTd'_{i,-1}( B_j\one_{\ov{\lambda}})$};
\end{tikzpicture}
\\
\hline
\begin{tabular}{c}
$ c_{i,\tau i}=2$
\end{tabular}
&
\begin{tikzpicture}[baseline=0]
\node at (0, -0.15) { 
$(-1)^{ \alpha} \sum\limits_{\substack{r+s=\alpha\\ \ov{r}=\ov{\lambda_i+1}}} (-q_i)^{r+\left\lfloor\frac{ 1-\alpha-\delta_{\ov{\la}_i,\ov{1}}}{2}\right\rfloor} 
\dv{i,\ov{\lambda_i+\alpha}}{s} B_j \dv{i,\ov{\lambda_i}}{r}\one_{\ov{\bs_i\lambda}}$};
 \node at (0, -1.5) {
$+(-1)^{ \alpha}\sum\limits_{u\geq 0}\sum\limits_{\substack{r+s+2u=\alpha
\\ \ov{r}=\ov{\lambda_i}}} (-q_i)^{r+u+\left\lfloor\frac{ 1-\alpha-\delta_{\ov{\la}_i,\ov{1}}}{2}\right\rfloor} 
\dv{i,\ov{\lambda_i+\alpha}}{s} B_j \dv{i,\ov{\lambda_i}}{r}\one_{\ov{\bs_i\lambda}}$};
\end{tikzpicture}
\\
\hline
$c_{i,\tau i}= 0$
&
\begin{tabular}{c}
$\sum\limits_{u=0}^{ \min(\alpha,\beta)} \sum\limits_{r=0}^{\alpha-u}
\sum\limits_{s=0}^{\beta-u} (-1)^{r+s+u+\beta}q_i^{r(-u+1)+s(u+1)+u-\beta-u\la_{i,\tau}}$
\\ 
$\qquad\qquad\qquad\qquad\qquad\qquad
\times B_i^{(\alpha-r-u)}B_{\tau i}^{(\beta-s-u)} B_j B_{\tau i}^{(s)}B_i^{(r)} \one_{\ov{\bs_i\lambda}}$
 \end{tabular}
\\
\hline
\begin{tikzpicture}[baseline=0]
\node at (0, -1.15) {
$c_{i,\tau i}= -1$};
\end{tikzpicture}
&
\begin{tikzpicture}[baseline=0]
\node at (0, -0.15) {
$\sum\limits_{u,w\geq 0} \sum\limits_{t=0}^{\beta-w} \sum\limits_{s=0}^{\beta+\alpha-w-u} \sum\limits_{r=0}^{\alpha-u} (-1)^{t+r+s+u+w+\alpha+\beta}
 q_i^{t(3u-2w+1) + r(u+1)}$};
 \node at (0, -1.15) {
$\qquad \times q_i^{s(w-2u+1)+uw+u+w-\alpha-\beta-\frac{u(u-1)+w(w+1)}{2}+(w-u)\la_{i,\tau}}$ };
 \node at (0, -2.15) {
 $\qquad\qquad\qquad\times\vs_{\tau i}^{w-u} B_i^{(\beta-w-t)}B_{\tau i}^{(\alpha+\beta-w-u-s)}B_i^{(\alpha-u-r)} B_j B_i^{(r)} B_{\tau i}^{(s)}  B_i^{(t)} \one_{\ov{\bs_i\lambda}}$};
\end{tikzpicture}
\\
\hline
\end{tabular}
}
\end{table}


\begin{table}[H]
\caption{Rank two formulas for $\TTd''_{i,+1}(B_j\one_{\ov{\lambda}})$, for $i\in \Ifin$ and $j \neq i,\tau i$
\qquad\qquad\qquad ${}$ }
     \label{table:rktwo-int2}
\resizebox{5.5 in}{!}{%
\begin{tabular}{| c | c|}
\hline
\begin{tikzpicture}[baseline=0]
\node at (0, -0.15) {
$c_{i,\tau i}$};
\end{tikzpicture}
&
\begin{tikzpicture}[baseline=0]
\node at (0, -0.15) {
$\TTd''_{i,+1}(B_j\one_{\ov{\lambda}})$};
\end{tikzpicture}
\\
\hline
\begin{tabular}{c}
$c_{i,\tau i}=2$
\end{tabular}
&
\begin{tikzpicture}[baseline=0]
\node at (0, -0.15) { 
$(-1)^{ \alpha}\sum\limits_{\substack{r+s= \alpha\\ \ov{s}=\ov{\lambda_i+1}}} (-q_i)^{-s-\left\lfloor\frac{1-\alpha-\delta_{\ov{\la}_i,\ov{1}}}{2}\right\rfloor} 
\dv{i,\ov{\lambda_i+ \alpha}}{r} B_j \dv{i,\ov{\lambda_i}}{s} \one_{\ov{\bs_i\lambda}}$}; 
\node at (0, -1.5) {
$+(-1)^{ \alpha}\sum\limits_{u\geq 0}\sum\limits_{\substack{r+s+2u= \alpha
\\ \ov{s}=\ov{\lambda_i}}} (-q_i)^{-s-u-\left\lfloor\frac{1-\alpha-\delta_{\ov{\la}_i,\ov{1}}}{2}\right\rfloor} 
\dv{i,\ov{\lambda_i+\alpha}}{r} B_j \dv{i,\ov{\lambda_i}}{s} \one_{\ov{\bs_i\lambda}}$}; 
\end{tikzpicture}
\\
\hline
$c_{i,\tau i}=0$
&
\begin{tabular}{c}
$\sum\limits_{u=0}^{ \min(\alpha,\beta)} \sum\limits_{r=0}^{\alpha-u}\sum\limits_{s=0}^{\beta-u} (-1)^{r+s+u+\beta}q_i^{(r-\alpha)(-u+1)+(s-\beta)(u+1)+u+\alpha+u\la_{i,\tau}}$
\\ 
$\qquad \qquad\qquad\qquad\qquad\times B_i^{(r)}B_{\tau i}^{(s)} B_j B_{\tau i}^{(\beta-s-u)}  B_i^{(\alpha-r-u)}\one_{\ov{\bs_i\lambda}}$
 \end{tabular}
\\
\hline
\begin{tikzpicture}[baseline=0]
 \node at (0, -1.15) {$c_{i,\tau i}=-1$};
\end{tikzpicture}
&
\begin{tikzpicture}[baseline=0]
\node at (0, -0.15) {
$\sum\limits_{u,w\geq 0} \sum\limits_{t=0}^{\beta-w} \sum\limits_{s=0}^{\beta+\alpha-w-u} \sum\limits_{r=0}^{\alpha-u} (-1)^{t+r+s+u+w+\alpha+\beta}
 q_i^{t(3u-2w+1) + r(u+1)+s(w-2u+1)}$};
 \node at (0, -1.15) {
$\times q_i^{uw+u+w-\alpha-\beta-\frac{u(u+1)+w(w-1)}{2}+(u-w)(\la_{i,\tau}+\alpha-\beta)} $};
 \node at (0, -2.15) {
 $\qquad\qquad\times\vs_{\tau i}^{u-w} B_i^{(t)}B_{\tau i}^{(s)}B_i^{(r)} B_j B_i^{(\alpha-u-r)} B_{\tau i}^{(\alpha+\beta-w-u-s)}B_i^{(\beta-w-t)}\one_{\ov{\bs_i\lambda}}$};
\end{tikzpicture}
\\
\hline
\end{tabular}
}
\end{table}
\vspace{1cm}

\appendix

\section{Proof of Proposition~\ref{prop:BBij2}}
\label{sec:proof}

In this appendix, we prove Proposition~\ref{prop:BBij2} by induction on $k$, with the base case for $k=0$ being trivial. We shall refer to \eqref{Com:qrootB2} and \eqref{Com:qrootB3} as \eqref{Com:qrootB2}$_k$ and \eqref{Com:qrootB3}$_{k}$, respectively. The inductive steps are carried out below in two steps: $\boxed{\eqref{Com:qrootB3}_k \Rightarrow\eqref{Com:qrootB2}_{k+1}}$ and $\boxed{\eqref{Com:qrootB2}_k + \eqref{Com:qrootB3}_{k-1}\Rightarrow\eqref{Com:qrootB3}_{k+1}}$.
We fix $\alpha=-c_{ij}$. 

\subsection{The implication ${\eqref{Com:qrootB3}_k \Rightarrow\eqref{Com:qrootB2}_{k+1}}$ }

 \begin{lemma}
\label{lem:BB1}
We have for any $m\geq 0$,
\begin{align*}
    b_{i,j;n,m} B_i = q_i^{2m-n \alpha} B_i b_{i,j;n,m} -[m+1]_i q_i^{2m-n \alpha} b_{i,j;n,m+1}
    +[n \alpha-m+1]_i b_{i,j;n,m-1}.
\end{align*}
\end{lemma}

\begin{proof}
    Follows by the recursive definition of $b_{i,j;n,m}$ from \eqref{def:splitBij}. 
\end{proof}

\begin{lemma}
\label{lem:qbinom1}
We have
$[2b]_i\qbinom{a}{b}_{q_i^2}=[2a]_i\qbinom{a-1}{b-1}_{q_i^2},$
for any $a,b\in \N$.
\end{lemma}

\begin{proof}
    Follows by a direct computation. 
\end{proof}

It follows by \eqref{recursion} that  
\begin{align}
\label{eq:BBij1}
b_{i,j;n,n\alpha}\dvi{ \ov{k+1} }{k} B_i = [k+1]_i b_{i,j;n,n\alpha}\dvi{ \ov{k+1} }{k+1}.
\end{align}
Using Lemma~\ref{lem:BB1} and $\eqref{Com:qrootB3}_k$, we compute the LHS of \eqref{eq:BBij1} as follows: 
\begin{align}\notag
&\quad b_{i,j;n,n\alpha}\dvi{ \ov{k+1} }{k} B_i
\\\notag
=& \; \sum_{x=0}^{n\alpha} q_i^{(k-x)(n\alpha-x) } \left(\sum_{y=0}^{\lfloor \frac{n\alpha-x}{2} \rfloor} (-1)^y q_i^{2y(\lfloor \frac{n\alpha-x}{2} \rfloor-n\alpha+x)}\qbinom{\lfloor \frac{n\alpha-x}{2} \rfloor}{y}_{q_i^2} \dvi{\ov{k+1+n \alpha}}{k-x-2y}\right) \times
\\\notag
&\times \big( q_i^{n\alpha-2x} B_i b_{i,j;n,n\alpha-x}-q_i^{n\alpha-2x}[n\alpha-x+1]_i b_{i,j;n,n\alpha-x+1}+[x+1]_i b_{i,j;n,n\alpha-x-1}\big)
\\
=& \; \sum_{x=0}^{n\alpha} c_x b_{i,j;n,n\alpha-x},
\label{eq:BBij2}
\end{align}
where the coefficient $c_x$ is given by
\begin{align*}
&c_x=q_i^{n \alpha-2x} q_i^{(k-x)(n\alpha-x) } \sum_{y=0}^{\lfloor \frac{n\alpha-x}{2} \rfloor} (-1)^y q_i^{2y(\lfloor \frac{n\alpha-x}{2}  \rfloor-n\alpha+x)}\qbinom{\lfloor \frac{n\alpha-x}{2} \rfloor}{y}_{q_i^2} \dvi{\ov{k+1+n \alpha}}{k-x-2y}B_i 
\\
& -q_i^{-k-1}[n\alpha-x ]_i q_i^{(k-x)(n\alpha-x) }\sum_{y=0}^{\lfloor \frac{n\alpha-x-1}{2} \rfloor} (-1)^y q_i^{2y\big(\lfloor \frac{n\alpha-x-1}{2}  \rfloor-n\alpha+x+1\big)}\qbinom{\lfloor \frac{n\alpha-x-1}{2} \rfloor}{y}_{q_i^2} \dvi{\ov{k+1+n \alpha}}{k-x-2y-1}
\\
&+[x]_i q_i^{(k-x+1)(n \alpha-x+1) }\sum_{y=0}^{\lfloor \frac{n\alpha-x+1}{2} \rfloor} (-1)^y q_i^{2y(\lfloor \frac{n\alpha-x+1}{2} \rfloor-n\alpha+x-1)}\qbinom{\lfloor \frac{n\alpha-x+1}{2} \rfloor}{y}_{q_i^2} \dvi{\ov{k+1+n \alpha}}{k-x-2y+1}.
\end{align*}
We simplify this formula by separating it into two cases (1)--(2) below. 

(1) $\boxed{n\alpha-x \text{ is even}}$. Then $k-x-2y$ and $k+1+n\alpha$ have different parities, and thus by \eqref{recursion} we have $\dvi{\ov{k+1+n \alpha}}{k-x-2y}B_i=[k-x-2y+1]_i\dvi{\ov{k+1+n \alpha}}{k-x-2y+1}$. By Lemma~\ref{lem:qbinom1}, we have
\begin{align*}
[n\alpha-x]_i\qbinom{\lfloor \frac{n\alpha-x-1}{2} \rfloor}{y-1}_{q_i^2}=[n\alpha-x ]_i\qbinom{\frac{n\alpha-x }{2} -1}{y-1}_{q_i^2}=[2y]_i\qbinom{ \frac{n\alpha-x }{2} }{y }_{q_i^2}.
\end{align*}
Hence, we have
\begin{align*}
c_x&=q_i^{n\alpha-2x} q_i^{(k-x)(n\alpha-x) } \sum_{y=0}^{  \frac{n\alpha-x}{2} } (-1)^y q_i^{y(-n\alpha+x)}[k-x-2y+1]_i \qbinom{ \frac{n\alpha-x}{2} }{y}_{q_i^2} \dvi{\ov{k+1+n \alpha}}{k-x-2y+1} 
\\
&\quad +q_i^{-k-1}  q_i^{(k-x)(n\alpha-x) } \sum_{y=1}^{ \frac{n\alpha-x}{2} } (-1)^y q_i^{(y-1)(-n\alpha+x)}[2y]_i\qbinom{ \frac{n\alpha-x}{2}}{y}_{q_i^2} \dvi{\ov{k+1+n \alpha}}{k-x-2y+1}
\\
&\quad+[x ]_i q_i^{(k-x+1)(n\alpha-x+1) }\sum_{y=0}^{ \frac{n\alpha-x }{2} } (-1)^y q_i^{2y( \frac{n\alpha-x }{2} -1-n\alpha+x)}\qbinom{  \frac{n\alpha-x }{2} }{y}_{q_i^2} \dvi{\ov{k+1+n \alpha}}{k-x-2y+1}
\\
&= q_i^{(k+1-x)(n\alpha-x) }[k+1]_i \sum_{y=0}^{  \frac{n\alpha-x}{2} } (-1)^y q_i^{y(-n\alpha+x-2)} \qbinom{ \frac{n\alpha-x}{2} }{y}_{q_i^2} \dvi{\ov{k+1+n \alpha}}{k-x-2y+1}.
\end{align*}
This formula, together with \eqref{eq:BBij1}--\eqref{eq:BBij2}, verifies the coefficient of $b_{i,j;n,n\alpha-x}$ in $\eqref{Com:qrootB2}_{k+1}$ for any $n\alpha-x$ even.

(2) $\boxed{n\alpha-x \text{ is odd}}$. Then $k-x-2y$ and $k+1+n\alpha$ have the same parity, and thus by \eqref{recursion} we have 
\begin{align*}
\dvi{\ov{k+1+n \alpha}}{k-x-2y}B_i=[k-x-2y+1]_i \dvi{\ov{k+1+n \alpha}}{k-x-2y+1} - q_i^{-1} [k-x-2y]_i \dvi{\ov{k+1+n \alpha}}{k-x-2y-1}.
\end{align*}
Hence, we have
\begin{align*}
c_x &=q_i^{(k+1-x)(n \alpha-x)-x } \sum_{y=0}^{  \frac{n\alpha-x-1}{2} } (-1)^y q_i^{y(x-n\alpha-1)}[k-x-2y+1]_i\qbinom{ \frac{n\alpha-x-1}{2} }{y}_{q_i^2} \dvi{\ov{k+1+n \alpha}}{k-x-2y+1} 
\\
&+q_i^{(k+1-x)(n\alpha-x) -x-1 } \sum_{y=1}^{  \frac{n\alpha-x+1}{2} } (-1)^y q_i^{(y-1)(x-n\alpha-1)}[k-x-2y+2]_i\times
\\
&\qquad\qquad \qquad\qquad\qquad\qquad\qquad \times \qbinom{ \frac{n\alpha-x-1}{2} }{y-1}_{q_i^2} \dvi{\ov{k+1+n \alpha}}{k-x-2y+1} 
\\
&+q_i^{-k-1}  q_i^{(k-x)(n\alpha-x) } [n\alpha-x]_i\sum_{y=1}^{ \frac{n\alpha-x+1}{2} } (-1)^y q_i^{(y-1)(-n\alpha+1+x)} \qbinom{ \frac{n\alpha-x-1}{2}}{y-1}_{q_i^2} \dvi{\ov{k+1+n \alpha}}{k-x-2y+1}
\\
&+[x ]_i q_i^{(k-x+1)(n\alpha-x+1) } \sum_{y=0}^{ \frac{n\alpha-x+1}{2} } (-1)^y q_i^{ y(-n\alpha-1+x)}\qbinom{  \frac{n\alpha-x+1}{2} }{y}_{q_i^2} \dvi{\ov{k+1+n \alpha}}{k-x-2y+1}.
\end{align*}
Applying the $q$-binomial identity 
$\tiny \qbinom{\frac{n\alpha-x-1}{2} }{y}_{q_i^2} =q_i^{2y }\qbinom{\frac{n\alpha-x+1}{2} }{y}_{q_i^2} -q_i^{n\alpha-x+1} \qbinom{ \frac{n\alpha-x-1}{2} }{y-1}_{q_i^2}$
to the first line of the above formula, we obtain
\begin{align*}
&c_x= 
 q_i^{n\alpha-x-k} q_i^{(k-x)(n\alpha-x) } \sum_{y=1}^{  \frac{n\alpha-x+1}{2} } (-1)^y q_i^{(y-1)(-n\alpha+1+x)} \qbinom{ \frac{n\alpha-x-1}{2} }{y-1}_{q_i^2} \dvi{\ov{k+1+n \alpha}}{k-x-2y+1} 
\\
& \quad + q_i^{(k-x)(n\alpha-x)-k-1 } [n\alpha-x]_i\sum_{y=1}^{ \frac{n\alpha-x+1}{2} } (-1)^y q_i^{(y-1)(-n\alpha+1+x)} \qbinom{ \frac{n\alpha-x-1}{2}}{y-1}_{q_i^2} \dvi{\ov{k+1+n \alpha}}{k-x-2y+1}
\\
&\quad + q_i^{(k-x+1)(n\alpha-x) } \sum_{y=0}^{ \frac{n\alpha-x+1}{2} } (-1)^y q_i^{ y(-n\alpha+1+x)}[k-2y+1]_i\qbinom{\frac{n\alpha-x+1}{2} }{y}_{q_i^2} \dvi{\ov{k+1+n \alpha}}{k-x-2y+1}
\\
&= q_i^{(k-x)(n\alpha-x) -k} [n\alpha-x+1]_i\sum_{y=1}^{ \frac{n\alpha-x+1}{2} } (-1)^y q_i^{(y-1)(-n\alpha+1+x)} \qbinom{ \frac{n\alpha-x-1}{2}}{y-1}_{q_i^2} \dvi{\ov{k+1+n \alpha}}{k-x-2y+1}
\\
&\quad + q_i^{(k-x+1)(n\alpha-x) } \sum_{y=0}^{ \frac{n\alpha-x+1}{2} } (-1)^y q_i^{ y(-n\alpha+1+x)}[k-2y+1]_i\qbinom{\frac{n\alpha-x+1}{2} }{y}_{q_i^2} \dvi{\ov{k+1+n \alpha}}{k-x-2y+1}
\\
& \overset{\text{Lem}\; \ref{lem:qbinom1}}{=} q_i^{(k-x+1)(n\alpha-x)-k-1}  \bigg(\sum_{y=0}^{ \frac{n\alpha-x+1}{2} } (-1)^y q_i^{y (-n\alpha+1+x)} [2y]_i \qbinom{ \frac{n\alpha-x+1}{2}}{y }_{q_i^2} \dvi{\ov{k+1+n \alpha}}{k-x-2y+1}\bigg)
\\
&\quad + q_i^{(k-x+1)(n\alpha-x) } \bigg(\sum_{y=0}^{ \frac{n\alpha-x+1}{2} } (-1)^y q_i^{ y(-n\alpha+1+x)}[k-2y+1]_i\qbinom{\frac{n\alpha-x+1}{2} }{y}_{q_i^2} \dvi{\ov{k+1+n\alpha}}{k-x-2y+1}\bigg)
\\
&= q_i^{(k-x+1)(n\alpha-x) } [k +1]_i\bigg(\sum_{y=0}^{ \frac{n\alpha-x+1}{2} } (-1)^y q_i^{ y(-n\alpha-1+x)}\qbinom{\frac{n\alpha-x+1}{2} }{y}_{q_i^2} \dvi{\ov{k+1+n \alpha}}{k-x-2y+1}\bigg).
\end{align*}
This formula, together with \eqref{eq:BBij1}--\eqref{eq:BBij2}, verifies the coefficient of $b_{i,j;n,n\alpha-x}$ in the formula $\eqref{Com:qrootB2}_{k+1}$ for any $n\alpha-x$ odd.

Therefore, combining the above two cases, we have proved $\eqref{Com:qrootB3}_k \Rightarrow\eqref{Com:qrootB2}_{k+1}$.

\subsection{The implication ${\eqref{Com:qrootB2}_k + \eqref{Com:qrootB3}_{k-1}\Rightarrow\eqref{Com:qrootB3}_{k+1}}$}


It follows by \eqref{recursion} that  
\begin{align}
\label{eq:BBij3}
 [k+1]_i b_{i,j;n,n\alpha}\dvi{ \ov{k} }{k+1}=b_{i,j;n,n\alpha}\dvi{ \ov{k} }{k} B_i + q_i^{-1} [k]_i b_{i,j;n,n\alpha}\dvi{ \ov{k} }{k-1}.
\end{align}
Using Lemma~\ref{lem:BB1}, $\eqref{Com:qrootB2}_k$ and $\eqref{Com:qrootB3}_{k-1}$, we compute the RHS of \eqref{eq:BBij3} as follows:
\begin{align}\notag
&\quad b_{i,j;n,n\alpha}\dvi{ \ov{k} }{k} B_i+ q_i^{-1} [k]_i b_{i,j;n,n\alpha}\dvi{ \ov{k} }{k-1}
\\\notag
&= \sum_{x=0}^{n\alpha} q_i^{(k-x)(n\alpha-x) } \bigg(\sum_{y=0}^{\lceil \frac{n\alpha-x}{2} \rceil} (-1)^y q_i^{2y(\lceil \frac{n\alpha-x}{2} \rceil-1-n\alpha+x)}\qbinom{\lceil \frac{n\alpha-x}{2} \rceil}{y}_{q_i^2} \dvi{\ov{k+n \alpha}}{k-x-2y}\bigg) \times
\\\notag
&\qquad \times \big( q_i^{n\alpha-2x} B_i b_{i,j;n,n\alpha-x}-q_i^{n\alpha-2x}[n\alpha-x+1]_i b_{i,j;n,n\alpha-x+1}+[x+1]_i b_{i,j;n,n\alpha-x-1}\big)
\\\notag
&+[k]_i \sum_{x=0}^{n\alpha} q_i^{(k-1-x)(n\alpha-x)-1} \bigg(\sum_{y=0}^{\lfloor \frac{n\alpha-x}{2} \rfloor} (-1)^y q_i^{2y(\lfloor \frac{n\alpha-x}{2} \rfloor -n\alpha+x)}\qbinom{\lfloor \frac{n\alpha-x}{2} \rfloor}{y}_{q_i^2} \dvi{\ov{k+n \alpha}}{k-1-x-2y}\bigg)\times
\\\notag
& \qquad  \times b_{i,j;n,n\alpha-x}
\\
&= \sum_{x=0}^{n\alpha} d_x b_{i,j;n,n\alpha-x},
\label{eq:BBij4}
\end{align}
where the coefficient $d_x$ of $b_{i,j;n,n\alpha-x}$ is given by
\begin{align*}
&d_x=q_i^{n\alpha-2x} q_i^{(k-x)(n\alpha-x) } \bigg(\sum_{y=0}^{\lceil \frac{n\alpha-x}{2} \rceil} (-1)^y q_i^{2y(\lceil \frac{n\alpha-x}{2}  \rceil-1-n\alpha+x)}\qbinom{\lceil \frac{n\alpha-x}{2} \rceil}{y}_{q_i^2} \dvi{\ov{k+n \alpha}}{k-x-2y}B_i\bigg)
\\
& -[n\alpha-x]_i q_i^{(k-x)(n\alpha-x)-k-1 } \times
\\
&\qquad\quad\times\bigg(\sum_{y=0}^{\lceil \frac{n\alpha-x-1}{2} \rceil} (-1)^y q_i^{2y(\lceil \frac{n\alpha-x-1}{2}  \rceil-n\alpha+x)}\qbinom{\lceil \frac{n\alpha-x-1}{2} \rceil}{y}_{q_i^2} \dvi{\ov{k+n \alpha}}{k-x-2y-1}\bigg)
\\
&+[x]_i q_i^{(k-x+1)(n\alpha-x+1) } \bigg(\sum_{y=0}^{\lceil \frac{n\alpha-x+1}{2} \rceil} (-1)^y q_i^{2y(\lceil \frac{n\alpha-x+1}{2} \rceil-n\alpha+x-2)}\qbinom{\lceil \frac{n\alpha-x+1}{2} \rceil}{y}_{q_i^2} \dvi{\ov{k+n \alpha}}{k-x-2y+1}\bigg)
\\
&+[k]_i q_i^{(k-1-x)(n\alpha-x)-1} \bigg(\sum_{y=0}^{\lfloor \frac{n\alpha-x}{2} \rfloor} (-1)^y q_i^{2y(\lfloor \frac{n\alpha-x}{2} \rfloor -n\alpha+x)}\qbinom{\lfloor \frac{n\alpha-x}{2} \rfloor}{y}_{q_i^2} \dvi{\ov{k+n \alpha}}{k-1-x-2y}\bigg).
\end{align*}
We simplify this formula for $d_x$ by separating it into two cases (1)--(2) below.

(1) $\boxed{n\alpha-x \text{ is odd}}$. In this case, $k-x-2y$ and $k+n\alpha$ has different parities, and thus by \eqref{recursion} we have $\dvi{\ov{k+n \alpha}}{k-x-2y}B_i=[k-x-2y+1]_i\dvi{\ov{k+n \alpha}}{k-x-2y+1}$.

We simplify $d_x$ as
\begin{align*}
d_x&=  q_i^{(k-x+1)(n\alpha-x)-x } \sum_{y=0}^{ \frac{n\alpha-x+1}{2} } (-1)^y q_i^{ y( x-1-n\alpha)}[k-x-2y+1]_i\qbinom{ \frac{n\alpha-x+1}{2} }{y}_{q_i^2} \dvi{\ov{k+n \alpha}}{k-x-2y+1}
\\
&+q_i^{-k-1}[n\alpha-x ]_i q_i^{(k-x)(n\alpha-x) } \sum_{y=1}^{ \frac{n\alpha-x+1}{2} } (-1)^y q_i^{(y-1)( -1-n\alpha+x)}\qbinom{ \frac{n\alpha-x-1}{2} }{y-1}_{q_i^2} \dvi{\ov{k+n \alpha}}{k-x-2y+1}
\\
&+[x ]_i q_i^{(k-x+1)(n\alpha-x+1) }  \sum_{y=0}^{ \frac{n\alpha-x+1}{2} } (-1)^y q_i^{y( -n\alpha+x-3)}\qbinom{ \frac{n\alpha-x+1}{2} }{y}_{q_i^2} \dvi{\ov{k+n \alpha}}{k-x-2y+1}
\\
&-[k]_i q_i^{(k-1-x)(n\alpha-x)-1} \sum_{y=1}^{ \frac{n\alpha-x+1}{2} } (-1)^y q_i^{(y-1)(-1 -n\alpha+x)}\qbinom{  \frac{n\alpha-x-1}{2}}{y-1}_{q_i^2} \dvi{\ov{k+n \alpha}}{k-x-2y+1}
\\
&= q_i^{(k-x+1)(n\alpha-x) } \sum_{y=0}^{ \frac{n\alpha-x+1}{2} } (-1)^y q_i^{ y( -1-n\alpha+x)}[k-2y+1]_i\qbinom{ \frac{n\alpha-x+1}{2} }{y}_{q_i^2} \dvi{\ov{k+n \alpha}}{k-x-2y+1}
\\
&+q_i^{-k-1}[n\alpha-x ]_i q_i^{(k-x+1)(n\alpha-x)+1} \sum_{y=1}^{ \frac{n\alpha-x+1}{2} } (-1)^y q_i^{y( -1-n\alpha+x)}\qbinom{ \frac{n\alpha-x-1}{2} }{y-1}_{q_i^2} \dvi{\ov{k+n \alpha}}{k-x-2y+1}
\\
&-[k]_i q_i^{(k-1-x)(n\alpha-x)-1} \sum_{y=1}^{ \frac{n\alpha-x+1}{2} } (-1)^y q_i^{(y-1)(-1 -n\alpha+x)}\qbinom{\frac{\alpha-x-1}{2}}{y-1}_{q_i^2} \dvi{\ov{k+n \alpha}}{k-x-2y+1}.
\end{align*}
Rewriting the first line via $[k-2y+1]_i=q_i^{-2y}[k+1]_i-q_i^{-k-1}[2y]_i$ and applying Lemma~\ref{lem:qbinom1}, we obtain
\begin{align*}
d_x&=  q_i^{(k-x+1)(n\alpha-x) }[k+1]_i \sum_{y=0}^{ \frac{n\alpha-x+1}{2} } (-1)^y q_i^{ y( -3-n\alpha+x)}\qbinom{ \frac{n\alpha-x+1}{2} }{y}_{q_i^2} \dvi{\ov{k+n \alpha}}{k-x-2y+1}
\\
&\;- q_i^{(k-1-x)(n\alpha-x)+x-n\alpha-1}[k+1]_i \sum_{y=1}^{ \frac{n\alpha-x+1}{2} } (-1)^y q_i^{y(-1 -n\alpha+x)}\qbinom{\frac{n\alpha-x-1}{2}}{y-1}_{q_i^2} \dvi{\ov{k+n \alpha}}{k-x-2y+1}.
\end{align*} 
Finally, applying the following $q$-binomial identity to this formula of $d_x$
\begin{align*}
\qbinom{ \frac{n\alpha-x+1}{2} }{y}_{q_i^2}=q_i^{2y}\qbinom{ \frac{n\alpha-x-1}{2} }{y}_{q_i^2}+q_i^{x-n\alpha+2y-1}\qbinom{ \frac{n\alpha-x-1}{2} }{y-1}_{q_i^2},
\end{align*}
we see that $d_x/[k+1]_i$ equals the coefficient of $b_{i,j;n,n\alpha-x}$ in $\eqref{Com:qrootB3}_{k+1}$ for any $n\alpha-x$ odd.

(2) $\boxed{n\alpha-x \text{ is even}}$. In this case, $k-x-2y$ and $k+n\alpha$ has the same parity, and thus by \eqref{recursion} we have 
\begin{align*}
\dvi{\ov{k+n \alpha}}{k-x-2y}B_i =[k-x-2y+1]_i\dvi{\ov{k+n \alpha}}{k-x-2y+1} - q_i^{-1}[k-x-2y]_i\dvi{\ov{k+n \alpha}}{k-x-2y-1}.
\end{align*}

We simplify $d_x$ as
\begin{align*}
&d_x=  q_i^{(k-x+1)(n\alpha-x)-x} \sum_{y=0}^{ \frac{n\alpha-x}{2} } (-1)^y q_i^{ y( -2-n\alpha+x)}[k-x-2y+1]_i\qbinom{\frac{n\alpha-x}{2}}{y}_{q_i^2} \dvi{\ov{k+n \alpha}}{k-x-2y+1}
\\
&\quad+q_i^{(k-x+2)(n\alpha-x)-x+1 } \sum_{y=1}^{ \frac{n\alpha-x}{2}+1 } (-1)^y q_i^{y(x-2-n\alpha)}[k-x-2y+2]_i\qbinom{\frac{n\alpha-x}{2}}{y-1}_{q_i^2} \dvi{\ov{k+n \alpha}}{k-x-2y+1}
\\
& \quad+q_i^{-k-1}[n\alpha-x]_i q_i^{(k-x)(n\alpha-x) }  \sum_{y=1}^{\frac{n\alpha-x}{2}+1} (-1)^y q_i^{ (y-1)( -n\alpha+x)}\qbinom{ \frac{n\alpha-x}{2}}{y-1}_{q_i^2} \dvi{\ov{k+n \alpha}}{k-x-2y+1}
\\
&\quad+[x]_i q_i^{(k-x+1)(n\alpha-x+1) } \sum_{y=0}^{ \frac{n\alpha-x}{2}+1} (-1)^y q_i^{ y(-2-n\alpha+x)}\qbinom{\frac{n\alpha-x}{2}+1}{y}_{q_i^2} \dvi{\ov{k+n \alpha}}{k-x-2y+1}
\\
&\quad-[k]_i q_i^{(k-1-x)(n\alpha-x)-1} \sum_{y=1}^{ \frac{n\alpha-x}{2}+1} (-1)^y q_i^{ (y-1)( -n\alpha+x)}\qbinom{\frac{n\alpha-x}{2}}{y-1}_{q_i^2} \dvi{\ov{k+n \alpha}}{k -x-2y+1} .
\end{align*}
Applying the following $q$-binomial identity 
\begin{align*}
\qbinom{\frac{n\alpha-x}{2}+1}{y}_{q_i^2}=q_i^{2y} \qbinom{\frac{n\alpha-x}{2} }{y}_{q_i^2}+ q_i^{x-n\alpha+2y-2} \qbinom{\frac{n\alpha-x}{2} }{y-1}_{q_i^2},
\end{align*}
to the above formula of $d_x$, we obtain
\begin{align*}
d_x& =  q_i^{(k+1-x)(n\alpha-x) } \sum_{y=0}^{ \frac{n\alpha-x}{2} } (-1)^y q_i^{ y(-n\alpha+x)} d'_{x,y} \qbinom{\frac{n\alpha-x}{2}}{y}_{q_i^2} \dvi{\ov{k+n \alpha}}{k-x-2y+1}
\end{align*}
where 
\begin{align*}
d'_{x,y}&=q_i^{-x-2y}[k-x-2y+1]_i+q_i^{k-x+1}[x]_i +
\\
&\quad + \frac{[2y]_i}{[n\alpha-x-2y+2]_i}\big( q_i^{n\alpha-2x-2y+1}[k-x-2y+2]_i+q_i^{ -k-1}[n\alpha-x ]_i
\\
&\qquad \qquad\qquad \qquad\qquad +q_i^{k-n\alpha-1} [x]_i -q_i^{x-n\alpha-1}[k]_i\big)
\\
& =[k+1]_i.
\end{align*}
This formula shows that $d_x/[k+1]_i$ equals the coefficient of $b_{i,j;n,n\alpha-x}$ in $\eqref{Com:qrootB2}_{k+1}$ for any $n\alpha-x$ even as desired.

Combining the above two cases, we have shown $\eqref{Com:qrootB2}_k + \eqref{Com:qrootB3}_{k-1}\Rightarrow\eqref{Com:qrootB3}_{k+1}$.

\vspace{4mm}
{\bf Acknowledgment}
WW is partially supported by DMS--2401351. WZ is partially supported by the New Cornerstone Science Foundation through the New Cornerstone Investigator Program awarded to Xuhua He.

\vspace{4mm}

{\bf Conflict of interest}
The authors have no conflicts of interest to declare that are relevant to the content of this article.

\vspace{4mm}

{\bf Data Availability}
The authors declare that the data supporting the findings of this study are available within the paper.


\end{document}